\documentclass[a4paper]{article}
\usepackage{amsmath,amssymb,amsthm,amsfonts}
\usepackage{enumerate,color, graphicx}
\usepackage{supertabular}
\usepackage{booktabs}
\usepackage{url}
\usepackage{xcolor}
\usepackage{colortbl}
\usepackage{subfig}
\newcommand{\R}{\ensuremath{\mathbb{R}}}

\newcommand{\N}{\ensuremath{\mathbb{N}}}

\definecolor{dgreen}{rgb}{0.,0.6,0.}

\newcommand{\cotan}{\mathrm{cotan}}

\numberwithin{equation}{section}

\newtheorem{thm}{Theorem}[section]
\newtheorem{prop}[thm]{Proposition}

\newtheorem{lem}[thm]{Lemma}

\newtheorem{rem}[thm]{Remark}

\newtheorem{conjecture}[thm]{Conjecture}

\title{Courant-sharp Robin eigenvalues for the square and other planar domains}

\author{K. Gittins
\footnote{ Universit\'e de Neuch\^atel, Institut de Math\'ematiques, Rue Emile-Argand 11, CH-2000 Neuch\^atel.
Email: \texttt{katie.gittins@unine.ch}}\,
and B. Helffer
\footnote{Laboratoire de Math\'ematiques Jean Leray, Universit\'e de Nantes, 2 rue de la Houssini\`ere, 44 322 Nantes CEDEX 3 - FRANCE.
Email: \texttt{Bernard.Helffer@univ-nantes.fr}}
}

\date{\today}

\begin{document}
	\maketitle
\begin{abstract}	

This paper is devoted to the determination of the cases where there is
equality in Courant's nodal domain theorem in the case of a Robin boundary condition.
For the square, we partially extend the results that were obtained by Pleijel, B\'erard--Helffer,
Helffer--Persson--Sundqvist for the Dirichlet and Neumann problems.

After proving some general results that hold for any value of the Robin parameter $h$, we focus on the case when $h$ is large. We hope to come back to the analysis when $h$ is small in a second paper.

We also obtain some semi-stability results for the number of nodal domains of a Robin eigenfunction
of a domain with $C^{2,\alpha}$ boundary ($\alpha >0$) as $h$ large varies.

\end{abstract}

\paragraph{MSC classification (2010):}    35P99, 58J50, 58J37.


\newpage
\tableofcontents
\newpage
\section{Introduction.}

Let $\Omega \subset \R^m$, $m \geq 2$, be a bounded, connected, open set with Lipschitz boundary and
let $h \in \R$, $h \geq 0$. The case when $h< 0$ is mathematically interesting but less motivated by Physics. The Robin eigenvalues of the Laplacian on $\Omega$ with parameter
$h$ are $\lambda_{k,h}(\Omega) \in \R$, $k \in \N$, $k \geq 1$, such that there exists a function $u_k \in
H^1(\Omega)$ which satisfies
\begin{align*}
-\Delta u_k(x) &= \lambda_{k,h}(\Omega)u_k(x)\,, \quad x \in \Omega\,, \notag \\
\frac{\partial}{\partial \nu} u_k(x) &+ h\, u_k(x) = 0\,,  \quad x \in \partial\Omega\,, 
\end{align*}
where $\nu$ is the outward-pointing unit normal to $\partial \Omega$.

We recall that by the minimax principle, the Robin problem is associated with the quadratic form:
$$
H^1(\Omega) \ni u \mapsto \int_\Omega |\nabla u|^2 + h \int_{\partial \Omega} |u_{\partial \Omega}|^2 d\sigma\, ,
$$
 where $u_{\partial \Omega}$ is the trace of $u$.
So the spectrum is monotonically increasing with respect to $h$ for $h\in [0,+\infty)$. That is, the Robin eigenvalues
with $h >0$ interpolate between the Neumann eigenvalues ($h=0$) and the Dirichlet eigenvalues ($h=+\infty$).

The Robin eigenvalues satisfy the celebrated Courant nodal domain theorem \cite{CH} stating that any eigenfunction corresponding to $\lambda_{k,h}(\Omega)$ has at most $k$ nodal domains. We consider the Courant-sharp Robin eigenvalues of $\Omega$.
We call a Robin eigenvalue $\lambda_{k,h}(\Omega)$ Courant-sharp if it has a corresponding eigenfunction
that has exactly $k$ nodal domains. As for  the Dirichlet and Neumann eigenvalues,
 $\lambda_{1,h}(\Omega)$ and $ \lambda_{2,h}(\Omega)$ are Courant-sharp for all $h \geq 0$.

An interesting question is whether it is possible to follow the Courant-sharp (Neumann) eigenvalues with $h=0$ to
Courant-sharp (Dirichlet) eigenvalues as $h \to + \infty$, or whether there are some critical  values
$h^*(k, \Omega)$ after which the Robin eigenvalues $\lambda_{k,h}(\Omega)$, $h \geq h^*(k,\Omega)$
 become Courant-sharp or are no longer Courant-sharp.

We note that throughout this paper, we denote the Dirichlet eigenvalues by $\lambda_k^D$
and the Neumann eigenvalues by $\lambda_k^N$.\\

We consider the particular example where $\Omega$ is a square $S$ in $\R^2$ of side-length $\pi$ and the main question is:\\
{\it Is it possible to determine the Courant-sharp Robin eigenvalues of this square?}

 As  $\lambda_{2,h}(S)=\lambda_{3,h}(S)$ by a symmetry argument, it  follows immediately 
that $\lambda_{3,h}(S)$ is not Courant-sharp for any $h \geq 0$.
 In addition,
$\lambda_{4,h}(S)$ is Courant-sharp for all $h \geq 0$, see Subsection~\ref{SS:2.3}.

It was asserted by Pleijel in \cite{Pl} that the only Courant-sharp Dirichlet eigenvalues of the square are for $k=1,2,4$. This was shown rigorously in \cite{BH1}. The only Courant-sharp Neumann eigenvalues of the square are for $k=1,2,4,5,9$,
as shown in \cite{HPS1}.

The first step to obtain the results of \cite{BH1,HPS1} is to reduce the number of potential Courant-sharp eigenvalues by invoking an argument which was inspired by the founding paper of Pleijel \cite{Pl}.
We employ a similar argument in Section~\ref{s3} to reduce the possible cases that may give rise to Courant-sharp Robin eigenvalues. We have the following theorem.

\begin{thm}\label{prop:p1} Let $h\geq 0$.
If $\lambda_{k,h}(S)$ is an eigenvalue of $S$ with $ k \geq 520$, then it is not Courant-sharp.
\end{thm}

We note that in the case of  a Dirichlet boundary condition, the equivalent statement in \cite{Pl} gives $k \geq 34$ and in the case of a Neumann boundary condition, \cite{HPS1}, $ k \geq 209$.
The strategies of \cite{BH1, HPS1} are then either to re-implement the Faber-Krahn inequality, or to use symmetry properties of the corresponding eigenfunctions to further eliminate potential Courant-sharp eigenvalues. One is then reduced to the analysis of the nodal structure of very few families of eigenfunctions that belong to two-dimensional  eigenspaces.

We will show that the Robin eigenfunctions satisfy analogous symmetry properties.  We were not able to eliminate potential Courant-sharp cases via symmetry as it is possible that a Robin eigenvalue has multiplicity larger than 2 and the corresponding eigenfunctions have no common symmetries (see Subsection~\ref{ssec:9CS}).

In addition, for a Robin eigenvalue $\lambda_{k,h}(S)$, we do not know how to take the relationship between $k$ and $h$ into account in an efficient way. Indeed, to prove Theorem~\ref{prop:p1} our arguments are independent of $h$ as they rely on the monotonicity of the Robin eigenvalues and comparison to the corresponding Dirichlet and Neumann eigenvalues.\\

We note that the recent articles \cite{FK18, kGcL} also consider the Robin eigenvalues of Euclidean domains
and make use of this monotonicity property.
In \cite{kGcL}, upper bounds are obtained for the Courant-sharp Neumann and Robin eigenvalues with $h >0$ of a bounded, connected, open set $\Omega \subset \R^n$ with $C^2$ boundary.
In \cite{FK18}, it is shown that the Robin eigenvalues with $h>0$ on rectangles and unions of rectangles with prescribed area satisfy P\'olya-type inequalities.\\

In addition, we treat the problem asymptotically as $h\rightarrow +\infty$.
Hence we show that for $h$ large enough the only Courant-sharp Robin eigenvalues are for $k=1,2,4$.
 \begin{thm}\label{thm:hlarge}
 There exists $h_1>0$ such that for $h\geq h_1$, the Courant-sharp cases for the Robin problem are the same
 as those for $h=+\infty\,$  (i.e. the Dirichlet case).
 \end{thm}

In order to prove this theorem, we follow the strategy due to Pleijel, \cite{Pl}.
It is  therefore necessary to estimate the number of nodal domains whose boundaries intersect the boundary of the square on at least a non-trivial interval. For such nodal domains, we cannot use the Faber-Krahn inequality
for the Dirichlet problem. Nevertheless, there is a Faber-Krahn inequality for the Robin problem when $h>0$ (see \cite{Bos, BG2, D}). We will see how this can be used for $h$ sufficiently large in  Subsection~\ref{ss3.3} and Section~\ref{S:hlarge}.\\

In Section \ref{sp5}, we analyse the number of nodal domains of Robin eigenfunctions in the general context of a planar domain with piecewise $C^{2, \alpha}$ boundary ($\alpha >0$). We obtain some semi-stability results for the number of
nodal domains as the Robin parameter ($h$ large) varies.

For the square, the results of Section~\ref{sp5} allow us to deal with the remaining case $k=5$ which is not covered by Pleijel's strategy. In Section~\ref{s:5CS}, we describe explicitly the situations where the eigenfunction corresponding to the fifth Robin eigenvalue has $2$, $3$, $4$ nodal domains respectively (for $h>0$ sufficiently large).

In a second paper, \cite{GHPS},  we hope to look at the situation where the Robin parameter $h$ tends to $0$ and  to discuss the following conjecture.

\begin{conjecture}\label{thm:hsmall}
 There exists $h_0>0$ such that for $0<h \leq h_0$, the Courant-sharp cases for the Robin problem are the same, except the fifth one, as those for $h=0$ (i.e. the Neumann case)\,.
\end{conjecture}

In light of the results of \cite{Pl, BH1, HPS1} and of the previous asymptotic results, a key question is to what extent is it possible to follow the Courant-sharp (Neumann) eigenvalues with $h=0$ to Courant-sharp Robin eigenvalues as $h \to + \infty$?\\
We prove a first general result concerning the possible crossings between curves corresponding to Robin eigenvalues.

We then focus on the cases $k=9, 25$.
 For the case $k=9$, we investigate if there exist critical values $ \overline{h}_9^*(S)$, respectively $\underline{h}_9^*(S)$, after which the Robin eigenvalue $ \lambda_{9,h}(S)$ is not Courant-sharp, respectively before which it is Courant-sharp, the next question being whether we have $\underline{h}_9^*=\overline{h}_9^*$. This question will be addressed in Subsection~\ref{ssec:9CS}. In Subsection \ref{ss5.3}, for $k=25$, we show that there exists $h_{25}^*$ such that $\lambda_{25,h}$ is not Courant-sharp for $h < h_{25}^*$, and we also investigate the structure of the nodal partitions of $S$ for $h$ sufficiently large.\\

\textbf{Acknowledgements:}\\
We are very grateful to Pierre B\'erard, Dorin Bucur and Thomas Hoffmann-Ostenhof for useful discussions, and to
Alexander Weisse for introducing us to the mathematics software system ``SageMath'' and helping us to produce some graphs of the Robin eigenvalues of the square. KG acknowledges support from the Max Planck Institute for Mathematics, Bonn, from October 2017 to July 2018.

All figures presented in this paper were created using SageMath.

\section{Formulas for the eigenvalues and eigenfunctions of the Robin Laplacian for a rectangle.}\label{s4}
\subsection{Main formulas.}\label{ss:2.0}
Here we follow the description given in \cite{GN} and we specialise to 2 dimensions.
For rectangles $\Omega=(0,\ell_1) \times (0,\ell_2) \subset \R^2$ and $(x,y) \in \Omega$,
an orthonormal basis for the Robin problem is given by
\begin{equation}\label{eq:Refnct}
u_{p,q,h} (x,y) = u_{p,h}(x) u_{q,h}(y),
\end{equation}
where, for  $p,q\in \mathbb N$ (where $\mathbb N$ is the set of the non-negative integers), $u_p$ is the $(p+1)$-st eigenfunction of the Robin problem in $(0,\ell_1)$:
$$u_{p,h}(x) = \sin (\alpha_p(h)  x/\ell_1) + \frac{\alpha_p(h)}{h \ell_1} \cos (\alpha_p(h) x/\ell_1)\,,
$$
and similarly for $u_{q,h}(y)$ with $y \in (0,\ell_2)$.
One should assume $\alpha_p(h) \neq 0$ (resp. $\alpha_q(h)\neq 0$), which holds for $h\neq 0$. For $h=0$, the solution is trivial, hence not the right one!
Here $\alpha_p=\alpha_p(h)$ is the solution  in $[p\pi, (p+1)\pi)$ of
\begin{equation}\label{eq:Robin}
\frac{2 \alpha_p}{h \ell_1} \cos \alpha_p  + \left(1 -\frac{(\alpha_p)^2}{h^2\ell_1^2}\right) \sin \alpha_p  =0\,.
\end{equation}
The Robin eigenvalues are then given by
$$ \bigg(\frac{\alpha_p}{\ell_1}\bigg)^2 + \bigg(\frac{\alpha_q}{\ell_2}\bigg)^2. $$
So in 2 dimensions, the Robin eigenvalues correspond to pairs of non-negative integers $(p,q)$.\\
We analyse the $1D$-situation  in more detail and delete the reference to $p,q,h$.
We note that  the condition \eqref{eq:Robin} reads (for $h\neq 0$ and $\alpha \neq 0$),
$$
\frac {\alpha}{h\ell} = \pm \left ( \sin \alpha +\frac{\alpha}{h\ell} \cos \alpha\right)\,.
$$
In this way one understands the symmetry properties of the eigenfunctions better (see Lemma \ref{lemsymrobin}).\\
One also obtains the  localisation of the eigenvalues in the following way.\\
If we consider the symmetric case, $
\frac {\alpha}{h\ell} =  \left ( \sin \alpha +\frac{\alpha}{h\ell} \cos \alpha\right)\,
$,
we get
$$
\frac {2\alpha}{h\ell}  \,  \sin^2\left(\frac \alpha 2\right) = \sin \alpha\,,
$$
which leads to
\begin{equation}\label{eq:alphaneven}
\alpha \tan \left(\frac \alpha 2\right) = h \ell\,.
\end{equation}

Similarly, if we consider the antisymmetric case, $
\frac {\alpha}{h\ell} = -  \left ( \sin \alpha +\frac{\alpha}{h\ell} \cos \alpha\right)\,
$,
we get
$$
\frac {2\alpha}{h\ell}  \,  \cos^2\left(\frac \alpha 2\right) = - \sin \alpha\,,
$$
which leads to
\begin{equation}\label{eq:alphanodd}
\frac {\alpha}{h\ell} = -  \tan  \left(\frac \alpha 2\right)\,.
\end{equation}

With these formulas in mind, we get  simpler
expressions for the eigenfunctions.\\
 In the first case, we observe that
$$
\begin{array}{ll}
u (x) & = \sin (\alpha x/\ell ) + \frac{\alpha}{h \ell} \cos (\alpha x/\ell)\\ & = \sin (\alpha x/\ell ) + \cotan (\frac \alpha  2) \cos  (\alpha x/\ell)\\ &  = \frac{1}{\sin \frac \alpha 2} \, \cos ( \frac{\alpha x}{\ell} -\frac \alpha 2 )  \,.
\end{array}
$$
In the second case, we observe that
$$
\begin{array}{ll}
u (x) &= \sin (\alpha x/\ell ) + \frac{\alpha}{h \ell} \cos (\alpha x/\ell) \\ & = \sin (\alpha x/\ell ) - \tan (\frac \alpha  2) \cos  (\alpha x/\ell)\\ &  = \frac{1}{\cos \frac \alpha 2} \,  \sin ( \frac{\alpha x}{\ell} -\frac \alpha 2)  \,.
\end{array}
$$
In this way, we clearly see the symmetry properties of the eigenfunctions and we are
closer to the Neumann case by considering $x\mapsto  \cos ( \frac{\alpha x}{\ell} -\frac \alpha 2 )$ or $x \mapsto \sin ( \frac{\alpha x}{\ell} -\frac \alpha 2)$  as eigenfunctions.\\
 The first case corresponds to  $p$ even.
 When $h=0$, we have $\alpha = p\pi$ and
 $ \cos ( \frac{\alpha x}{\ell} -\frac \alpha 2 ) = (-1)^\frac p2 \cos  ( \frac{p\pi  x}{\ell})$.\\
 The second case corresponds to $p =2n+1$ odd ($ n\in \mathbb N$). When $h=0$, we have $ \alpha = p\pi$ and
 $ \sin ( \frac{\alpha x}{\ell} -\frac \alpha 2) = \pm  \cos (\frac{ p \pi x}{\ell})$.

\begin{figure}[!ht]
 \begin{center}
\includegraphics[width=10cm]{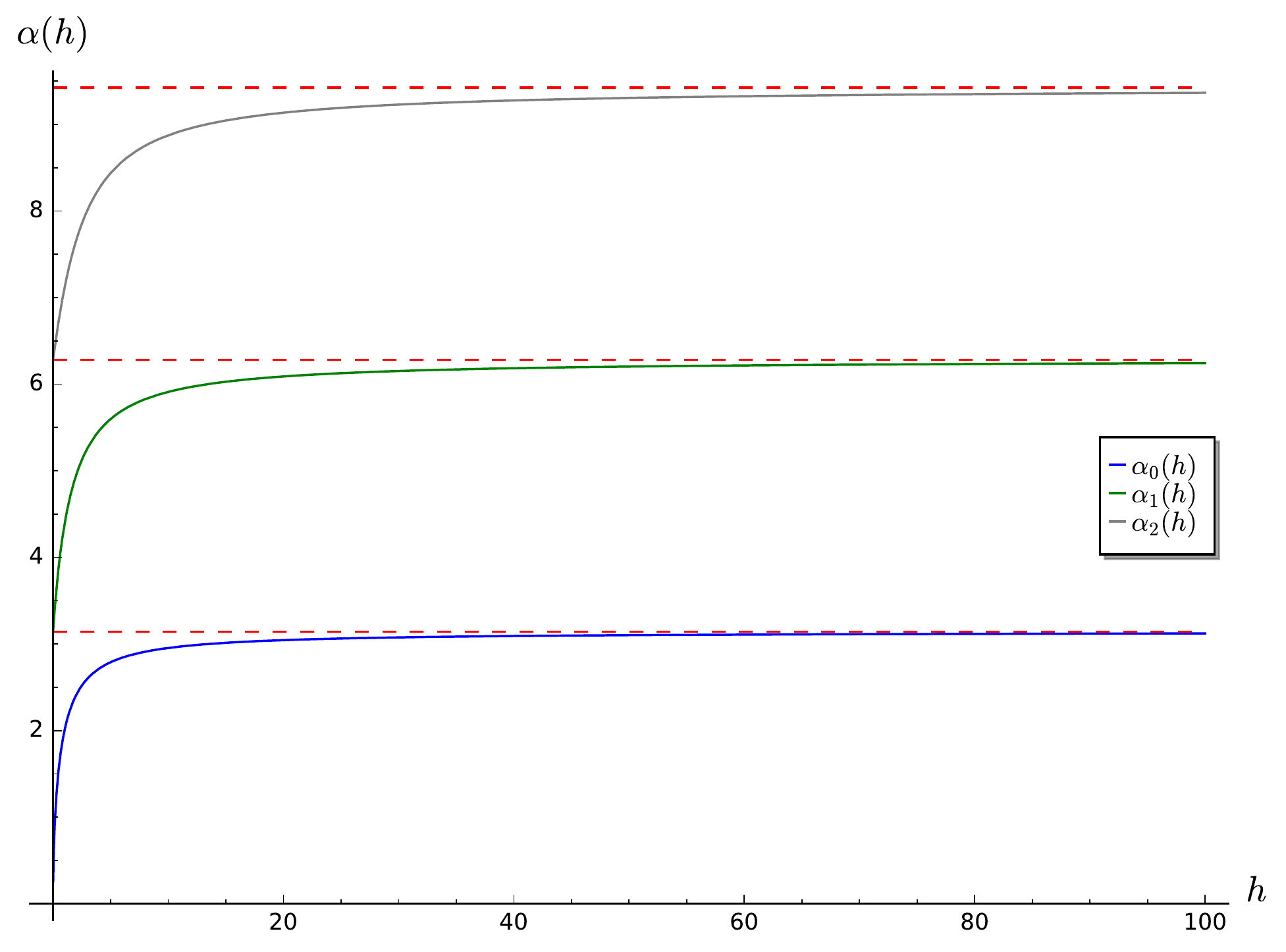}
 \caption{Solutions $\alpha_0(h)$, $\alpha_1(h)$, $\alpha_2(h)$ for $h \leq 100\,$.}
 \label{fig0}
  \end{center}
 \end{figure}

 By setting $\ell = \pi$ and then translating $x \mapsto x + \frac{\pi}{2}$, we have that the Robin eigenfunctions (assuming $h>0$)
 of the square $ S:=(-\frac{\pi}{2},\frac{\pi}{2})^2$ are given by \eqref{eq:Refnct} with
 \begin{equation}\label{form1}
 u_p(x)=\frac{1}{\sin \frac {\alpha_p} 2} \, \cos \left( \frac{\alpha_p x}{\pi}\right)\,,
 \end{equation}
 when $p$ is even, and
 \begin{equation}\label{form2}
 u_p(x)=\frac{1}{\cos \frac {\alpha_p} 2} \,  \sin \left( \frac{\alpha_p x}{\pi}\right)\,,
 \end{equation}
 when $p$ is odd.
 In Figure~\ref{fig0}, we plot $\alpha_0(h)$, $\alpha_1(h)$, $\alpha_2(h)$ for $h \leq 100\,$.

\subsection{Particular cases $k=1,2,3,4$.}\label{SS:2.3}
We recall from the introduction that $\lambda_{1,h}$ and $\lambda_{2,h}$ (which for eigenfunctions of the form $u_{p,q}(x,y)$ correspond to $(p,q) = (0,0), (1,0)$ respectively) are Courant-sharp via Courant's nodal domain theorem and orthogonality of eigenfunctions.
We note that $\lambda_{3,h}(S)$ is not Courant-sharp since it corresponds to the case
where $(p,q) = (0,1)$ so $\lambda_{3,h}(S)=\lambda_{2,h}(S)$.

Consider $\lambda_{4,h}(S)$ with $h>0$. Then $p = q= 1$ and the corresponding eigenfunction is
\begin{equation*}
u_{1,1}(x,y) = \frac{1}{\cos^2 \frac{\alpha_1}{2}} \,  \sin \left( \frac{\alpha_1 x}{\pi}\right)
\sin \left( \frac{\alpha_1 y}{\pi}\right),
\end{equation*}
 for $(x,y) \in (-\frac{\pi}{2},\frac{\pi}{2})^2$.\\
We see that $ x=0$ and $ y=0$ are nodal lines of $u_{1,1}(x,y)$ which partition $S$ into
$4$ nodal domains. There cannot be any further nodal lines of $u_{1,1}(x,y)$ as these would give rise to
additional nodal domains so we would get a contradiction to Courant's nodal domain theorem. Therefore
$\lambda_{4,h}(S)$ with $h\geq 0$ is Courant-sharp.\\
Hence,  from this point onwards, we are only interested in the remaining eigenvalues, i.e. in the eigenvalues $\lambda_{n,h}(S)$ with $n \geq 4$.
Note that, due to the monotonicity of the Robin eigenvalues with respect to $h$, we have  for $n \geq 4$,
\begin{equation}\label{eq:lam4}
\lambda_{n,h}(S) \geq \lambda_{4,h}(S) \geq \lambda_{4,0}(S) = 2\, .
\end{equation}

\subsection{Symmetry properties.}\label{s8}
The use of symmetries was quite powerful in the context of the Neumann case, \cite{HPS1},
 via an argument due to Leydold, \cite{Ley}. That is, a Courant nodal theorem for eigenfunctions that satisfy
certain symmetry properties. In addition, the number of nodal domains inherits some particular properties from these symmetries.
The goal of this subsection is to show that this invariance by symmetry is common to all the Robin problems on the interval and the square.

\subsubsection{Symmetry of Robin eigenfunctions in 1D.}
We recall that $h=0$ corresponds to  the Neumann case and $h=+\infty$ corresponds to  the Dirichlet case.
The Robin condition for $ [-\frac{\ell}{2},\frac{\ell}{2}]$ reads
$$
\frac{du}{dx} (-\ell /2)  = h u (-\ell /2) \,,\, \frac{du}{dx} (\ell /2)  = - h u (\ell /2) \,.
$$
We also observe the following invariance by symmetry.\\
 \begin{lem}\label{lemsymrobin}
 If $u$ is an eigenfunction of the $1D$-Robin problem, the function $ \tilde u(x)= u (-x)$ is also an eigenfunction of the same problem.
 \end{lem}
 Hence, we necessarily have (using the conservation of the norm) $ u(- x) = \pm u (x)$. Moreover, if $ u(0)\neq 0$, we have $ u(-x) = u (x)$ and if $ u'(0)\neq 0$ we get $ u(-x) =- u (x)$.
 Therefore, the eigenfunctions $ u_p$ (see \eqref{form1} and \eqref{form2}) are alternately symmetric and antisymmetric:
\begin{equation}\label{eq:2.8a}
u_p (-x) = (-1)^{p} u_{p}(x)\, ,
\end{equation}
like in the Dirichlet or Neumann case.

\subsubsection{Symmetry of Robin eigenfunctions in 2D.}

In 2D, we now consider the possible symmetries of a general eigenfunction associated with
the eigenvalues $\lambda_{n,h}$ of $(-\frac{\pi}{2},\frac{\pi}{2})^2$ which reads, 
\begin{equation}\label{eq:2.9}
   u(x,y) = \sum_{i,j : \lambda_{n,h}(S) = \pi^{-2}(\alpha_{i}^2 + \alpha_{j}^2)}
    a_{ij}\,  u_{i}(x)u_{j}(y)\,,
 \end{equation}
 where $ u_{p}$ (or $ u_{p,h}$  if we want to mention the reference to the Robin parameter) is the $ (p+1)-st$ eigenfunction of the $h$-Robin problem in $ (-\frac{\pi}{2},\frac{\pi}{2})$.\\
 By considering the transformation $ (x,y) \mapsto (-x,-y)$, we obtain
 \begin{equation}\label{eq:2.9b}
   u(-x,- y) = \sum_{i,j : \lambda_{n,h}(S) = \pi^{-2}(\alpha_{i}^2 + \alpha_{j}^2)}
    a_{ij}\, (-1)^{i+j}  u_{i}(x)u_{j}(y)\,.
 \end{equation}
 \begin{rem}\label{remsym}
 We note that if $(i+j)$ is odd for any pair $(i,j)$ such that $\lambda_{n,h}(S) = \pi^{-2}(\alpha_{i}^2 + \alpha_{j}^2)$, then we get by \eqref{eq:2.9b}, $ u(-x,-y) = - u(x,y)$ and as a consequence $u$ has an even number of nodal domains.
 As we shall see later, other symmetries related to the finite group generated by the identity and the symmetries $ (x,y)\mapsto (-x,y)$ and $ (x,y)\mapsto (x,-y)$ can be considered {(see \cite{GHPS})}.
 \end{rem}

In what follows, we obtain an upper bound for the number of Courant-sharp Robin eigenvalues of $S$
via arguments that do not depend on the parameter $h$.

\section{Upper  bound for the number of Courant-sharp Robin eigenvalues of a square.}\label{s3}
In this section, we prove $h$-independent bounds  for the number of Courant-sharp Robin eigenvalues.
This was indeed the first step proposed by Pleijel \cite{Pl} in the Dirichlet case to reduce the analysis of the Courant-sharp cases to the analysis of finitely many eigenvalues.  His proof was a combination of the Faber-Krahn inequality  and the Weyl formula. In the Neumann case considered in \cite{HPS1}, a new difficulty arises  as it is not possible to apply the Faber-Krahn inequality to the elements of the nodal partition whose boundaries touch the boundary of the square at more than isolated points. In this section, we extend the analysis to the Robin case.

\subsection{Lower bound for the Robin counting function.}

Recall that for $\lambda>0$, the Robin counting function for the corresponding
eigenvalues of $\Omega$ is defined as
\begin{equation}\label{eq:cfR}
  N_{\Omega}^{R,h}(\lambda) := \#\{k \in \N :  k \geq 1, \, \lambda_{k,h}(\Omega) < \lambda\}.
\end{equation}
Similarly we have the Dirichlet counting function
\begin{equation}\label{eq:cfD}
  N_{\Omega}^{D}(\lambda) := \#\{k \in \N :  k \geq 1, \, \lambda_{k}^D(\Omega) < \lambda\},
\end{equation}
and the Neumann counting function
\begin{equation}\label{eq:cfN}
  N_{\Omega}^{Ne}(\lambda) := \#\{k \in \N :  k \geq 1, \, \lambda_{k}^N(\Omega) < \lambda\}.
\end{equation}

  Due to the monotonicity of the Robin eigenvalues with respect to $h \in [0,+\infty)$, it is rather easy to have a lower bound for the $N_{\Omega}^{R,h} (\lambda)$. In particular, we have
 $$
 N_{\Omega}^{R,h}(\lambda) \geq N_{\Omega}^{R,+\infty} (\lambda) = N_{\Omega}^{D}(\lambda)\,.
 $$

We also recall that for the Neumann counting function of $S$, we have
\begin{equation}
\label{eq:WeylN}
 \frac{\pi}{4}\lambda + 2\lfloor\sqrt{\lambda}\rfloor +1 \geq  N_{S}^{Ne} (\lambda)>\frac{\pi}{4}\lambda,
\end{equation}
and for the Dirichlet counting function of $S$, if $\lambda \geq 2$, we have by \cite{Pl},
\begin{equation}
\label{eq:WeylD}
N_{S}^{D} (\lambda)>\frac{\pi}{4}\lambda - 2 \sqrt{\lambda} + 1 \,.
\end{equation}

Assume that $\lambda \geq 2$  (this is true for $n \geq 4$ by \eqref{eq:lam4}). Then, by  \eqref{eq:WeylD}
and monotonicity of the Robin eigenvalues with respect to $h$,
\begin{equation}
  N_{S}^{R,h}(\lambda) \geq N_{S}^{D}(\lambda)
  > \frac{\pi}{4}\lambda - 2\sqrt{\lambda} + 1. \label{eq:R1}
\end{equation}

With $\lambda=\lambda_{n,h}>\lambda_{n-1,h}$ and $\Psi$ an associated
eigenfunction, \eqref{eq:R1} becomes
\begin{equation}
\label{eq:R2}
n > \frac{\pi}{4}\lambda_{n,h} - 2\sqrt{\lambda_{n,h}} + 2\,.
\end{equation}

 We now work analogously to the proof of Proposition 2.1 in \cite{HPS1}.
Denote by $\Omega^{\text{inn}}$ the union of nodal domains of $\Psi$ whose boundaries do not
touch the boundary of $\Omega$ (except at isolated points), and $\mu^{\text{inn}}(\Psi)$ the number of
nodal domains of $\Psi$ in $\Omega^{\text{inn}}$. Similarly denote by $\Omega^{\text{out}}$ the nodal
domains in $\Omega \setminus \Omega^{\text{inn}}$, and $\mu^{\text{out}}(\Psi)$ the number of
nodal domains of $\Psi$ in $\Omega^{\text{out}}$.
 We have that $$\mu^{\text{inn}}(\Psi)= \mu(\Psi)- \mu^{\text{out}}(\Psi)$$ and we require an upper bound
for $\mu^{\text{out}}(\Psi)$.

 \subsection{Counting the number of nodal domains touching the boundary for the Robin problem.}\label{s5}
   We give a proof which  holds for all the Robin problems in the square, except the Dirichlet case.
 We make use of the following theorem that is due to Sturm, \cite{St, BH2}.
\begin{thm}[Sturm, 1836]\label{T-st2r}
Let $u = a_m u_m + \cdots + a_n u_n$ be a  non-trivial linear
combination of eigenfunctions of the 1D-Robin
  problem in $(-\frac{\pi}{2},\frac{\pi}{2})$, with $1 \leq m \le n$, and $\{a_j, m \leq j \le
n\}$ real constants such that $a_m^2 + \cdots + a_n^2 \not = 0$. Then,
the function $u$ has at least $(m-1)$, and at most $(n-1)$ zeros in $(-\frac{\pi}{2},\frac{\pi}{2})$.
\end{thm}
As observed originally by $\AA$.~Pleijel \cite{Pl},
the analysis of the zeros of linear combinations of eigenfunctions appear in the following context.
We observe that if an eigenfunction associated with $\lambda_{n,h}$ (see \eqref{eq:2.9}) satisfies the Robin condition on the square, then its restriction to one side satisfies the Robin condition relative to the interval and is not zero (except of course in the Dirichlet case).
In general, when the multiplicity is not one, this is no longer an eigenfunction but a linear combination of eigenfunctions  on the segment $ (-\frac{\pi}{2},\frac{\pi}{2})$.

For example, the restriction to one side of the square, say $ x=\frac{\pi}{2}$, is a linear combination of eigenfunctions on the segment $ (-\frac{\pi}{2},\frac{\pi}{2})$:
  $$
   u( \pi /2,y) = \sum_{i,j : \lambda_{n,h}(S) = \pi^{-2}(\alpha_{i}^2 + \alpha_{j}^2)}
    a_{ij}\,  u_{i}(\pi /2 )u_{j}(y)\,.
   $$
 We can then use Theorem~\ref{T-st2r} which gives a lower-bound on the number of zeros of
    $ u(\pi /2,y)$ in $ (-\frac{\pi}{2},\frac{\pi}{2})$ by
    $$  i_n(h):= \min (i : \lambda_{n,h}(S) = \pi^{-2}(\alpha_{i}^2 + \alpha_{j}^2))\, ,$$\,
    and an upper-bound by
    \begin{equation}\label{defjn}
      j_n(h):=\max ( j : \lambda_{n,h}(S) = \pi^{-2}(\alpha_{i}^2 + \alpha_{j}^2))\,.
     \end{equation}
  We have
   $$\lambda_{n,h}(S) = (\alpha_{i_n(h)}^2 + \alpha_{j_n(h)}^2)/\pi^2 \geq i_n(h)^2 + j_n(h)^2 \geq j_n(h)^2,$$
   which gives that $$j_n(h)\leq \sqrt{\lambda_{n,h}(S)}\,.
   $$

We can argue in the same way for the other sides of the square.
Therefore, the number of zeros of $u(x,y)$ on the boundary of $S$ is bounded from above by $4 \sqrt{\lambda_{n,h}(S)}$. \\
    Coming back to the number of ``boundary'' nodal domains, we have the following lemma.

\begin{lem}\label{lemmuout}
  Let $\lambda$ be a Robin eigenvalue of $S$ with $h < +\infty$. If $\Psi$ is a Robin eigenfunction associated to $\lambda$, then
  \begin{equation}
  \mu^{out}(\Psi) \leq 4 \sqrt{\lambda}\,.
  \end{equation}
\end{lem}

\begin{rem}
There are other proofs given in Pleijel \cite{Pl} and \cite{HPS1}, but the one given above is much more general and not  restricted
to two-dimensional eigenspaces (and also not based on an explicit knowledge of the eigenfunctions). On the other hand, the claim in \cite{HPS1} is much more  involved. It says that taking the whole boundary into consideration, the number of points  on
the boundary in the nodal set of an eigenfunction $\cos \theta\, u_i(x)u_j(y) + \sin \theta\, u_j(x)u_i(y)$ ($i\neq j$)  is comparable with $i+j$ (See Section 5 of \cite{HPS1}). The proof\footnote{There is a small gap in the proof which can be repaired using Theorem~\ref{T-st2r} due to Sturm.} is restricted to eigenfunctions whose corresponding eigenvalues have multiplicity 2. It would be interesting to prove the same result for the Robin case for $h < +\infty\,$.
\end{rem}

\subsection{Upper bound for Courant-sharp Robin eigenvalues of a square.}\label{ss3.3}
By Lemma \ref{lemmuout}, we have
\begin{equation}
\label{eq:R3}
\mu^{\text{inn}}(\Psi)\geq \mu(\Psi)-4\sqrt{\lambda_{n,h}}\,.
\end{equation}

Now, $\Omega^{\text{inn}}=\bigcup_i \omega^{\text{inn}}_{i}$ is a finite
union of nodal domains of $\Psi$.  Assuming that $\Omega^{\text{inn}}$ is not empty, we get, on each $\omega^{\text{inn}}_{i}$, by Faber-Krahn (see~\cite{Pl}), that
\begin{equation}\label{eq:R3a}
\frac{A(\omega^{\text{inn}}_i)}{\pi\mathbf{j}^2}\geq \frac{1}{\lambda_{n,h}}\,,
\end{equation}
where $A(\omega^{\text{inn}}_i)$ denotes the area of
$\omega^{\text{inn}}_i$ and $\mathbf{j}$ denotes
the first positive zero of the Bessel function $J_0$. Adding, and
invoking~\eqref{eq:R3}, we find
\begin{equation*}
\frac{\pi}{\mathbf{j}^2}=\frac{A(S)}{\pi\mathbf{j}^2}>
\frac{A(\Omega^{\text{inn}})}{\pi \mathbf{j}^2}\geq \frac{\mu^{\text{inn}}(\Psi)}{\lambda_{n,h}}\geq \frac{\mu(\Psi)-4\sqrt{\lambda_{n,h}}}{\lambda_{n,h}},
\end{equation*}
from which we obtain
\begin{equation}\label{eq:R4}
\frac{\pi}{\mathbf{j}^2} \geq \frac{\mu(\Psi)-4\sqrt{\lambda_{n,h}}}{\lambda_{n,h}}\,.
\end{equation}
Due to \eqref{eq:R3}, this inequality is still true if $\Omega^{\text{inn}}$ is empty.

If we are in the Courant-sharp situation, then $\mu(\Psi)=n$.
Combining~\eqref{eq:R2} and~\eqref{eq:R4}, we find that
\begin{equation}
\label{eq:R5}
0.543229\approx \frac{\pi}{\mathbf{j}^2}
> \frac{n-4\sqrt{\lambda_{n,h}}}{\lambda_{n,h}}
> \frac{\pi}{4} + \frac{2}{\lambda_{n,h}}-\frac{6}{\sqrt{\lambda_{n,h}}}\, .
\end{equation}
The mapping
\[
\lambda\mapsto f(\lambda)= \frac{2}{\lambda} - \frac{6}{\sqrt{\lambda}} + \frac{\pi}{4} - \frac{\pi}{\mathbf{j}^2}
\]
is increasing for $\lambda \geq 4/9$. Moreover, $f(597)<0$ and $f(598)>0$. Thus,
if $\lambda_{n,h}\geq 598$, we violate inequality~\eqref{eq:R5},  and
we are not in the Courant-sharp situation.
So, similarly to \cite{Pl} and \cite[Proposition 2.1]{HPS1}, we obtain the following proposition.
\begin{prop}
\label{lem:R1}
If $\lambda_{n,h} \geq 598$ is a Robin eigenvalue of the Laplacian for $S$, then it is not Courant-sharp.
Alternatively, any Courant-sharp Robin eigenvalue of $S$, $\lambda_{n,h} < 598$.
\end{prop}

\subsection{Proof of Theorem \ref{prop:p1}.}\label{ss3.4}

By invoking the upper bound of \eqref{eq:WeylN}, we obtain an upper bound for $n$ such that $\lambda_{n,h}(S) <598$.
Indeed, suppose $\lambda_{n,h}(S) <598$, then
\begin{align}
n-1 = N_{S}^{R,h}(\lambda_{n,h}(S)) &= \#\{k \in \N :  k \geq 1, \, \lambda_{k,h}(S) < \lambda_{n,h}(S)\} \notag\\
& \leq \frac{\pi}{4} \lambda_{n,h}(S) + 2 \lfloor \sqrt{\lambda_{n,h}(S)} \rfloor + 1 < 518.67.
\end{align}
Hence we have shown Theorem \ref{prop:p1}.\\

We remark that the above arguments do not depend on the Robin parameter $h$.
In the 
sections  that follow, we consider the case where $h$ is large and improve the result.

\section{Analysis as $h\rightarrow +\infty$.}\label{S:hlarge}
In this section we show that for $h$ sufficiently large, the Courant-sharp Robin eigenvalues of the square are the same as those in the Dirichlet case, \cite{Pl,BH1}, that is the first, second and fourth, except possibly the fifth which we deal with in Section~\ref{sp5}. We first briefly revisit the strategy that was used by Pleijel for the Dirichlet problem.
\subsection{Pleijel's approach for Dirichlet.}\label{ss:hlarge1}

Let us come back to Pleijel's argument. We recall from \eqref{eq:WeylD} that if $ \lambda_n \geq 2$ is Courant-sharp, then
\begin{equation}\label{i1}
n >\frac \pi 4 \lambda_n - 2 \sqrt{\lambda_n}  + 2\,.
\end{equation}
On the other hand, if $\lambda_n$ is Courant-sharp, the Faber-Krahn inequality gives the necessary condition
\begin{equation}\label{i2}
\frac{n}{\lambda_n} \leq \pi {{\bf j}}^{-2} <  0.54323\,.
\end{equation}
Recall that ${\bf j}$ is the smallest positive zero of  $J_0$ the
Bessel function of order $0$, and that $ \pi {{\bf j}}^2$ is
the ground state energy of the disc of area $1$.
Combining \eqref{i1} and \eqref{i2}, leads to the inequality
\begin{equation}\label{i3a}
\pi {{\bf j}}^{-2} > \frac \pi 4  - 2 \lambda_n^{-\frac 12}   + 2 \lambda_n^{-1}\,,
\end{equation}
and to
\begin{equation}\label{i3}
\lambda_n \le  50 \,.
\end{equation}

Then the proof is achieved in the following steps (see \cite{BH1} for the full details).
\begin{itemize}
\item  By a direct computation of the quotient of $\frac{n}{\lambda_n}$, it is possible to eliminate all the eigenvalues except for $n=1,2,4,5,7$ and $9\,$.
\item The eigenvalues  for $n=7$ and $n=9$ are eliminated by symmetry arguments  (analogously to Remark~\ref{remsym}).
\item The final step is to analyse the fifth eigenvalue for which a specific analysis of the nodal structure can be done
(see \cite{BH1}).\end{itemize}
In the 
subsections  that follow, we  work through
these steps and investigate the extent to which they still work for $h$ large.

\subsection{Faber-Krahn for the Robin case.}\label{ssFK}
We recall the result of Bossel-Daners \cite{Bos, D}, which asserts that the Robin eigenvalues of the Laplacian
satisfy the following Faber-Krahn inequality. For a Lipschitz domain $\omega \subset \R^2$ and $h > 0$,
\begin{equation}\label{eq:FKR}
  \lambda_{1,h}(\omega) \geq \lambda_{1,h}(D_{\omega}),
\end{equation}
where $D_{\omega} \subset \R^2$ is a disc such that $A( D_{\omega} ) = A(\omega )$.

We note that this  analysis is only interesting for the nodal domains whose boundary meet the boundary of $\Omega$ along at least some arc. For the interior nodal domains, the
approach via the standard Faber-Krahn inequality still applies.
For the boundary domains, we have mixed boundary conditions with Robin on some arcs and Dirichlet on the remaining arcs. For a lower bound, by monotonicity, we can indeed use the
Faber-Krahn inequality (with  a Robin boundary condition on all the boundary).\\

Consider a scaling of the domain $\omega$ by $t>0$, $t\omega :=\{tx \in \R^2 : x \in \Omega\}$.
As observed by Antunes, Freitas and Kennedy, \cite{AFK}, the Robin eigenvalues satisfy the following scaling property.
\begin{equation}\label{eq:Rscale}
  \lambda_{n,h}(\omega) = t^2 \lambda_{n,h/t}(t\omega),
\end{equation}
A serious  issue here is that the scaling also affects the Robin parameter.  So, in particular, replacing $D$ by $D_1$, the disc of area $1$, we have
\begin{equation}\label{eq:Rscale2}
 \lambda_{1,h} (D_{\omega})= \lambda_{1, h A(\omega)^\frac 12 } ( D_1) / A(\omega)\,.
\end{equation}
When $h=+\infty$, the reference is $ \lambda_{1, +\infty}  ( D_1) $. In the Robin case, if we start from $h$ large, we will not necessarily have $h A(\omega)^\frac 12$ large if we use this inequality with $\omega$ a ``boundary'' nodal domain. Hence we have to be careful in the application of the Faber-Krahn argument. This is actually the main difficulty.

We recall the asymptotic behaviour of the first Robin eigenvalue as the Robin parameter tends to $+\infty$ or to $0$ (see, for example,  \cite{GN}). \\

We recall that $\lambda_{1,\tilde h} (D_1) \mapsto \lambda_{1,+\infty} (D_1)= \pi {\bf j}^2$ as $\tilde h\rightarrow +\infty$, and that there exists $c >0$ such that,  as $\tilde h\rightarrow +\infty$,
 \begin{equation}\label{maxlambda1}
 \lambda_{1,\tilde h} (D_1) =  \lambda_{1,+\infty } (D_1) - \frac {c}{\tilde h} + \mathcal O \left(\frac{1}{\tilde h^2}\right)\,.
 \end{equation}
We also recall that there exists $d >0$ such that as $\tilde h \rightarrow 0$,
  \begin{equation}\label{minlambda1}
 \lambda_{1,\tilde h} (D_1) = d \,  \tilde h + \mathcal O (\tilde h^2)\,.
 \end{equation}
We give the proof for completion. To determine the first eigenvalue for the disc of area $1$ and radius $\pi^{-1/2}$\,,
one looks for an eigenfunction of the form $J_0 (\alpha  \pi^{1/2} r)$  where the corresponding eigenvalue is $\pi \alpha^2$.
For the asymptotic behaviour near $h=0$ or $h=+\infty$, we use the Taylor expansion of $J_0$ or $J_0'$ at $\alpha=0$ and $\alpha ={\bf j}$.
The Robin condition\footnote{Note that there is a misprint in \cite{GN} after formula (3.9) for the Robin eigenvalue which is corrected here.} reads
   $$\alpha  \pi^{1/2} J'_0(\alpha) +  h J_0(\alpha)=0\,.$$
We recall that $J_0'(0) =0$ and $J_0 ''(0) < 0$. We get for $h \geq 0$, for the first solution $\alpha ^2  \pi^{1/2}J'_0(0) \sim -  h J_0(0)\,.$   Hence the corresponding eigenvalue satisfies as $h\rightarrow 0$,
 $$
 \lambda_{1,h} (D_1) = - (2\pi^{1/2} J_0(0))  /( J_0''(0)) \, h 
 + \mathcal O (h^2)\,.
 $$

We also have $J_0 ({\bf j})=0$ and $J'_0 ({\bf j}) \neq 0$.  With $\tau =\frac{1}{h}$, we write
 $$\tau \alpha  \pi^{1/2} J'_0(\alpha) +  J_0(\alpha)=0,$$ and expanding at $\alpha ={\bf j}$, we obtain:
$$
 \alpha = {\bf j} -  \pi^{\frac 12} {\bf j} \,\tau + \mathcal O (\tau^2)\,,
 $$
 and
 $$
 \pi \alpha^2 = \pi {\bf j}^2 - 2\pi^{\frac 32} {\bf j} ^2 \, \tau   +  \mathcal O (\tau^2)\,.
 $$
The proof gives an explicit value for the constants $c$ and $d$ in \eqref{maxlambda1} and \eqref{minlambda1}.\\

We will apply the Faber-Krahn inequality to a nodal domain of a Robin eigenfunction $u=u_{n,h}$ associated with $\lambda_{n,h}$.
We observe that an eigenfunction $u$ can be extended to all of $\R^2$ as a solution $\tilde u$ of
$-\Delta \tilde u = \lambda \tilde u$ (we have an explicit expression as a trigonometric polynomial). Hence the nodal sets of $\tilde u$ have a nice local structure (see P. B\'erard \cite{Be} for a survey) and
have the same properties as in the Dirichlet case. In particular, these nodal sets are locally Lipschitz domains (actually with piecewise analytic boundary). If we observe that a nodal set of $u$ is the intersection of a nodal set of $\tilde u$ with the square $S$, we immediately deduce that the $\omega_i^{\text{inn}}$ are Lipschitz domains.

The regularity of the ``boundary domains'' $\omega_j^{\text{out}}$ has to be analysed. By Lemma~\ref{lemmuout},
the nodal set intersects the boundary finitely many times, so $\partial \omega_j^{\text{out}}$ consists of a finite number of arcs belonging either to $S$ or to $\partial S$. So we can apply Theorem~4.1 of \cite{BG2}.  Alternatively, we can use the strategy given in Section~3 of \cite{K09} to obtain \eqref{eq:FKR} for these domains
(see also \cite[p. 3620]{K11}).  We will discuss the regularity of the nodal domains further in Section \ref{sp5}.\\

Note also that for a ``boundary'' domain $\omega_j^{\text{out}}$, $u \vert_{\omega_j^{\text{out}}}$ satisfies a mixed Robin-Dirichlet condition on its boundary but we can use the monotonicity with respect to the Robin parameter which leads to
 \begin{equation}\label{eq:4.9}
 \lambda_{n,h} \geq \lambda_{1,h} (\omega_j^{\text{out}})\,,
 \end{equation}
 and then use the pure Robin Faber-Krahn inequality.

\subsection{Pleijel's approach as $h\rightarrow  +\infty$\,.}
In light of what was recalled in Subsection \ref{ss:hlarge1} for $h=+\infty$, we now consider
the different steps in the limit $h\rightarrow +\infty$.

We first recall that the eigenvalues depend continuously on $h$ until $+\infty$, in particular
\begin{equation}\label{eq:4.10}
\forall n\in \mathbb N, \lim_{h\rightarrow +\infty} \lambda_{n,h} =\lambda_n^D\,.
\end{equation}

We keep the notation of the previous section. If we are in the Courant-sharp situation, then $\mu( u)=n$,
where $u$ is an eigenfunction associated with $\lambda_{n,h}$.\\

If there exists $\omega_i^{\text{inn}}$ such that $A(\omega_i^{\text{inn}}) \leq A(S)/n$, we are done like in the Dirichlet case. We combine  the latter inequality with inequality \eqref{eq:R3a} to obtain \eqref{i2}.
Together with \eqref{i1}, this gives $\lambda_{n,h} \leq  50$. In particular, for these eigenvalues $n$ is finite and using \eqref{eq:4.10} we get that for $h$ sufficiently large, \eqref{i2} is not satisfied for $n\neq 1,2,4,5,7,9$.

If not, the situation is more delicate, but we can assume that there exists $\omega_j^{\text{out}}$ such that
\begin{equation}\label{eq:4.11}
A(\omega_j^{\text{out}}) \leq A(S)/n\,,
\end{equation}
and we take one of smallest area with this property.\\
Combining~\eqref{eq:R2}, \eqref{eq:4.9}, \eqref{eq:FKR}, \eqref{eq:4.11} and~\eqref{eq:Rscale2}, we find that

\begin{equation}\label{eq:R9}
\frac{A(S)}{\lambda_{1,hA(\omega_j^{\text{out}})^{1/2} }(D_1)}
> \frac{\pi}{4}
-\frac{2}{ \sqrt{\lambda_{n,h}}}
+\frac{2}{\lambda_{n,h}}\, .
\end{equation}

Here, comparing with \eqref{i3a},  we need to have $\tilde h: = hA(\omega_j^{\text{out}})^{1/2}$ large enough if we want to arrive at the same conclusion as for  the Dirichlet case. So we have to find a lower bound for $A(\omega_j^{\text{out}})^{1/2}$. This seems difficult, at least with explicit lower bounds. We will use our initial $h$-independent upper bound from the previous section. Hence, we can assume in this Courant-sharp situation, that
 \begin{equation}\label{eq:R9a}
 n\leq 520\,.
  \end{equation}

Below, we do not try to  obtain
explicit constants. The first claim is that, according to \eqref{minlambda1}, there exist $c_1 >0$ and $h_1 >0$ such that
$$
\lambda_{1,\tilde h} \geq c_1 \tilde h \mbox{ if } 0 \leq  \tilde h\leq h_1\,.
$$
We now assume that $n \leq 520 $ and $\lambda_{n,h}$ is Courant-Sharp and get
$$
\lambda_{520,\infty} \geq \lambda_{n,h} \geq c_1 h A(\omega_j^{\text{out}} )^{-\frac 12} \, ,
$$
if $h A(\omega_j^{\text{out}})^\frac 12 \leq h_1$. This gives a contradiction if $ c_1 h^2 >  h_1 \lambda_{520,\infty}$. Hence, assuming $$h >h_1^\frac 12 c_1^{-\frac 12}  \lambda_{520,\infty}^{\frac 12}\,,$$ we can now assume that
$$
h A(\omega_j^{\text{out}})^\frac 12 > h_1\,.
$$
Now, we have
$$
\lambda_{520,\infty}  \geq \lambda_{n,h} \geq A(\omega_{j}^{\text{out}} )^{-1} \lambda_{1, h_1} (D_1)\,,
$$
which implies
$$
A(\omega_{j}^{\text{out}} )\geq \lambda_{1, h_1} (D_1)/ \lambda_{520,\infty} \,.
$$
This gives the existence of $c_2 >0$ such that
$ A(\omega_{j}^{\text{out}} )\geq c_2$  (see also Lemma~\ref{lem5.1}).\\
Coming back to \eqref{eq:R9}, we have
\begin{equation} \label{eq:R9bis}
\frac{\pi^2}{\lambda_{1,c_2^{1/2}h}(D_1)} > \frac{\pi}{4}
-\frac{2}{ \sqrt{\lambda_{n,h}}}
+\frac{2}{\lambda_{n,h}}\, .
\end{equation}
Hence for $h$ large enough, we also get in this case that $\lambda_{n,h} \leq 50$
 (compare with inequality \eqref{i3a}).

 We can now follow the proof of Pleijel for the Dirichlet case.\\
 The first step was to achieve (assuming $h$ large enough) the restriction to the three cases left by Pleijel.
  This step  now follows (using the continuity \eqref{eq:4.11} of the eigenvalues with respect to $h$ as $ h \rightarrow +\infty$ as already observed in the previous case).
  \\
The second step is to rule out the cases $\lambda_{7,h}(S)$ and, for $h$ sufficiently large, $\lambda_{9,h}(S)$.
Here the symmetry argument due to Leydold holds in the same way as for the Dirichlet case \cite{BH1}
for the two cases corresponding to the seventh and the ninth Robin eigenvalues.
We briefly recall the relevant particular case of the argument due to Leydold.

\begin{lem}\label{lem4.1} Let $0\leq h < +\infty$. Suppose that $\lambda_{n,h}(S)$ is a Robin eigenvalue with corresponding
eigenfunction defined in \eqref{eq:2.9}. Suppose that $n$ is odd and that the conditions of Remark \ref{remsym} are satisfied. Then $\lambda_{n,h}(S)$
is not Courant-sharp.
\end{lem}
We know indeed by the standard Courant nodal  domain theorem that the number of nodal domains is not larger than $n$ and by Remark \ref{remsym} that it is even. Hence the number is less than $n$.\\

As an application, we observe that any eigenfunction corresponding to the seventh  Robin eigenvalue is a linear combination of $u_{2,1}(x,y)$ and $u_{1,2}(x,y)$ (see Figure~\ref{fig:18evals} and Appendix~\ref{sA}) and that $1+2$ is odd. So $\lambda_{7,h}(S)$ is not Courant-sharp for any $h \geq 0$.

 Similarly, for $h$ large, any eigenfunction corresponding to the ninth  Robin
 eigenvalue is a linear combination of $u_{3,0}(x,y)$ and $u_{0,3}(x,y)$
  (see Figure~\ref{fig:18evals} and Appendix~\ref{sA}) and $0+3$ is odd.\\

 Hence at this stage, we have proved the following  proposition.
 \begin{prop}\label{prop:hlarge}
 There exists $h_1>0$ such that for $h\geq h_1$, the Courant-sharp cases for the Robin problem are the same, except possibly for $k=5$,
 as those for $h=+\infty\,$.
 \end{prop}

 So, having in mind what was done for the Dirichlet case \cite{BH1}, in order to prove Theorem \ref{thm:hlarge}
 for $h$ large enough it remains to count the number of nodal domains of any eigenfunction corresponding
 to the fifth eigenvalue.  This will be analysed in Section~\ref{s:5CS} as a direct consequence of Section \ref{sp5}.

  \section{A general perturbation argument.}\label{sp5}
  \subsection{Preliminary discussion.}

  We analyse a $\theta$-dependent family
  $\Phi_{h,\theta}$ of eigenfunctions, more explicitly
  $$
    \Phi_{h,\theta,p,q}(x,y)   =\cos \theta \, u_{p,h} (x) u_{q,h} (y) + \sin \theta\,  u_{p,h} (y) u_{q,h} (x) \, ,$$
  for $(x,y) \in (-\frac{\pi}{2},\frac{\pi}{2})^2$.\\
 For most of the arguments  in this section, we will not use the explicit expression of the eigenfunction, but only the property that $\Phi_{h,\theta}$ is a very smooth family of eigenfunctions (with respect to $h$ and $\theta$)
  where,  for $h \in (0,+\infty]$, $\Phi_{h,\theta}$ is an eigenfunction of the $h$-Robin  Laplacian
  associated with a smooth eigenvalue $\lambda(h)$. The parameter $\theta$, which above belongs to $ \mathbb R /(2\pi \mathbb Z)$, could also be thought of as belonging to some open neighbourhood of some point $\mathbb \theta_0$  in $\R$.\\ 
 In addition, most of the arguments extend to more general domains.
 We consider the case of bounded, planar domains with piecewise $C^{2,\alpha}$ ($\alpha>0$) boundary.

 For  $h=+\infty$ (or $h=h_0 >0$) and $\theta=\theta_0$, we assume that the number of nodal domains is known (for example,
 that the corresponding eigenvalue is not Courant-sharp). The aim of this section is to prove that by perturbation (i.e. for $|\frac 1 h -\frac{1}{h_0}| + |\theta-\theta_0|$ small enough) the number of nodal domains cannot increase (see Proposition \ref{propPerturb}). \\
 The proof involves various general statements which are interesting in a more general context\footnote{We thank T. Hoffmann-Ostenhof for the useful suggestion to establish and use Lemma~\ref{lem5.1}. We also thank D. Bucur for his enlightening  explanation of the results of \cite{BG1} and \cite{BG2}.}, hence not restricted to the case of the square.

\subsection{Robin Faber-Krahn inequality revisited.}

 \begin{prop}\label{prop5.1}
 Given $h_1>0$ and $M>0$, we consider a smooth family $\Phi_{h,\theta}$ of $h$-Robin eigenfunctions  on
 $\Omega$, where $\Omega$ is a connected, bounded set with piecewise $C^{2,+}$ boundary\footnote{This means $C^{2,\alpha}$ for some $\alpha >0$.} , $\lambda(h) \leq M $ and $h\in I \subset [h_1,+\infty)$ ($I$ being a finite or infinite interval).  Any nodal domain of $\Phi_h$ satisfies the $h$-Faber-Krahn inequality.
        \end{prop}

  \begin{rem}
  We note that the square satisfies the assumptions of Proposition~\ref{prop5.1} but  in this case there is a more direct proof.  As in Subsection~\ref{ssFK}, we indeed observe that $\Phi_{h,\theta}$ admits an extension $\widetilde  \Phi_{h,\theta}$ to $\mathbb R^2$ such that $-\Delta \widetilde  \Phi_{h,\theta} =\lambda(h) \widetilde  \Phi_{h,\theta}$. This gives more information about the local nodal structure of $\Phi_{h,\theta}$ up to
  the boundary (actually in a neighbourhood of $ \overline{S}$ 
  ).
  \end{rem}

  \begin{proof}
 The proposition holds for an open set with  $C^{2,+}$ boundary (hence without corners) as a direct application of Theorem~\ref{thm:nodinfo} in Appendix~\ref{appB}.
 Hence $\Omega$ is a domain with rectifiable boundary of finite length and thus the Faber-Krahn inequality
holds by \cite{BG2} (as mentioned in Subsection~\ref{ssFK}). The same is true for the nodal domains whose
boundaries do not touch a corner.\\

It  remains to treat the corners. The Dirichlet case was addressed by Helffer, Hoffmann-Ostenhof and Terracini in \cite{HHOT}. This  argument involves a local conformal change of coordinates which leads to the analysis of an operator with higher singularities.

We do not know an appropriate reference for the Robin case.
The guess is that the boundary of a nodal domain (whose closure touches the corner) consists of Lipschitz arcs of finite length, including the arcs for which one end touches the corner, which would allow us to use the Robin Faber-Krahn inequality for Lipschitz domains.
Instead we use that according to \cite{BG1}, the $h$-Faber-Krahn inequality holds for any open set with finite area. In this general case, the first eigenvalue is defined as in Definition 4.2 of \cite{BG1}. It is also proven in \cite{BG1} that with this choice  of definition, this eigenvalue is not larger than any other definition given in a more regular situation.
\end{proof}

     \begin{lem}\label{lem5.1}
   Let $h_0 >0$ and $M >0$. Then, under the same hypotheses as in Proposition~\ref{prop5.1},  there exists $ \epsilon_0 >0$ such that no nodal domain of an eigenfunction $\Phi_h$ associated with $\lambda(h)$ for the Robin problem with parameter  $h\geq h_0$ in some open set $\Omega$ and $\lambda(h) \leq M$ can have area less than $\epsilon_0$. (This includes the Dirichlet case).
   \end{lem}

   \begin{proof}
   This follows  directly from the $h$-Faber Krahn inequality.  If $\omega$ is a nodal domain of $\Phi_h$ satisfying the assumptions of the lemma, we have
 \begin{equation}
 M\geq   \lambda(h)\geq \lambda(h_0) \geq \lambda_{ 1, h_0} (D_\omega) = \lambda_{ 1, h_0A(\omega)^\frac 12} (D_1)/A(\omega) \sim d \, h_0/A(\omega)^{\frac 12}\,.
  \end{equation}
 This shows that as soon as we avoid the Neumann situation, the ground state energy in a domain $\omega$ tends to $+\infty$ as the area of the domain tends to $0$.
   \end{proof}
   \subsection{On the nodal set at the boundary.}

          \begin{prop}\label{eq:ass}
 Under the assumptions of Proposition \ref{prop5.1}, there exists $ C>0$ such that, for any $ h\in I$ and any $\theta$, the number of zeros of $\Phi_{h,\theta}$ at the boundary is less than $C$.
        \end{prop}
        \begin{rem}
      In the case of the square the proposition follows
  from Sturm's theorem.\end{rem}

\begin{proof}
We will use the Euler formula with boundary. The conditions for its application are satisfied by using Theorem~\ref{thm:nodinfo} and it reads as follows (see, for example, \cite{HOMN}).

\begin{prop}\label{chapBH.Euler}
Let $\Omega$ be an open set in $\R^2$ with $C^{2,+}$ boundary, $u$ a Robin eigenfunction with $k$ nodal domains, $N(u)$ its zero-set.
 Let $b_0$ be the number of components of $\partial \Omega$ and $b_1$ be the number of components of $N(u) \cup\partial \Omega$. Denote by $\nu({\bf x}_i)$ and $\rho({\bf y}_i)$ the numbers of curves ending at  critical point ${\bf x}_i\in N(u)$, respectively ${\bf y}_i \in N(u)\cap \partial \Omega$. Then
\begin{equation}\label{chapBH.Emu}
k= 1 + b_1-b_0+\sum_{{\bf x}_i}\Big(\frac{\nu({\bf x}_i)}{2}-1\Big)
+\frac{1}{2}\sum_{{\bf y}_i}\rho({\bf y}_i)\,.
\end{equation}
\end{prop}

In our application, we immediately obtain that the number $\rho(u)$ of boundary points (actually counted with multiplicity) in the nodal set of $u$ satisfies
$$
\rho(u) \leq 2k -2\,.
$$
To achieve the proof, we observe that by Courant's  nodal domain theorem, $k$ is less than the minimal labelling of $\lambda (h)$ and that this labelling is uniformly bounded if $\lambda (h)$  is uniformly bounded. By monotonicity, this labelling is indeed bounded by the maximal  labelling of an eigenvalue $\lambda_j (h_1)$ satisfying $\lambda_j(h_1) \leq M$.\\

It remains to treat what is going on in the neighbourhood of a corner $ x_c$.
We first show that there cannot exist an infinite sequence of zeros of $u$ in the boundary  (outside the corner)
tending to the corner $x_c$.
Indeed, by Proposition~\ref{prop5.1}, similarly to the proof of Lemma~\ref{lem5.1}, there exists
some sufficiently small $\epsilon >0$
such that any line starting from one of these zeros (which necessarily belongs to the boundary of one nodal domain)
should cross $ \partial D(x_c,\epsilon) \cap \Omega $  transversally and only once.
Hence the number of points is finite, and moreover not greater than the cardinality of
$N(u) \cap D(x_c,\epsilon) \cap \Omega $.
Observing that,  by Lemma~\ref{lem5.1}, the number of nodal domains of $u$ in $\Omega$ is the same as the number
of nodal domains of $u$ in $\Omega \setminus D(x_c) $,
we can apply the Euler Formula in $\Omega \setminus D(x_c) $ and get the same bound.
\end{proof}

 \subsection{On the variation of the cardinality of the nodal domains by perturbation.}
We assume that  $\Omega$ is a bounded, planar domain with piecewise $C^{2,+}$ boundary.
Our main result is  the following proposition.
\begin{prop}\label{propPerturb}
Under the previous assumptions on $\Omega$ and the family $\Phi_{h,\theta}$, let $\rho(h,\theta)$  denote
the cardinality of the nodal domains of  $\Phi_{h,\theta}$.
For any $\theta_0$, $h_0\in (0,+\infty]$, there exists $\eta_0>0$ such that if $|\frac 1h-\frac{ 1}{h_0}| + | \theta-\theta_0| <\eta_0 $, then
$$
\rho(h,\theta) \leq \rho(h_0,\theta_0)\,.
$$
\end{prop}
We prove this proposition in the following subsections by analysing
what is going on at the interior critical points and at the boundary points of the zero set.

\subsubsection{Analysis in a neighbourhood of an interior point.}~\\
We treat what is going on at an interior point $z_0$.
   We assume that $z_0$ is a critical point of $\Phi_{h_0,\theta_0}$ associated with an eigenvalue $\lambda(h_0)$.
   We choose $ \epsilon_0 > 0$ small enough such that
   \begin{itemize}
   \item  $D(z_0,\epsilon_0) \subset \Omega$;
   \item Lemma~\ref{lem5.1} applies with $M > \lambda(h_0)$;
   \item  the circle $\mathcal C(z_0,\epsilon_0)$  crosses the $2\ell$ half-lines emanating from $z_0$ transversally at $2 \ell$ points $z_j(h_0,\theta_0)$ ($j=1,\dots,2\ell)$.
   \end{itemize}
   Here we have used the general results on the local structure of an eigenfunction of the Laplacian (see \cite{Be}
    and Appendix~\ref{appB}).

   \begin{lem}\label{lem5.3}
  With the previous notations and assumptions  of Lemma~\ref{lem5.1},  there exists $\eta_0>0$ such that if $|\frac 1h-\frac{ 1}{h_0}| + | \theta-\theta_0| <\eta_0 $, then the number of nodal domains of $\Phi_{h,\theta}$ intersecting
  the  disc $D(z_0,\epsilon_0)$ cannot increase.\\
  \end{lem}

  \begin{proof}
  If we look at the nodal structure inside $D(z_0,\epsilon_0)$, we have $2\ell $ local nodal domains.

  By local nodal domain of an eigenfunction $\Phi_{h,\theta}$, we mean the nodal domains of the restriction of $\Phi_{h,\theta}$ to $D(z_0,\epsilon_0)$. We note that any local  nodal  domain belongs to a global nodal domain but that two distinct local  nodal domains can be included in the same global  nodal domain.

  In this case, there exists a path $\gamma$ in $\Omega$ joining these two local domains on which $\Phi_{h,\theta}$ is positive  (or negative), which necessarily will not be included in $D(z_0,\theta_0)$.\\

Starting from $(h_0,\theta_0)$ we  now look at a small perturbation.  By considering the restriction of $\Phi_{h,\theta}$ to the circle $\partial D(z_0,\epsilon_0)$, we observe that the $2\ell$ zeros of $\Phi_{h,\theta}$ in $\partial D (z_0,\epsilon_0)$ move very smoothly, we denote them by $z_j(h,\theta)$.

 We indeed observe that the tangential derivative of $\Phi_{h_0,\theta_0}$ at each point $z_j(h_0,\theta_0)$ is not zero
 (again we use the general results for eigenfunctions, in particular the transversal property, see Appendix~\ref{appB}). By perturbation, this condition is still true if we choose $\eta_0$ small enough. Hence the restriction of $\Phi_{h,\theta}$ changes sign at each point $z_j(h,\theta)$. Moreover, there are $2\ell$ local domains $\omega_{j} (h,\theta)$  of $\Phi_{h,\theta}$ with the property that $\partial  \omega_{j} (h,\theta)$ intersects $\partial D (z_0,\epsilon_0)$ along the arc $(z_j (h,\theta), z_{j+1} (h,\theta))$ (with the convention that $j+1$ is $1$ for $j=2 \ell$).

  We now observe that if $\omega_{j} (h_0,\theta_0)$ and $\omega_{j'} (h_0,\theta_0)$
  belong to the same nodal domain ($j\neq j'$), the property remains true  for $(h,\theta)$ sufficiently close
   to $(h_0,\theta_0)$ (i.e. for $\eta_0$ in the lemma sufficiently small).\\
   If,  for $(\theta_0,h_0)$, $\omega_{j} (h_0,\theta_0)$ and $\omega_{j'} (h_0,\theta_0)$ do not belong to the same nodal domain, then there are two cases
   \begin{itemize}
 \item either the situation is unchanged by perturbation;
\item   or they belong after perturbation to the same nodal domain via a new path
 in $D(z_0,\epsilon_0)$.
\end{itemize}
In the second case, the number of nodal domains touching $\partial  D(z_0,\epsilon_0)$ is decreasing.\\

On the other hand, by Lemma \ref{lem5.1}, any nodal domain  that intersects $D(z_0,\epsilon_0)$  crosses
$\partial D(z_0,\epsilon_0)$. This achieves the proof.
\end{proof}

\begin{rem}\label{rem5.5}
If $\ell$=2, $\Phi_{h_0,\theta_0}$ is a Morse function whose Hessian has two non-zero eigenvalues of opposite sign. Then, for  $\epsilon_0$ small enough, $\Phi_{h,\theta}$ remains a Morse function for  $\eta_0$ small enough and admits a unique critical point $z_{h,\theta}$ in $D(z_0,\eta_0)$. Then there are four local nodal domains if $ \Phi_{h,\theta} (z_{h,\theta})=0$ and three  local nodal domains if $ \Phi_{h,\theta} (z_{h,\theta})\neq 0$ (see Subsection \ref{sssintp} for a detailed proof).
\end{rem}

\subsubsection{Analysis at the boundary.}~\\
It remains to control what is going on at the boundary.
We consider a point $z_0\in \partial \Omega$ such  that $z_0$ is a zero of $\Phi_{h_0,\theta_0}$
which in addition is assumed to be critical when $h_0=+\infty$. \\

We first assume that we avoid the corners and successively consider three cases:
\begin{itemize}
\item $h_0=+\infty$, perturbation only in $\theta$.
\item $0 < h_0 < +\infty$, general perturbation.
\item $h_0=+\infty$, general perturbation.
\end{itemize}
In the first case, the proof follows the same argument as that used in the proof of Lemma~\ref{lem5.3}
and uses the local structure of a Dirichlet eigenfunction at the boundary  (see \cite{Be} and Appendix~\ref{appB}).\\

For the second case, considering the proof of Lemma~\ref{lem5.3} once again,
we choose $\epsilon_0 >0$ sufficiently small such that $z_0$ is the only boundary point in the nodal set.
Then the proof goes in the same way.\\

In the third case, the situation is more delicate due to the complete vanishing of $\Phi_{+\infty,\theta_0}$ on the boundary, which should not be the case for $h < +\infty$.
To deal with this, we need the following lemma.

\begin{lem}
Let $\theta =\theta_0$ and $Z^{bnd}$  denote the intersection of the nodal set of $\Phi_{+\infty,\theta_0}$ with the boundary. Then for any $\epsilon >0$ there exists $h_\epsilon^*$ such that the set $\{d(z, \partial S) < \epsilon\} \cap  \{d(z, Z^{bnd})>\epsilon\}$ does not meet the zero set of $\Phi_{h,\theta}$
 for any $ h_\epsilon^* \leq h < +\infty$ and any $\theta$ such that $|\theta-\theta_0| < \frac{1}{h_\epsilon^*}$.
\end{lem}
In other words we have some nodal stability up to the boundary as $h\rightarrow +\infty$.\\
\begin{proof}
We consider the following two cases.\\
{\bf At a regular point of the boundary.}\\
 We consider a point $z_0$ of the boundary (or a closed interval $I$ in the boundary) which is not a
 critical point for $\Phi_{+\infty,\theta_0}$. By perturbation, this is still true for $|\theta -\theta_0 | $ small.
  In this case the normal derivative of $\Phi_{+\infty,\theta}$ for $z_0\in I$ does not vanish, and to fix the ideas we can assume that
  $$
  \partial_\nu \Phi_{+\infty,\theta} (z_0,\theta) >c >0
  $$
  (the other case would be treated similarly).
  By continuity, replacing $c$ by $\frac c2$, this is still true for $\Phi_{h,\theta}$, $z$ in a $h$-independent neighbourhood of $I$ and $\frac 1h$ small enough.\\
 On the other hand, we know that $\Phi_{h,\theta} $ satisfies the Robin condition:
 $$
 \partial_\nu \Phi_{h,\theta} (z_0,\theta)
 + h \Phi_{h,\theta} (z_0,\theta)  =0\,.
 $$
 Hence
 $$
 \Phi_{h,\theta} (z_0,\theta)   = - \frac 1 h  \partial_\nu \Phi_{h,\theta} (z_0,\theta)  < 0\,.
 $$
 This implies that  there exists a neighbourhood of $I$ and $\eta >0$  such that, for $\frac 1h + |\theta-\theta_0|<\eta$ ,
 $ \Phi_{h,\theta}$ is negative (actually $< - \frac{c}{ 2h}$).
 \\

\noindent
{\bf At a corner.}\\
 After translation, we assume that the corner is at $(0,0)$.
 We also assume that $(0,0)$ does not belong to the nodal set of $\Phi_{+\infty,\theta_0}$ and that
 $\Phi_{+\infty,\theta_0} < 0$ in $\Omega$ near the corner. \\
 We  now use the previous argument outside of $(0,0)$.
 For  $\epsilon_0 >0$ small enough we can take $\eta >0$ small enough such that, for $\frac 1h + |\theta-\theta_0|<\eta$,  $ \Phi_{h,\theta} (x,y) < 0$
  for $\{x^2 +y^2 = \epsilon_0^2 \} \cap \Omega $.\\
  Suppose now that $\Phi_{h,\theta} (x,y) > 0$ for some $(x,y) \in D((0,0),\epsilon_0)$.
  Then  there is
  a nodal domain inside $D((0,0), \epsilon_0)$  and this is excluded by Lemma~\ref{lem5.1}
  provided that  we have chosen $\epsilon_0$ sufficiently small.
 \end{proof}

   \begin{rem}
   We have not proven in full generality that $\Phi_{h,\theta}$ is negative at the boundary near the corner but this is not  required. We do not know what occurs if the corner belongs to the zero set. \\
   If the corner is not in the zero-set of the Dirichlet eigenfunction, we can prove by the previous argument that  this is still the case for $h$ large enough.
   In the case of the square, we get immediately that
  $$
  \partial^2_{x,y} \Phi_{+\infty,\theta_0} (0,0) < 0\,.
  $$
  We now estimate $\Phi_{h,\theta} (0,0)$. Using the Robin condition, we obtain that
  $$
   \Phi_{h,\theta} (0,0) \sim h^{-2}  \partial^2_{x,y} \Phi_{h,\theta} (0,0)
  $$
  By perturbation,  we also  have
  $$
  \partial^2_{x,y} \Phi_{h,\theta} (0,0) < 0\,.
  $$
  This implies
  $$
    \Phi_{h,\theta} (0,0)  < 0\,.
  $$
   \end{rem}

   This leads to the following result when $z_0 \in \partial \Omega $.  We assume that $z_0$ is a critical point of $\Phi_{+\infty,\theta_0}$ associated with an eigenvalue $\lambda(\infty)$.
   We choose $\epsilon_0$ small enough such that
   \begin{itemize}
   \item Lemma~\ref{lem5.1}  applies with $M > \lambda(h_0)$;
   \item $\mathcal C(z_0,\epsilon_0)\cap \Omega$ crosses the $\ell$ half-lines emanating from $z_0$ transversally at $ \ell$ points $z_j(h_0,\theta_0)$ ($j=1,\dots,\ell $).
   \end{itemize}
   Here we have used the general results for the local structure of an eigenfunction of the  Dirichlet Laplacian
   (see \cite{Be}, see also \cite{HHOT} for the case with corners).

   \begin{lem}\label{lem5.3b}
  With the previous  notation and assumptions  of Lemma~\ref{lem5.1}, there exists $\eta_0>0$ such that if $|\frac 1h-\frac{ 1}{h_0}| + | \theta-\theta_0| <\eta_0 $, then the number of nodal domains of $\Phi_{h,\theta}$ intersecting
  the  disc $D(z_0,\epsilon_0)$ cannot increase.\\
  If $\ell=1$, the number of nodal domains equals two and remains fixed.
  \end{lem}

  \subsection{Application to the square.}
  We come back to the case of the square and prove Theorem~\ref{thm:hlarge}.
  To this end, having in mind Proposition~\ref{prop:hlarge}, it is sufficient to obtain the following.
    \begin{prop}\label{pro:hlarge5}
    There exists $h_0 > 0$ such that for any $h > h_0$, any eigenfunction corresponding to $\frac{1}{\pi^2} (\alpha_0(h)^2 + \alpha_2(h))^2$ has 2, 3, or 4 nodal domains (as in the Dirichlet case). Hence for $h > h_0$, $\lambda_{5,h}$ is not Courant-sharp.
  \end{prop}
  \begin{proof}
  The property is indeed true for $h=+\infty$ and,  by the results of the preceding sections, the number of nodal domains cannot increase and is necessarily $>1$.
  \end{proof}

   In the next section, we carry out a deeper analysis for the eigenfunction associated with the fifth eigenvalue, where we
   count the nodal domains case by case. For some cases, the proof will use the explicit properties of the eigenfunctions $\Phi_ {h,\theta}$ (see below).

   In relation to Proposition~\ref{pro:hlarge5}, we note that by choosing non-critical values of $\theta$ we can
   obtain that $2$, $3$ and $4$ nodal domains are attained for $h_0$ large enough.

\section{Particular case $k=5\,$.}\label{s:5CS}

 \subsection{Main statement.}

 Looking at the fifth eigenvalue corresponding to the pair $(0,2)$, which is Courant-sharp for Neumann and not Courant-sharp for Dirichlet, we consider the family of eigenfunctions in $(-\frac \pi 2,+\frac \pi 2)^2$  with $\theta \in (-\pi,\pi]$:
  \begin{align}\label{fiftheigen}
   \Phi_{h,\theta,0,2}(x,y)&:=\cos \theta \cos (\alpha_0(h) x /\pi)  \cos (\alpha_2(h) y /\pi) \notag\\
  & \ \ \ \ \ \ + \sin \theta \cos (\alpha_2(h) x/\pi)  \cos (\alpha_0(h)  y/\pi)\,.
  \end{align}
\noindent
  Up to changing the sign of the eigenfunction, it is sufficient to consider $\theta \in [0,\pi)$.
We  prove the following proposition.
  \begin{prop}\label{pro:hlarge5a}
    There exists $h_0 > 0$ such that for any $h > h_0$, any eigenfunction corresponding to $\frac{1}{\pi^2} (\alpha_0(h)^2 + \alpha_2(h))^2$ has 2, 3, or 4 nodal domains (as in the Dirichlet case).
 More precisely, there are three critical values $\theta_j^*(h)\in [0,\pi)$ ($j=1, 2,3$) such that
 $$
 \theta_1^*(h) =  \arctan \left(-  \frac{1}{q_2(h)}\right)\,, \,\theta_2^*(h) = \frac \pi 2 - \theta_1^*(h)\,, \,\theta_3^*= \frac {3\pi}{4}\,,
 $$
 where
 $$q_2(h)=  \frac{\cos \left( \frac{\alpha_{2} }{2} \right)}{  \cos \left( \frac{\alpha_{0} }{2}\right)}\, ,$$
 and  such that $\Phi_{h,\theta}$ has:
 \begin{itemize}
\item  $3$ nodal domains for $\theta \in [0,\theta^*_1(h)]$;
\item $2$ nodal domains for $\theta \in (\theta_1^*(h),\theta^*_2(h))$;
\item $3$ nodal domains for  $\theta \in [\theta_2^*(h),\theta_3^*)$;
\item $4$ nodal domains for $\theta=\theta_3^*$;
\item $3$ nodal domains for $\theta\in (\theta_3^*,\pi)$.
\end{itemize}
      \end{prop}

  Note that for the whole family of eigenfunctions, we have symmetry with respect to the two axes. In addition,
  the corresponding eigenvalue $\frac{1}{\pi^2} (\alpha_0(h)^2 +\alpha_2(h)^2)$ is the fifth eigenvalue
  for any $h\in [0,+\infty]$  (due to monotonicity of the Robin eigenvalues with respect to $h$ and the table given
  in Appendix~\ref{sA}).\\
  For $h=0$, we have $\alpha_0(0)=0$ and $\alpha_2(0) = 2 \pi  $.

  \subsection{The Dirichlet case.}

\begin{figure}[!h]
\centering
  \includegraphics[width=8cm]{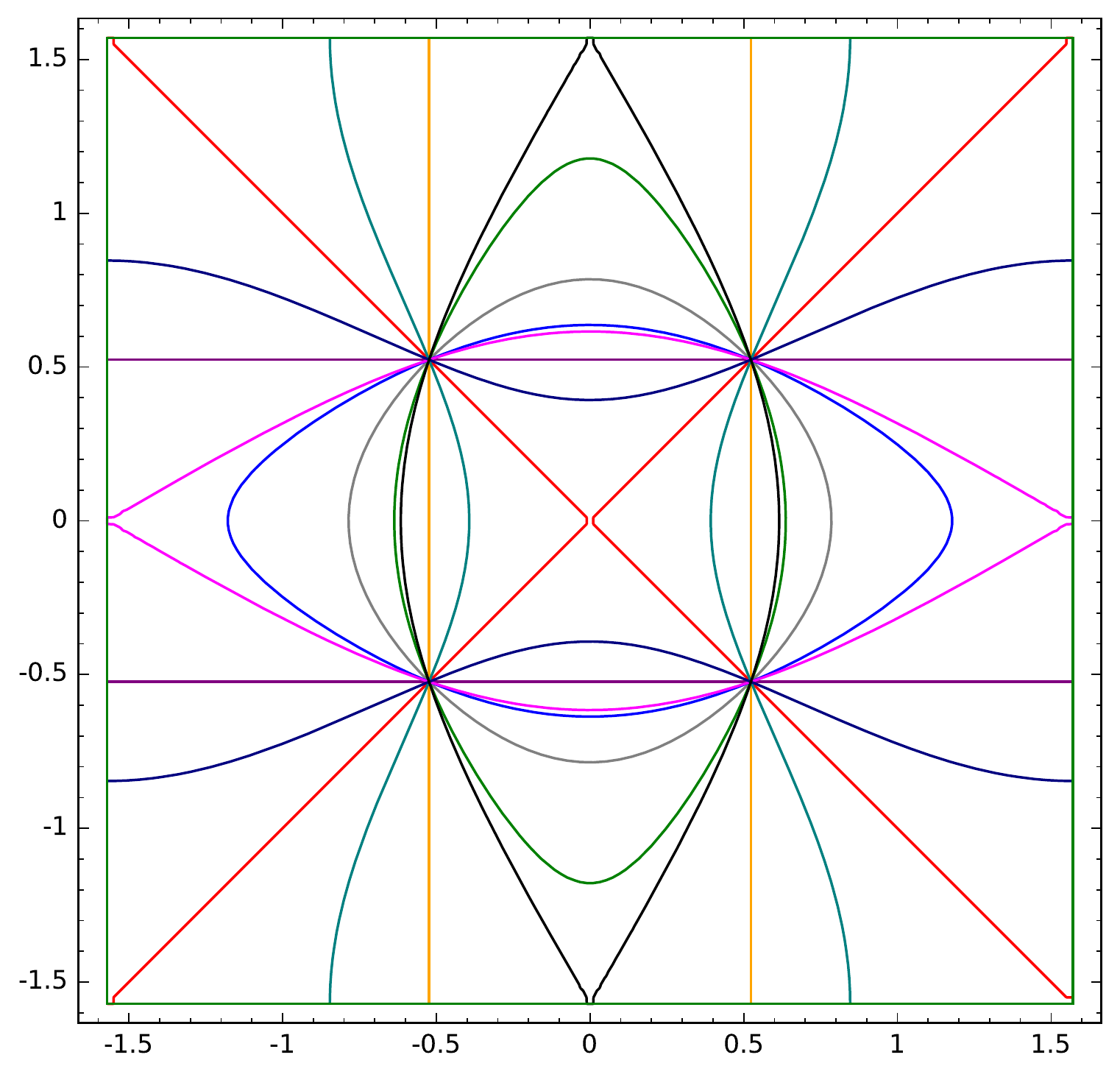}
\caption{The fifth Dirichlet eigenfunction  $\Phi_{+\infty,\theta,0,2}$ on the square $ (-\frac{\pi}{2},\frac{\pi}{2})^2$, for various values of $\theta$.
The values $\theta=0$,  $\theta_1^*=\arctan (1/3)$, $\frac{\pi}8$,
 $\frac{\pi}4$, $\frac{3\pi}{8}$, $\theta_2^*=\frac{\pi}2 - \arctan(1/3)$, $\frac{\pi}2$, $\frac{5\pi}8$, $\theta_3^*=\frac{3\pi}4$, $\frac{7\pi}8$ correspond to the purple, magenta, blue,
grey,  green, black, orange, teal, red, navy curves respectively.}
\label{fig:hDircom}
\end{figure}

For $h=+\infty$, i.e. in the Dirichlet case, we have $\alpha_0(+\infty) =  \pi$ and $\alpha_2(+\infty) = 3 \pi$. The figures of Pockel, \cite{Poc}, give the various possibilities as a function of $\theta$. We refer to \cite{BH1} for a more rigorous mathematical analysis but note that Pockel gives all the possible topologies. He also gives the pictures for the $\theta$ corresponding to transitions  between these topologies.
In Figure~\ref{fig:hDircom}, we plot the fifth Dirichlet eigenfunction
 $$\Phi_{+\infty,\theta,0,2}(x,y) = \cos \theta \cos(x)\cos(3y) + \sin\theta \cos(3x)\cos(y)$$
for $(x,y) \in (-\frac{\pi}{2},\frac{\pi}{2})^2$ and various values of $\theta$. \\
The critical values of $\theta$ corresponding to a change in the number of  interior critical points or the number of boundary  critical  points in the nodal set are $\theta_1^*= \arctan (1/3)$, $\theta_2^*= \frac{\pi}2 - \arctan(1/3)$,
and $\theta_3^* = \frac {3\pi}{4}$.\\

As was proven in \cite{BH1} and can be seen in Figure~\ref{fig:hDircom}, the fifth  Dirichlet eigenfunction has either 2, 3 or 4 nodal domains. More precisely, we have for $\theta \in [0,\pi)$:
\begin{itemize}
\item  $3$ nodal domains for $\theta \in [0,\theta_1^*]$;
\item $2$ nodal domains for $\theta \in (\theta_1^*,\theta_2^*)$;
\item $3$ nodal domains for  $\theta \in [\theta_2^*,\theta_3^*)$;
\item $4$ nodal domains for $\theta=\theta_3^*$;
\item $3$ nodal domains for $\theta\in (\theta_3^*,\pi)$.
\end{itemize}
In what follows, we prove that this holds for $h$ sufficiently large.

\subsection{Application of Section \ref{sp5}.}

For $h$ large enough, we analyse
\begin{align*}
 \psi(\theta,x,y):= \Phi_{h,\theta,0,2}(x,y) &= \cos \theta \cos \left( \frac{\alpha_{0} x}{\pi}\right) \cos \left( \frac{\alpha_{2} y}{\pi} \right)\\
  & \ \ \ \ \ \ + \sin \theta \cos \left( \frac{\alpha_{2} x}{\pi} \right) \cos \left( \frac{\alpha_{0} y}{\pi} \right)\,.
\end{align*}
This solution has a double symmetry with respect to $x\mapsto  -x $ and $y \mapsto  -y$.\\

 \subsubsection{Interior critical points.}\label{sssintp}
   We can look at the critical points of $\psi$  as a function of $\theta$.
   In the case of Dirichlet,  the only possible critical point is
   for $x=y=0$ and can only occur for $\cos \theta +\sin \theta=0$ (we assume $\theta \neq \mathbb Z \frac \pi 2$).

For $\cos \theta + \sin \theta =0$, $x=\pm y$ belong to the zero set  of $\psi$.
We  show that the zero set is exactly given by $x=\pm y$.
We observe that the Hessian of $\psi$ at $(x,y)=(0,0)$ is
    \[
  H_{(x,y)=(0,0)}= \frac{\cos \theta}{\pi^2}
  \begin{pmatrix}
    \alpha_2^2 - \alpha_0^2 & 0 \\
    0 & \alpha_0^2 - \alpha_2^2
  \end{pmatrix},
  \]
  which has negative determinant so $(x,y)=(0,0)$ is a non-degenerate critical point of $\psi$.
  We see that $H_{(x,y)=(0,0)}$ has one positive eigenvalue and one negative eigenvalue, so the Morse index
  of the critical point $(0,0)$ is $1$. By the Morse Lemma, in a neighbourhood $U$ of $(0,0)$, there is a
  diffeomorphism $\phi = (u,v) : U \mapsto V \subset \R^2$ with $\phi(0,0)=(0,0)$ such that $\tilde{\psi}
  :=\psi \circ \phi^{-1}$ has the form
  \begin{equation*}
  \tilde{\psi}(u,v) = \tilde{\psi}(0,0) - u^2 + v^2 = \cos\theta + \sin\theta - u^2 + v^2.
  \end{equation*}
  So we see immediately that the critical point $(0,0)$ is isolated.
  With the condition that $\cos\theta + \sin\theta =0$, the zero set is given by $u= \pm v$.
  Since $\phi$ is a bijection and $x= \pm y$ is contained in the zero set of $\psi$, the zero set of $\psi$
  is given by $x = \pm y$.\\
  More generally, the same proof gives that the zero set of  $\psi(\theta,\cdot) - (\cos \theta + \sin \theta)$ is given near $(0,0)$ by $x=\pm y$. We remark that in this case there are 4 nodal domains.\\

\subsubsection{Boundary edge.}

Considering the boundary edge $x=\frac{\pi}{2}$, we have that either $y = \pm \frac{\pi}{2}$ is in the nodal set, in which case there are 4 nodal domains by symmetry, or $y=\pm \frac{\pi}{2}$ is not in the nodal set. In the latter case,
Theorem~\ref{T-st2r} gives that there are at most two points on the boundary edge $x=\frac{\pi}{2}$ that are in the nodal set. If there are exactly two such points in the nodal set, then this corresponds to 3 nodal domains. If there are no boundary points in the nodal set, then this corresponds to 2 nodal domains. For example, see Figure~\ref{fig:hDircom}.

\subsubsection{Double point  on  the boundary.}
  We now analyse what is going on at the double point on the boundary. This occurs for Dirichlet when $\tan \theta=\frac 13$ and for $y=0$.  Here the situation is simple
   (see \cite{Poc}). We observe that $y=0$ is a double point
   for $\tan \theta = -  \frac{1}{q_2(h)}$. From $\Psi(\theta, \frac{\pi}{2}, y)=0$, we have
   $$\cos \left( \frac{\alpha_{2} y}{\pi} \right) + t   \frac{\cos \left( \frac{\alpha_{2} }{2} \right)}{  \cos \left( \frac{\alpha_{0} }{2}\right)}  \cos \left( \frac{\alpha_{0} y}{\pi} \right) =0\,.
     $$
     The critical $t=\tan \theta$ is defined by $t= - 1/q_2(h)$ with
     $$
     q_2(h)=  \frac{\cos \left( \frac{\alpha_{2} }{2} \right)}{  \cos \left( \frac{\alpha_{0} }{2}\right)}\,.$$
     Hence $t=\frac 13 + \mathcal O (\frac 1h)$, and we have near $y=0$,
     $$ y^2 = \left(c  + \mathcal O \left(\frac 1h\right)\right) \left(t + \frac{1}{q_2(h)}\right).$$

     Again, this is the perturbation of a Morse function depending on the parameters $h$ and $\theta$ with the particularity that when $\psi=0$ and $y=0$, the critical point is always $(\frac \pi 2,0)$.
     We remark that in this case there are 3 nodal domains.

     \subsection{Interior critical points for any $h>0$.}
  In this subsection, we show that there are no other critical points than $(0,0)$ without  any restriction
  on $h>0$.
  It is immediate that $(0,0)$ is a critical point and we get the same condition  as in the Dirichlet case.
   Writing $\psi =0$ and $\nabla \psi =0$, we get as a necessary condition that
   \begin{equation}\label{crit2}
   \alpha_2 \tan \left(\frac{\alpha_2 x}{\pi}\right) = \alpha_0 \tan \left(\frac{\alpha_0 x}{\pi}\right)
    \,,\, \alpha_2 \tan \left(\frac{\alpha_2 y}{\pi}\right) = \alpha_0 \tan \left(\frac{\alpha_0 y}{\pi}\right) \,.
    \end{equation}

  \begin{lem}\label{lemabove} Let $\alpha_0$ and $\alpha_2$ satisfy \eqref{eq:alphaneven}. For $x\in (-\frac \pi 2,+ \frac \pi 2)$,
 $ \alpha_0 \tan (\alpha_0 x/\pi) = \alpha_2 \tan (\alpha_2 x/\pi) $ if and only if $x =0$.
 \end{lem}
 \begin{proof}
 Let us look at the function $$[0,\frac \pi 2] \ni  x\mapsto W(x)= \alpha_0 \sin (\alpha_0 x) \cos (\alpha_2 x) -  \alpha_2 \sin (\alpha_2 x) \cos (\alpha_0 x)\,.
 $$
  Up to some multiplicative renormalisation of the eigenfunctions, we recognise the Wronskian of the eigenfunctions  $u_0$ and $u_2$.
 But for the Wronskian, we have
 $$
 W'(x) =( \lambda_0-\lambda_2) u_0 (x) u_2 (x)\,.
 $$
 Now we observe that  $W(0)=0$ and that by \eqref{eq:alphaneven}, $ W(\frac \pi 2)=0$. Moreover  $W$ has a unique critical point in $(0,\frac \pi 2)$ at the  first  zero of $u_2$. Hence $W(x)$ cannot vanish except at $x=0$ and $\frac \pi 2$.
  \end{proof}
It is clear that this implies that $(0,0)$ is the only possible critical point in $(-\frac \pi 2, \frac \pi 2)^2$. The condition that $\psi (0,0)=0$ implies $\cos \theta + \sin \theta=0$. \\

\section{Analysis of  crossings.}\label{sec:s5}
In this section, we analyse the possible crossings of two curves $h\mapsto \lambda_{p,q,h}(S)$ and $h\mapsto \lambda_{p',q',h}(S)$ defined in an interval of $[0,+\infty)$.  This is indeed quite important  as we want to follow the labelling of these eigenvalues when $h$ varies.
\subsection{A general result.}

\begin{prop}\label{p:crossing}
For distinct pairs $(p,q)$ and $(p',q')$, with $p\leq q$ and $p'\leq q'$, there is at most one value of $h$ in $[0,+\infty)$ such that
$\lambda_{p,q,h}(S) = \lambda_{p',q',h}(S)$.
\end{prop}
\paragraph{Proof}
Suppose that $\lambda_{p,q,h}(S) = \lambda_{p',q',h}(S)$.
Without loss of generality, suppose $p<p' \leq q'<q$.
Consider the variation of
$$
(0,+\infty)\ni h\mapsto \sigma(h):=\frac{1}{\pi^2}\left( \alpha_{p}(h)^2 + \alpha_{q} (h)^2 - \alpha_{p'}(h)^2
-\alpha_{q'}(h)^2\right).
$$

The zeros of $\sigma$ correspond to the values of $h$ for which the curves corresponding to $(p,q), (p',q')$
intersect. To analyse its variation, we note that
$$
\sigma'(h) = \frac{2}{\pi^2}\left( \alpha_{p}(h)\alpha'_{p}(h)  + \alpha_{q} (h)\alpha'_{q}(h) - \alpha_{p'}(h)\alpha'_{p'}(h) -\alpha_{q'}(h)\alpha'_{q'}(h) \right)\,.
$$
Now, we deduce from \eqref{eq:alphaneven} and \eqref{eq:alphanodd}, that $h\mapsto \alpha_k(h)$ satisfies the differential equation
\begin{equation}
\frac{\alpha'_k}{\alpha_k} \left( h\pi + \frac{\alpha_k^2}{2} + \frac{h^2 \pi^2}{2} \right)=\pi\, ,
\end{equation}
which implies
\begin{equation}
\alpha'_k \alpha_k \left( h\pi + \frac{\alpha_k^2}{2} + \frac{h^2 \pi^2}{2} \right)=\pi  \alpha_k^2\, .
\end{equation}
We introduce for $h > 0$ and $k\in \mathbb N$,
$$
a_k (h) =  h\pi + \frac{\alpha_k^2}{2} + \frac{h^2 \pi^2}{2} > 0\, .
$$
We deduce
$$
\sigma'(h) =  \frac{2}{\pi} \left( \frac{\alpha_{p}^2}{ a_{p}}+ \frac{\alpha_{q}^2}{ a_{q}} - \frac{\alpha_{p'}^2}{ a_{p'}} - \frac{\alpha_{q'}^2}{ a_{q'}} \right)
=  - \frac{4}{\pi} \left(h\pi +   \frac{h^2 \pi^2}{2} \right)\left( \frac{1}{a_{p}} + \frac{1}{a_{q}}  -  \frac{1}{a_{p'}} - \frac{1}{a_{q'}} \right) \,.
$$
We now assume that $\sigma (h)=0$, which implies $$ a_{p} + a_{q} = a_{p'} + a_{q'}\,.$$
This gives
$$
\sigma'(h) =  - \frac{4}{\pi} \left(h\pi +   \frac{h^2 \pi^2}{2} \right)\left ( \frac{(a_{p}+a_{q})(a_{p'}a_{q'}-a_{p}a_{q})}{(a_{p}a_{q}a_{p'}a_{q'})}\right)\,.
$$
So the sign of $\sigma'(h)$ is the sign of $ a_{p}a_{q}-a_{p'}a_{q'}$.
For $\epsilon > 0$, we can now write $a_{p}=a_{p'} -\epsilon$ and $a_{q}=a_{q'} +\epsilon$, and compute
$$
a_{p}a_{q}-a_{p'}a_{q'}= (a_{p'}-\epsilon)(a_{q'}+\epsilon) - a_{p'} a_{q'}= (a_{p'}-a_{q'}) \epsilon -\epsilon^2 < 0\,.
$$
Since the derivative of $\sigma$ has constant sign, there can be at most one point of intersection.

\begin{rem}\label{rem:crossing}
  The proof of Proposition~\ref{p:crossing} shows that if $p < p' \leq q' <q$ and $\lambda_{p,q, h^*}=\lambda_{p',q',h^*}$ for some $h^*\geq 0$,
  then the map  $$h \mapsto \pi^{-2}( \alpha_{p'}(h)^2 + \alpha_{q'} (h)^2 - \alpha_{p}(h)^2 -\alpha_{q}(h)^2)$$ is increasing
  for $h>h^*$. Hence the curve $\pi^{-2}( \alpha_{p}(h)^2 + \alpha_{q} (h)^2)$ is below the curve
  $\pi^{-2}( \alpha_{p'}(h)^2 + \alpha_{q'} (h)^2)$ for $h>h^*$.
\end{rem}

\subsection{The eigenvalue $\lambda_{9,h}(S)$.}\label{ssec:9CS}
The ninth eigenvalue of the Neumann Laplacian  for the square is Courant-sharp, \cite{HPS1}, and corresponds to the eigenvalue  $2^2+2^2=8$. This eigenvalue is simple and corresponds to the labelling $(2,2)$. The eigenfunction  reads $$ \Phi_{0,\theta,2,2}(x,y)= \cos 2x\, \cos 2y\,,\mbox{  for } (x,y) \in (-\frac{\pi}{2},\frac{\pi}{2})^2\,. $$
 It is easy to see that the Courant-sharp property is still true for $h$ small enough.
 By deformation, the eigenfunction is
 $$
 \Phi_{h,\theta,2,2}(x,y)= \cos ( \alpha_{2}(h)x/\pi)\, \cos(\alpha_{2}(h) y/\pi),$$
  with corresponding eigenvalue $\frac{2}{\pi^2} (\alpha_{2}(h))^2$. The nodal structure is given by
  $$
   \frac{\alpha_{2}(h)x}{\pi} =- \frac \pi 2\,,\,  \frac{\alpha_{2}(h)x}{\pi}  = \frac{\pi}{2}\,, \frac{\alpha_{2}(h)y}{\pi} = - \frac \pi 2\,,\,  \frac{\alpha_{2}(h)y}{\pi}  = \frac{\pi}{2}\,,
   $$
   hence for this eigenfunction and for $h\in [0,+\infty)$,  there are
   nine nodal domains as long as $\frac{2}{\pi^2} ( \alpha_{2}(h))^2$ is the ninth eigenvalue.\\

   The issue is to follow its labelling and we observe that when $h=+\infty$ the eigenvalue is $18$ and, according to the ordered list of the Dirichlet eigenvalues, has minimal labelling $11$ (see Appendix~\ref{sA}). Because $9 < 11$, this eigenfunction is NOT Courant-sharp  for $h$ sufficiently large.\\

  On the other hand the eigenvalue  $\frac{1}{\pi^2} (\alpha_0(h)^2 + \alpha_3 (h)^2)$ which has minimal  labelling  $10$  for $h=0$  arrives with labelling $9$ at $h=+\infty$. Hence some transition occurs for at least one $h_9^* >0$ which satisfies
 $$
   \alpha_0(h)^2 + \alpha_3 (h)^2 = 2 \alpha_2(h)^2\,.
 $$
By Proposition~\ref{p:crossing}, there is at most one point of intersection between the curves corresponding to $(2,2)$ and $(3,3)$.

 We recall that $\alpha_0(0)=0\,, \,\alpha_1(0) =\pi\,,\, \alpha_2(0) =2\pi \,, \,\alpha_3(0)=3\pi$ and that $\alpha_0(+\infty)=\pi\,, \,\alpha_1(+\infty) =2 \pi\,, \,\alpha_2(+\infty) =3\pi \,, \, \alpha_3(+\infty)=4\pi\,,$ so
 $\alpha_0(h)^2 + \alpha_3 (h)^2 $ is increasing from $9 \pi^2$ to $17 \pi^2$ when  $2 \alpha_2(h)^2$ goes from $8 \pi^2$ to $18 \pi^2$.\\

 In order to show that the curves corresponding to the pairs $(2,2), (3,0)$ do not intersect the curves
 corresponding to the other pairs, we consider the table in Appendix~\ref{sA}.

 From above, we see that the eigenvalues corresponding to the pairs $(3,3)$, $(4,2)$, $(2,4)$ and so on are all larger
 than or equal to $18$.
 So we need to consider the eigenvalues corresponding to the pairs $(3,1)$, $(3,2)$, $(4,0)$, $(4,1)$ and show that they
 do not correspond to the ninth, tenth or eleventh eigenvalues for any $h>0$.
 Numerically we find that,
 \begin{align*}
 \lambda_{3,1,h}(S) &= \lambda_{1,3,h}(S) \geq 18 \text{ for } h > 11.4225, \\
 \lambda_{3,2,h}(S) &= \lambda_{2,3,h}(S) \geq 18 \text{ for } h > 2.6288, \\
 \lambda_{4,0,h}(S) &= \lambda_{0,4,h}(S) \geq 18 \text{ for } h > 1.2668, \\
 \lambda_{4,1,h}(S) &= \lambda_{1,4,h}(S) \geq 18 \text{ for } h > 0.4208.
 \end{align*}
 So we are left to consider $h \leq 11.4225$.

 From below, we see that the eigenvalues corresponding to the pairs $(0,0)$, $(1,0)$, $(0,1)$, $(1,1)$ are smaller
 than or equal to $8$ for all $0<h<+\infty$.
 So we need to consider the eigenvalues corresponding to the pairs $(2,0)$, $(2,1)$ and show that they
 do not correspond to the ninth, tenth or eleventh eigenvalues for any $h>0$.
 Numerically we find that,
 \begin{align*}
 &\lambda_{2,2,h}(S) \geq 13 \text{ for } h > 2.9804\,, \\
 \lambda_{3,0,h}(S) = &\lambda_{0,3,h}(S) \geq 13 \text{ for } h > 3.5468\,.
 \end{align*}
 So we are left to consider $h \leq 3.5468 < 11.4225\,$.

 With the table from Appendix~\ref{sA} in mind, 
 we now plot the Robin eigenvalues of the square $(\alpha_m(h)^2 + \alpha_n(h)^2)/\pi^2$ for $h\leq 12$ corresponding to the pairs $(0,0)$, $(1,0)$,  $(1,1)$, $(2,0)$, $(2,1)$, $(2,2)$, $(3,0)$, $(3,1)$, $(3,2)$, $(4,0)$, $(4,1)$.

 \begin{figure}[!h]
 \begin{center}
\includegraphics[width=10cm]{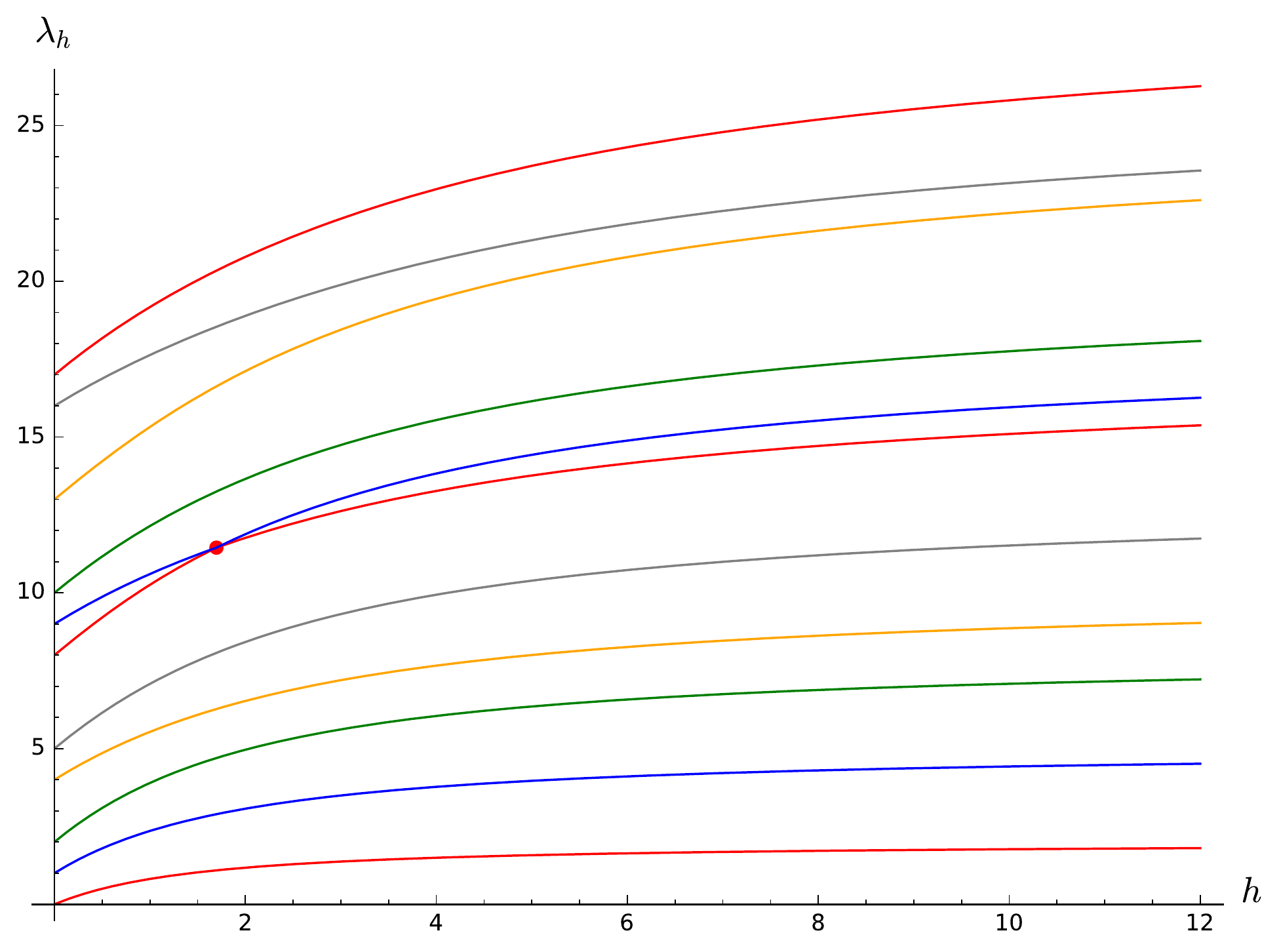}
 \caption{The Robin eigenvalues of the square $(\alpha_m(h)^2 + \alpha_n(h)^2)/\pi^2$ for $h\leq 12$ corresponding to the pairs $(0,0)$, $(1,0)$,  $(1,1)$, $(2,0)$, $(2,1)$, $(2,2)$, $(3,0)$, $(3,1)$, $(3,2)$, $(4,0)$, $(4,1)$.
  The intersection between the curves corresponding to $(2,2)$ and $(3,0)$ occurs at $(1.6970,11.4498)$.}
 \label{fig:18evals}
  \end{center}
 \end{figure}

 From Figure~\ref{fig:18evals}, we see that for $h \leq 12$ the curves corresponding to the pairs $(2,2)$, $(3,0)$ do not intersect the curves corresponding to the other pairs.  By Proposition~\ref{p:crossing}, the curves corresponding to $(2,2)$ and $(3,0)$ intersect for a unique value of $h = h_9^*>0\,$.

 Since $u_{2,2}(x,y)$ is an eigenfunction corresponding to $\lambda_{9,h_9^*}(S)$ that has
 9 nodal domains, we have proved:

\begin{prop}
There exists $h_9^* >0$ such that $\lambda_{9,h}$ is Courant-sharp for $0\leq h \leq h_9^*$ and not Courant-sharp for $h>  h_9^*\,$.
\end{prop}

By the bisection method, we compute $h_9^*$ numerically and find that \break $h_9^* \sim 1.6967$.

By the above, $\lambda_{9,h}$ is given by the pair $(2,2)$ for $h \leq h_9^*$ and the pair $(3,0)$
for $h > h_9^*$. Also, $\lambda_{10,h}$ is given by the pair $(0,3)$ and $\lambda_{11,h}$
is given by the pair $(3,0)$ for $h \leq h_9^*$ and the pair $(2,2)$ for $h > h_9^*\,$.

This shows that whether the eigenfunction corresponding to a Robin eigenvalue of the square
is an odd function or an even function depends on $h$ (in the case where there are crossings).

For example, for $\lambda_{9,h}$ with $h \leq h_9^*$, we have that
$u_{2,2}(-x,-y) = u_{2,2}(x,y)$. On the other hand, for $h > h_9^*$,
$$ u_{3,0}(-x,-y) = -u_{3,0}(x,y) \mbox{  and } u_{0,3}(-x,-y) = -u_{3,0}(x,y)\,.$$
So any linear combination of $u_{3,0}(x,y)$ and $u_{0,3}(x,y)$ is antisymmetric with respect to the
transformation $(x,y) \mapsto (-x,-y)\,$.
Hence $\lambda_{9,h}$ is not Courant-sharp for $h > h_9^*$ (via Lemma \ref{lem4.1}).

For $h = h_9^*$, any eigenfunction corresponding to $\lambda_{9,h_9^*}(S)$ is a linear combination
of $u_{2,2}(x,y), u_{3,0}(x,y)$ and $u_{0,3}(x,y)$, so in general it is neither symmetric nor antisymmetric
with respect to the transformation $ (x,y) \mapsto (-x,-y)$.

\subsection{The eigenvalue $\lambda_{25,h}(S)$.}\label{ss5.3}
Similarly  there are crossings between $\lambda_{25,h}$ and $\lambda_{27,h}$.
As for the ninth eigenvalue, we first show that the curves corresponding to the pairs $(4,3)$, $(5,1)$ do not intersect the curves corresponding to the other pairs by considering the table in Appendix~\ref{sA}.

 From above, we see that the eigenvalues corresponding to the pairs $(5,4)$, $(6,3)$, $(7,0)$ and so on are all larger
 than or equal to $41$.
 So we need to consider the eigenvalues corresponding to the pairs $(5,2)$, $(4,4)$, $(5,3)$, $(6,0)$, $(6,1)$, $(6,2)$ and show that they do not correspond to $\lambda_{25,h}(S), \lambda_{26,h}(S), \lambda_{27,h}(S), \lambda_{28,h}(S)$ for any $h>0$.
 Numerically we find that,
 \begin{align*}
 \lambda_{5,2,h}(S) &= \lambda_{2,5,h}(S) \geq 41 \text{ for } h > 12.6664, \\
 \lambda_{4,4,h}(S) &\geq 41 \text{ for } h > 4.9398, \\
 \lambda_{5,3,h}(S) &= \lambda_{3,5,h}(S) \geq 41 \text{ for } h > 3.4557, \\
 \lambda_{6,0,h}(S) &= \lambda_{0,6,h}(S) \geq 41 \text{ for } h > 3.8230, \\
 \lambda_{6,1,h}(S) &= \lambda_{1,6,h}(S) \geq 41 \text{ for } h > 2.0624, \\
 \lambda_{6,2,h}(S) &= \lambda_{2,6,h}(S) \geq 41 \text{ for } h > 0.4016.
 \end{align*}
 So we are left to consider $h \leq 12.6664$.

 The eigenvalues corresponding to the pairs $(3,2)$, $(3,1)$ and below in the table in Appendix~\ref{sA} are smaller
 than or equal to $25$ for all $0<h<\infty$.
 So we need to consider the eigenvalues corresponding to the pairs $(4,0)$, $(4,1)$, $(3,3)$, $(4,2)$, $(5,0)$
 and show that they do not correspond to $\lambda_{25,h}(S), \lambda_{26,h}(S), \lambda_{27,h}(S), \lambda_{28,h}(S)$
 for any $h>0$.
 Numerically we find that,
 \begin{align*}
 \lambda_{4,3,h}(S) &= \lambda_{3,4,h}(S) \geq  37 \text{ for } h > 11.5497\,, \\
 \lambda_{5,1,h}(S) &= \lambda_{1,5,h}(S) \geq  37 \text{ for } h > 15.3826\,.
 \end{align*}
 So we are left to consider $h \leq  12.6664 < 15.3826\,$.

 With the table from Appendix~\ref{sA} in  mind, 
 we now plot the Robin eigenvalues of the square $(\alpha_m(h)^2 + \alpha_n(h)^2)/\pi^2$ for  $h\leq 16$  corresponding to the pairs $(4,0)$, $(4,1)$, $(3,3)$, $(4,2)$,  $(5,0)$, $(5,1)$, $(4,3)$, $(5,2)$, $(4,4)$, $(5,3)$, $(6,0)$, $(6,1)$, $(6,2)$.

 \begin{figure}[!h]
 \begin{center}
\includegraphics[width=10cm]{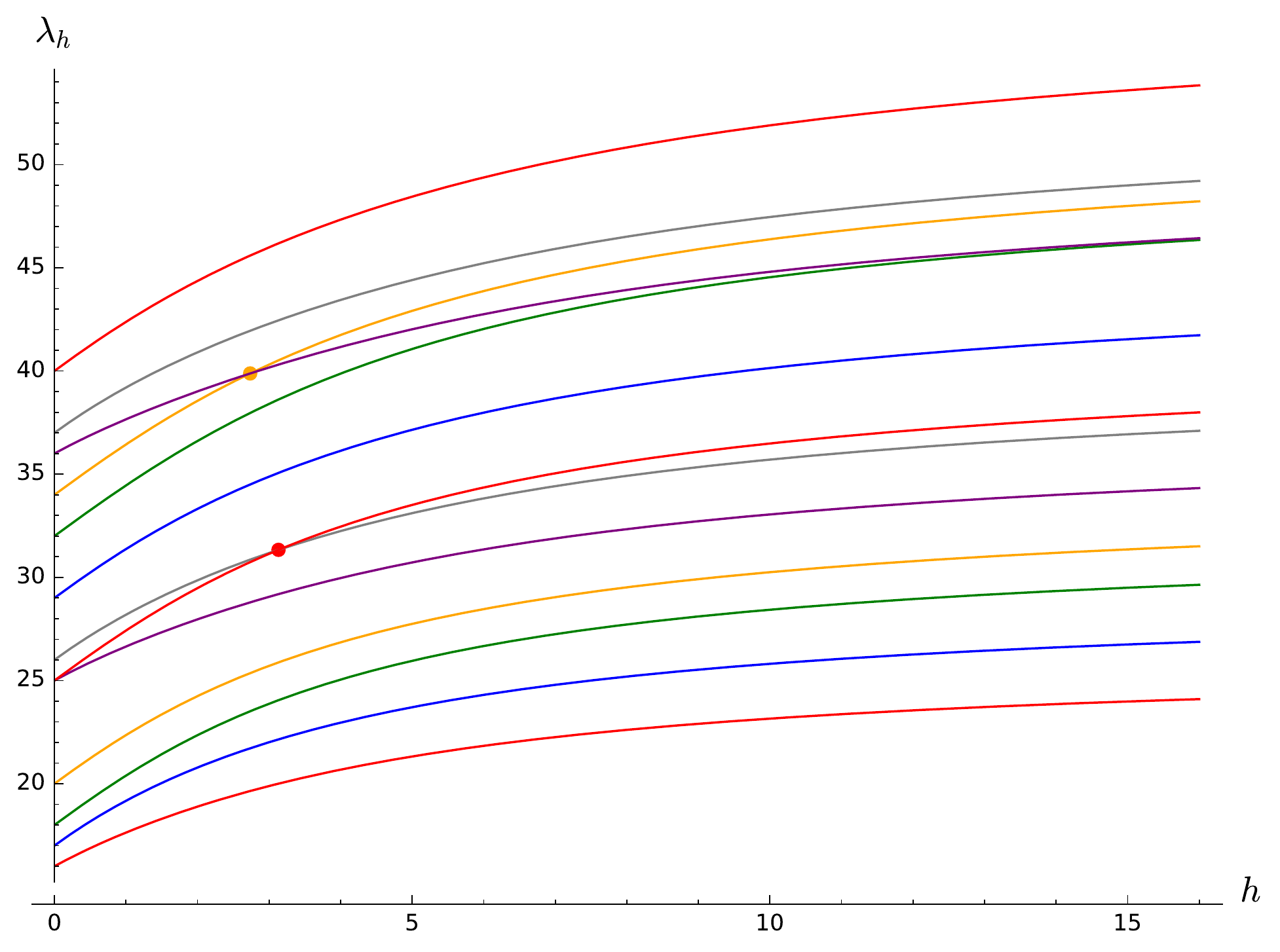}
 \caption{The Robin eigenvalues of the square $(\alpha_m(h)^2 + \alpha_n(h)^2)/\pi^2$ for  $h\leq 16$  corresponding to the pairs $(4,0)$, $(4,1)$, $(3,3)$, $(4,2)$,  $(5,0)$, $(5,1)$, $(4,3)$, $(5,2)$, $(4,4)$, $(5,3)$, $(6,0)$, $(6,1)$, $(6,2)$.}
 \label{fig:lambda25}
  \end{center}
 \end{figure}

 From Figure~\ref{fig:lambda25}, we see that for  $h \leq 16$  the curves corresponding to the pairs $(4,3)$, $(5,1)$ do not intersect the curves corresponding to the other pairs.  We also note that the curves corresponding to $(4,4)$
 and $(6,0)$ give rise to the same Dirichlet eigenvalue as $h \to \infty$ (see the table in Appendix~\ref{sA}).
 There are only two crossings in Figure~\ref{fig:lambda25}.

By Proposition~\ref{p:crossing}, there exists a unique value $h = h_{25}^*$ at which the crossing occurs.
So $\lambda_{25,h}$ is given by the pair $(4,3)$ for $h \leq h_{25}^*$ and by the pair $(5,1)$ for $h > h_{25}^*$.

Hence, we have obtained:
\begin{prop}
There exists $h_{25}^* >0$ such that  $\lambda_{25,h}$ is given by the pair $(4,3)$
for $h \leq h_{25}^*$ and by the pair $(5,1)$ for $h > h_{25}^*$.
\end{prop}

 By the bisection method, we compute $h_{25}^*$ numerically and find that \break $h_{25}^* \sim 3.1317$.

We note that $u_{4,3}(x,y)$ is antisymmetric with respect to the transformation $ (x,y) \mapsto (-x,-y)$,
 while $u_{5,1}(x,y)$ is symmetric with respect to this transformation. From the first observation,	Lemma \ref{lem4.1} gives that any  eigenvalue with corresponding eigenfunction a linear combination of $u_{4,3}(x,y)$ and
 $u_{3,4}(x,y)$ has an even number of nodal domains.
 Hence $\lambda_{25,h}$ is not Courant-sharp for $h < h_{25}^*$
 and $\lambda_{27,h}$ is not Courant-sharp for $h > h_{25}^*$.

For $h_{25}^* < h < \infty$ we investigate whether $\lambda_{25,h}(S)$ is Courant-sharp or not
by considering the corresponding eigenfunctions $u_{5,1}(x,y)$ and $u_{1,5}(x,y)$.\\

  For $(x,y) \in (-\frac{\pi}{2},\frac{\pi}{2})^2$, we consider the function
 \begin{align}
 \Phi_{h,\theta,5,1}(x,y)&=\cos\theta \, \sin(\alpha_5(h)x/\pi)\sin(\alpha_1(h)y/\pi) \notag\\
& \ \ \ \ \ \  + \sin \theta \, \sin(\alpha_1(h)x/\pi) \sin(\alpha_5(h)y/\pi)\,.
 \end{align}
\noindent
We also note that the lines $x=0$ and $y=0$ belong to the nodal set of $\Phi_{h,\theta, 5,1}$ for any $\theta$.

 It is known that $\lambda_{25,\infty}(S)$ is not Courant-sharp, \cite{BH1}.
 In addition, by Theorem~\ref{thm:hlarge}, we know that, for $h$ sufficiently large, $\lambda_{25,h}(S)$ is not Courant-sharp.

 In Figure~\ref{fig:25DirCrit}, we plot the twenty-fifth Dirichlet eigenfunction
 $$\Phi_{+\infty,\theta,5,1}(x,y) = \cos \theta \sin(6x)\sin(2y) + \sin \theta \sin(2x)\sin(6y)$$
for $\theta = 0$, $\arctan(1/3)$, $\frac{\pi}{4}$, $\arctan(3)$, $\frac{\pi}{2}$, $\frac{5\pi}{8}$, $\frac{3\pi}{4}$, $\frac{13\pi}{16}$, $\frac{15\pi}{16}$.\\

 \begin{figure}[!h]
 \begin{center}
\includegraphics[width=10cm]{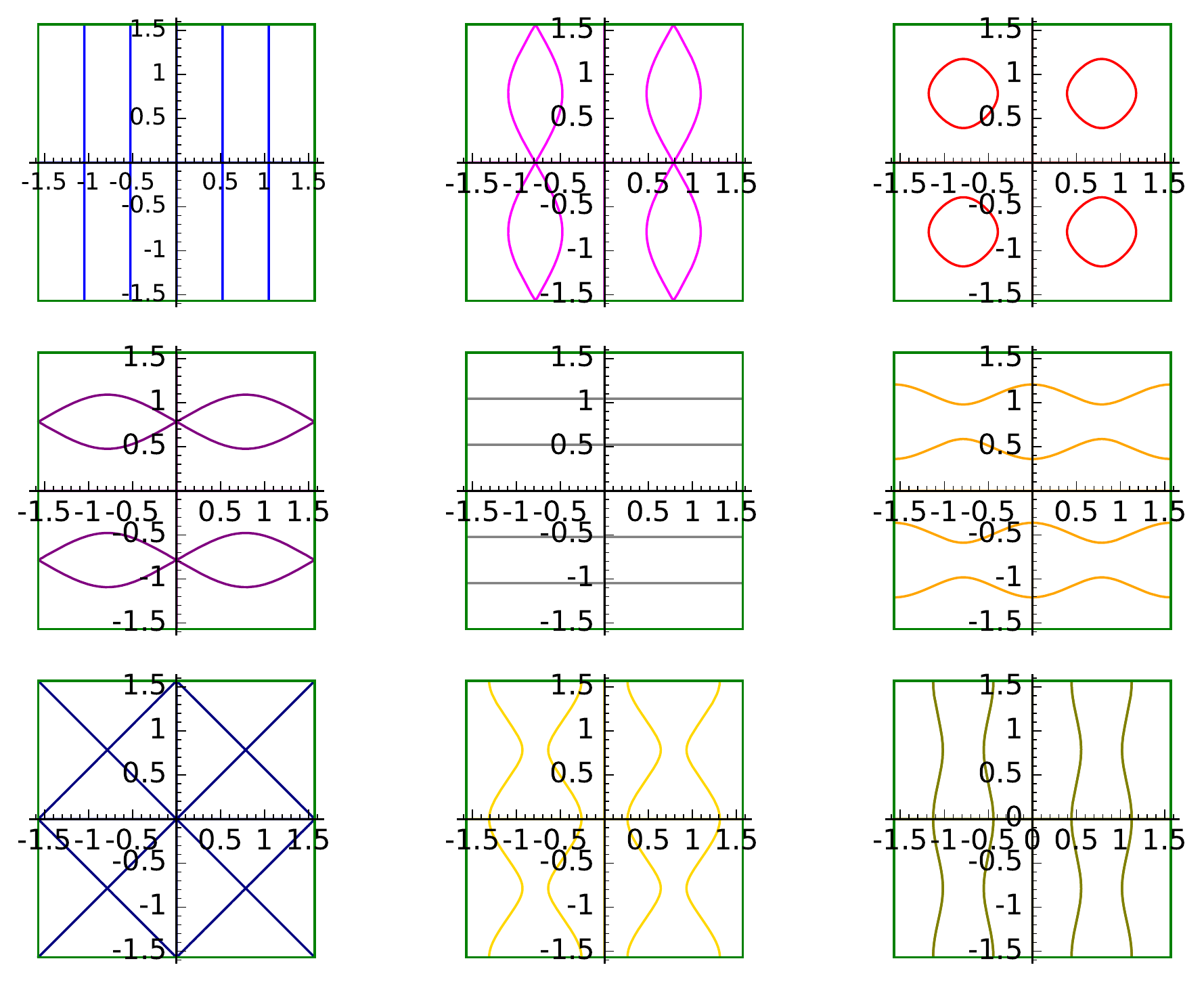}
 \caption{The  twenty-fifth Dirichlet eigenfunction $\Phi_{+\infty,\theta,5,1}$ for  $\theta = 0$, $\arctan(1/3)$, $\frac{\pi}{4}$, $\arctan(3)$, $\frac{\pi}{2}$, $\frac{5\pi}{8}$, $\frac{3\pi}{4}$, $\frac{13\pi}{16}$, $\frac{15\pi}{16}$ (blue, magenta, red, purple, grey orange, navy, gold, olive respectively).}
 \label{fig:25DirCrit}
  \end{center}
 \end{figure}

 We observe that for $\theta= \arctan(1/3)$ there are double points at $(0,0)$, $(\pm\frac{\pi}4, \frac{\pi}2)$,
 and $(\pm\frac{\pi}4,-\frac{\pi}2)$. There are triple points at $(\pm\frac{\pi}{4},0)$.
 Similarly for $\theta=\arctan(3)$, there are double points at $(0,0)$, $(\pm\frac{\pi}2, \frac{\pi}4)$,
 and $(\pm\frac{\pi}2,-\frac{\pi}4)$, and triple points at $(0,\pm\frac{\pi}4)$.

 The eigenfunction associated with the fifth Dirichlet eigenvalue on $(-\frac{\pi}2,\frac{\pi}2)^2$ is
 $ \tilde{\Phi}_{+\infty,\theta,2,0}(x,y) = \cos \theta \sin(3x)\sin(y) + \sin \theta \sin(x)\sin(3y)$.
 We see that
 $$ \tilde{\Phi}_{+\infty,\theta,5,1}(x,y) := \Phi_{+\infty,\theta,5,1}\left(\frac{x}2 -\frac{\pi}4,\frac{y}2 -\frac{\pi}4\right) = {\Phi}_{+\infty,\theta,2,0}(x,y).$$
 As in the proof of Lemma 4.2 of \cite{HPS1}, $ \Phi_{+\infty,\theta,5,1} $ can be constructed by taking
 its values in the square $(0,\frac{\pi}2)^2$ and folding evenly over $(-\frac{\pi}2,\frac{\pi}2)^2$,
 that is with respect to the axes $x=0$ and $y=0$. Compare Figure~\ref{fig:hDircom} with Figure~\ref{fig:25DirCrit}.

We now consider the case where $h>0$. In order to make a numerical comparison to the Dirichlet case, we choose $h=20$ in what follows. This value of $h$ is small enough that we see some differences compared to the Dirichlet case and large enough that we keep the asymptotic structure.\\

To determine the critical points on the side $y=\frac{\pi}{2}$, consider the function
\begin{equation*}
\psi(x,\theta):= \tilde{\Phi}_{h,\theta,5,1}(x,\pi /2) = \cos \theta \sin \left(\frac{\alpha_5 x}{\pi}\right) \sin \left(\frac{\alpha_1}{2}\right)
+ \sin \theta \sin \left(\frac{\alpha_1 x}{\pi}\right) \sin \left(\frac{\alpha_5}{2}\right).
\end{equation*}
\noindent
We have that $\psi(x,\theta) = 0$ gives
\begin{equation}\label{ec1}
\tan \theta = -\frac{\sin \left(\frac{\alpha_5 x}{\pi}\right)\sin \left(\frac{\alpha_1}{2}\right)}
{\sin \left(\frac{\alpha_1 x}{\pi}\right) \sin \left(\frac{\alpha_5}{2}\right)}.
\end{equation}
\noindent
In addition, $\frac{\partial \psi}{\partial x}(x,\theta) = 0$ gives
\begin{equation}\label{ec2}
\tan \theta = - \frac{\alpha_5 \cos \left(\frac{\alpha_5 x}{\pi}\right)\sin \left(\frac{\alpha_1}{2}\right)}
{\alpha_1 \cos \left(\frac{\alpha_1 x}{\pi}\right) \sin \left(\frac{\alpha_5}{2}\right)}.
\end{equation}
\noindent
Equating \eqref{ec1} and \eqref{ec2} gives that
\begin{equation}\label{ec3}
\alpha_5 \cot \left(\frac{\alpha_5 x}{\pi}\right) = \alpha_1 \cot \left(\frac{\alpha_1 x}{\pi}\right).
\end{equation}
\noindent
Let $x_c(h)$ denote a solution of \eqref{ec3}. For $h=20$, we compute numerically that
$x_c(20) \approx 0.8096522$.
Define
\begin{equation}\label{thetam}
\theta_m(h) := \arctan\left(- \frac{\sin(\frac{\alpha_5 x_c}{\pi}) \sin(\frac{\alpha_1}{2})}
{\sin(\frac{\alpha_1 x_c}{\pi}) \sin(\frac{\alpha_5}{2})}\right).
\end{equation}
For $h=20$, we compute numerically that $\theta_m(20) \approx 0.3324691$.\\

To determine the critical points on $x=0$, consider the function
\begin{align*}
\varphi(x, y, \theta) := \tilde{\Phi}_{h,\theta,5,1}(x,y) &= \cos \theta \sin \left(\frac{\alpha_5 x}{\pi}\right) \sin \left(\frac{\alpha_1 y}{\pi}\right)\\
& \ \ \ \ \ \ \ \ + \sin \theta \sin \left(\frac{\alpha_1 x}{\pi}\right) \sin \left(\frac{\alpha_5 y}{\pi}\right).
\end{align*}
\noindent
Then $\frac{\partial \varphi}{\partial x}(0,y,\theta) = 0$ gives
\begin{equation}\label{ec4a}
\tan \theta = - \frac{\alpha_5 \sin \left(\frac{\alpha_1 y}{\pi}\right)}
{\alpha_1 \sin \left(\frac{\alpha_5 y}{\pi}\right)}.
\end{equation}
We note that $\varphi(x,y,\theta) = 0$ if and only if
\begin{equation*}
\tan \theta = - \frac{\sin \left(\frac{\alpha_5 x}{\pi}\right) \sin \left(\frac{\alpha_1 y}{\pi}\right)}
{\sin\left(\frac{\alpha_1 x}{\pi}\right) \sin \left(\frac{\alpha_5 y}{\pi}\right)}.
\end{equation*}
We also note that
\begin{equation*}
\lim_{x \rightarrow 0} - \frac{\sin \left(\frac{\alpha_5 x}{\pi}\right) \sin \left(\frac{\alpha_1 y}{\pi}\right)}
{\sin\left(\frac{\alpha_1 x}{\pi}\right) \sin \left(\frac{\alpha_5 y}{\pi}\right)}
= - \frac{\alpha_5 \sin \left(\frac{\alpha_1 y}{\pi}\right)}
{\alpha_1 \sin \left(\frac{\alpha_5 y}{\pi}\right)}
\end{equation*}
by l'H\^opital's rule.

In addition, $\frac{\partial^2 \varphi}{\partial y \partial x}(0,y,\theta) = 0$ gives
\begin{equation}\label{ec5a}
\tan \theta = - \frac{\cos\left(\frac{\alpha_1 y}{\pi}\right)}{\cos\left(\frac{\alpha_5 y}{\pi}\right)}.
\end{equation}
Equating \eqref{ec4a} and \eqref{ec5a} gives 
 equation \eqref{ec3}.
Define
\begin{equation}\label{thetat}
\theta_t := \arctan\left( - \frac{\alpha_5 \sin(\frac{\alpha_1 x_c}{\pi})}
{\alpha_1 \sin(\frac{\alpha_5 x_c}{\pi})}\right).
\end{equation}
For $h=20$, we compute numerically that $\theta_t \approx 1.2492655$.\\

Using \eqref{eq:alphanodd}, we obtain the following asymptotic expansions for $\alpha_1(h)$
and $\alpha_5(h)$ when $h \rightarrow \infty$.
\begin{align}
\alpha_1(h) &= 2\pi - \frac{4}{h} + \mathcal{O}\left(\frac{1}{h^2}\right), \label{alpha1asymp}\\
\alpha_5(h) &= 6\pi - \frac{12}{h} + \mathcal{O}\left(\frac{1}{h^2}\right). \label{alpha5asymp}
\end{align}
 Substituting these expansions into \eqref{ec3} and solving for $x_c(h)$ gives that
\begin{equation}\label{ec6}
x_c(h) = \frac \pi 4 + \frac {1}{2 h}  + \mathcal{O}\left(\frac{1}{h^2}\right).
\end{equation}

Using the above asymptotic expansions for $\alpha_1(h)$, $\alpha_5(h)$ and $x_c(h)$,
we obtain that as $h \rightarrow \infty$,
\begin{equation*}
\tan \theta_m(h) = \frac 13 +\mathcal{O}\left(\frac{1}{h}\right),
\end{equation*}
\begin{equation*}
\tan \theta_t(h) = 3  + \mathcal{O}\left(\frac{1}{h}\right),
\end{equation*}
and
\begin{equation*}
\tan \left(\frac \pi 2 -\theta_t(h)\right)=\cot \theta_t(h) = \frac 13  +  \mathcal{O}\left(\frac{1}{h}\right).
\end{equation*}
For $h=+\infty$, we have
$$\theta_m (+\infty)= \frac \pi 2 - \theta_t (+\infty)= \arctan\left(\frac13\right)\,,$$
and we deduce from above  that
\begin{equation*}
\theta_m(h) =  \arctan\left(\frac13\right)   + \mathcal{O}\left(\frac{1}{h}\right),
\end{equation*}
and
\begin{equation*}
\frac \pi 2 -\theta_t(h) =  \arctan\left(\frac13\right)  +  \mathcal{O}\left(\frac{1}{h}\right).
\end{equation*}
At this stage we do not get  any information  about the sign of $\theta_m + \theta_t - \frac \pi 2$.
For this, we observe that by \eqref{thetam} and \eqref{thetat},
\begin{equation*}
\frac{\tan \theta_m}{\tan (\frac{\pi}{2} -\theta_t)} = \tan \theta_m \tan \theta_t
=\frac{\alpha_5 \sin\left(\frac{\alpha_1}{2}\right)}{\alpha_1 \sin\left(\frac{\alpha_5}{2}\right)}.
\end{equation*}
From \eqref{eq:alphanodd}, we have that for $n$ odd,
\begin{equation*}
\frac{\alpha_n}{\sin\left(\frac{\alpha_n}{2}\right)} = -\frac{h\pi}{\cos\left(\frac{\alpha_n}{2}\right)}.
\end{equation*}
Hence we obtain
\begin{equation*}
\frac{\tan \theta_m}{\tan (\frac{\pi}{2} -\theta_t)} =
\frac{\cos\left(\frac{\alpha_1}{2}\right)}{\cos\left(\frac{\alpha_5}{2}\right)}.
\end{equation*}
From the asymptotic expansions \eqref{alpha1asymp} and \eqref{alpha5asymp}, we obtain that as $h \rightarrow \infty$,
\begin{align*}
\cos\left(\frac{\alpha_1}{2}\right) &\sim -\left(1 - \frac{2}{h^2} + \mathcal{O}\left(\frac{1}{h^3}\right)\right),\\
\cos\left(\frac{\alpha_5}{2}\right) &\sim -\left(1 - \frac{18}{h^2} + \mathcal{O}\left(\frac{1}{h^3}\right)\right).
\end{align*}
We deduce that as $h \rightarrow \infty$,
\begin{equation}\label{above}
\frac{\tan \theta_m}{\tan (\frac{\pi}{2} -\theta_t)} = 1 + \frac{16}{h^2} + \mathcal{O}\left(\frac{1}{h^3}\right).
\end{equation}
Let $\delta \theta =   \theta_m + \theta_t-\frac \pi 2$.
Then, observing that $\theta_m- \arctan \frac 13 =\mathcal O (\frac 1h)$ and $\frac \pi 2 - \theta_t -\arctan \frac 13=\mathcal O (\frac 1h)$, we get
$$
\tan \theta_m = \tan \left(\frac \pi 2 - \theta_t \right) + \delta \theta  \frac{10}{9} \left(1+\mathcal O \left(\frac 1h\right)\right).
$$
Dividing by $\tan (\frac \pi 2 - \theta_t )$ leads to
$$
\frac{\tan \theta_m}{\tan (\frac{\pi}{2} -\theta_t)} = 1 + \delta \theta \, \frac{10}{3} \left(1+\mathcal O \left(\frac 1h\right)\right)\,.
$$
From \eqref{above}, we  then get
$$
\delta \theta \, \frac{10}{3} \left(1+\mathcal O \left(\frac 1h\right)\right) = \frac{16}{h^2}\,.
$$
From this we deduce that
$$
\delta \theta=\theta_m(h) + \theta_t (h) - \frac \pi 2 \sim \frac{24}{5h^2}\,.
$$
This gives the strict positivity  of $\theta_m(h) + \theta_t (h) - \frac \pi 2$  for $h$ large enough, a property which is numerically satisfied for $h=20$.
 In Figure~\ref{fig:g20500}, we plot $h \mapsto g(h):=\frac{\alpha_5 \sin\left(\frac{\alpha_1}{2}\right)}{\alpha_1 \sin\left(\frac{\alpha_5}{2}\right)}$ for $20 \leq h \leq 500$ and note that it approaches 1 from above.
 \begin{figure}[!ht]
 \begin{center}
 \includegraphics[width=0.8\textwidth]{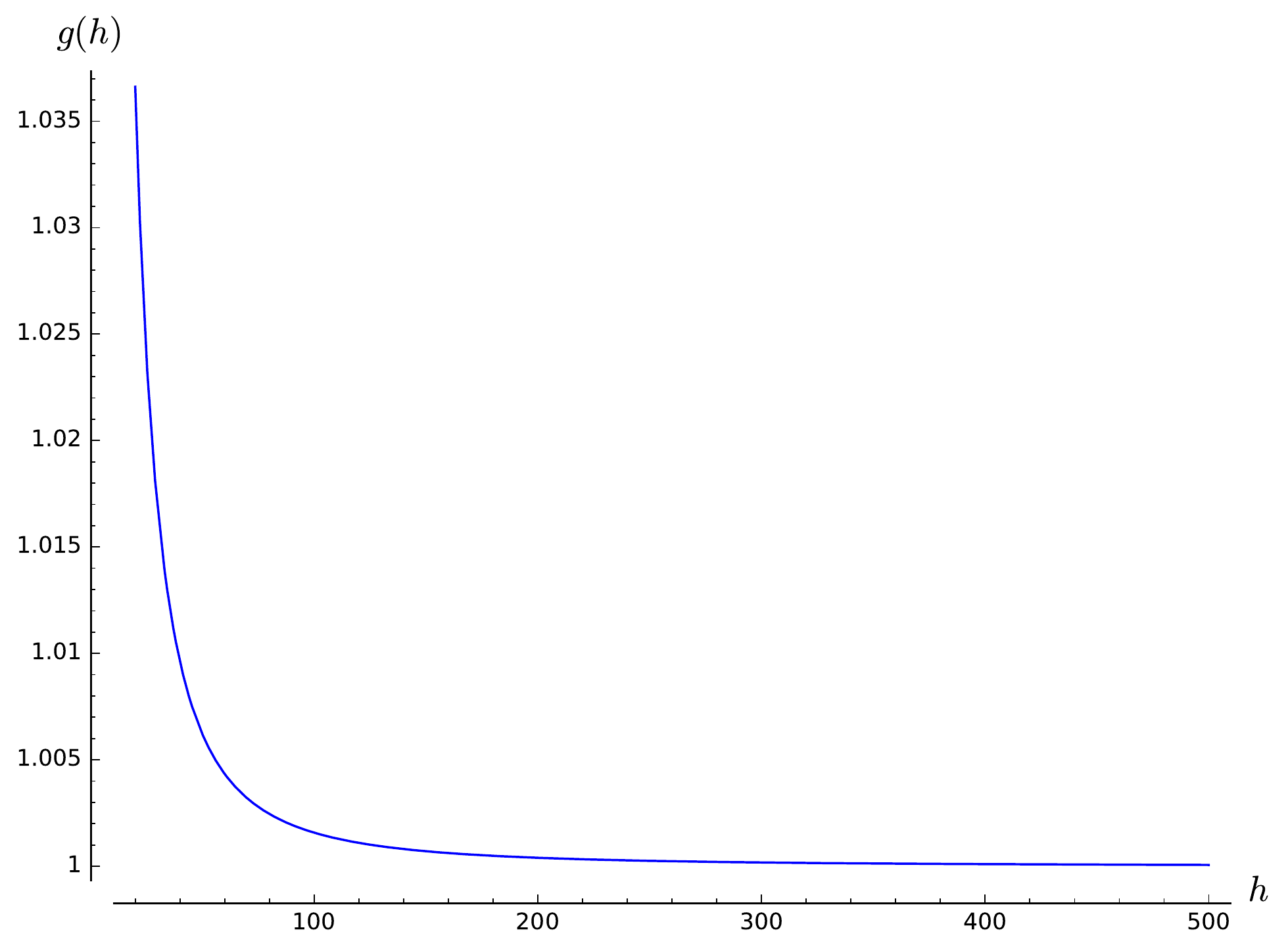}
 \caption{The graph of $h \mapsto g(h):=\frac{\alpha_5 \sin\left(\frac{\alpha_1}{2}\right)}{\alpha_1 \sin\left(\frac{\alpha_5}{2}\right)}$ for $20 \leq h \leq 500$.}
 \label{fig:g20500}
  \end{center}
 \end{figure}

In Figure~\ref{fig:25crith20}, we plot $ \Phi_{h,\theta,5,1}$ for  $h=20$ and $\theta = 0$, $\frac{\pi}{2} - \theta_t$, $\theta=\frac12 (\frac{\pi}{2} - \theta_t + \theta_m)$,
$\theta_m$, $\frac{\pi}{4}$, $\frac{\pi}{2}-\theta_m$, $\theta=\frac12 (\frac{\pi}{2} - \theta_m + \theta_t)$,
$\theta_t$, $\frac{\pi}{2}$, $\frac{5\pi}{8}$, $\frac{3\pi}{4}$, $\frac{13\pi}{16}$.
From this figure, we make the following observations.

 \begin{figure}
\centering
\subfloat[ $\theta=0$ (blue), $\theta=\frac{\pi}{2} - \theta_t$ (orange), $\theta=\frac12 (\frac{\pi}{2} - \theta_t + \theta_m)$ (maroon).\label{25crit1}]{%
  \includegraphics[width=\textwidth]{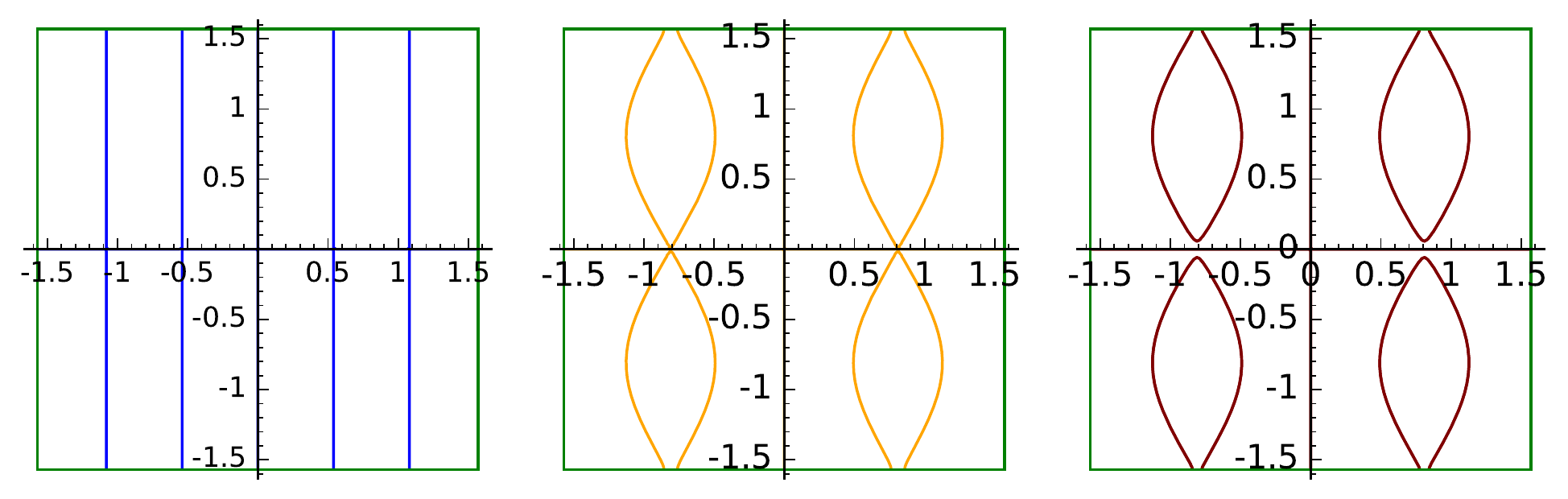}%
  }\par
\subfloat[ $\theta=\theta_m$ (magenta), $\theta=\frac{\pi}{4}$ (red), $\theta=\frac{\pi}2 - \theta_m$ (purple). \label{25crit2}]{%
  \includegraphics[width=\textwidth]{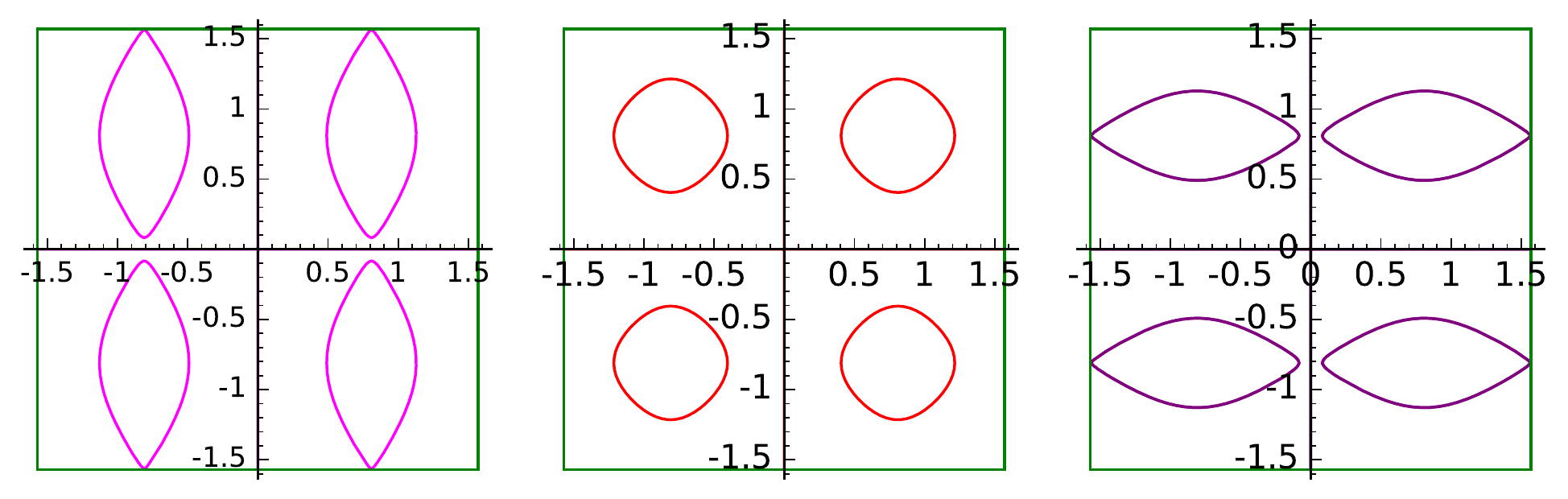}%
  }\par
\subfloat[ $\theta=\frac12 (\frac{\pi}{2} - \theta_m + \theta_t)$ (deep sky blue), $\theta=\theta_t$ (teal), $\theta=\frac{\pi}2$ (grey). \label{25crit3}]{%
  \includegraphics[width=\textwidth]{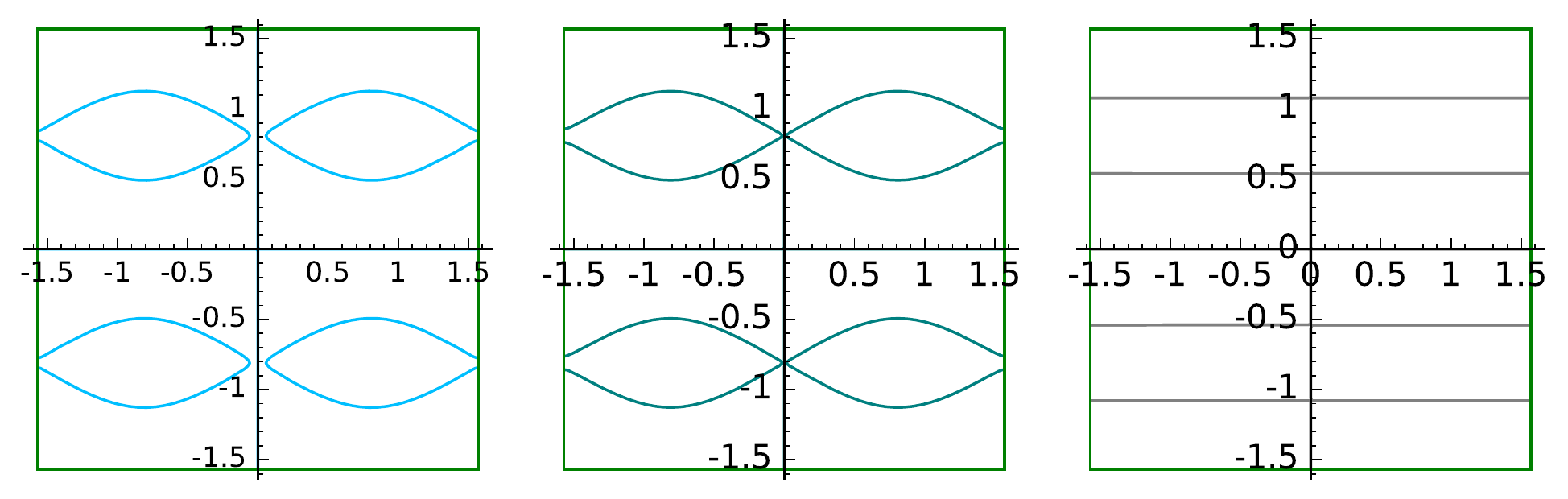}%
  }\par
\subfloat[ $\frac{5\pi}{8}$ (lime), $\theta=\frac{3\pi}{4}$ (navy), $\theta=\frac{13\pi}{16}$ (gold). \label{25crit4}]{%
  \includegraphics[width=\textwidth]{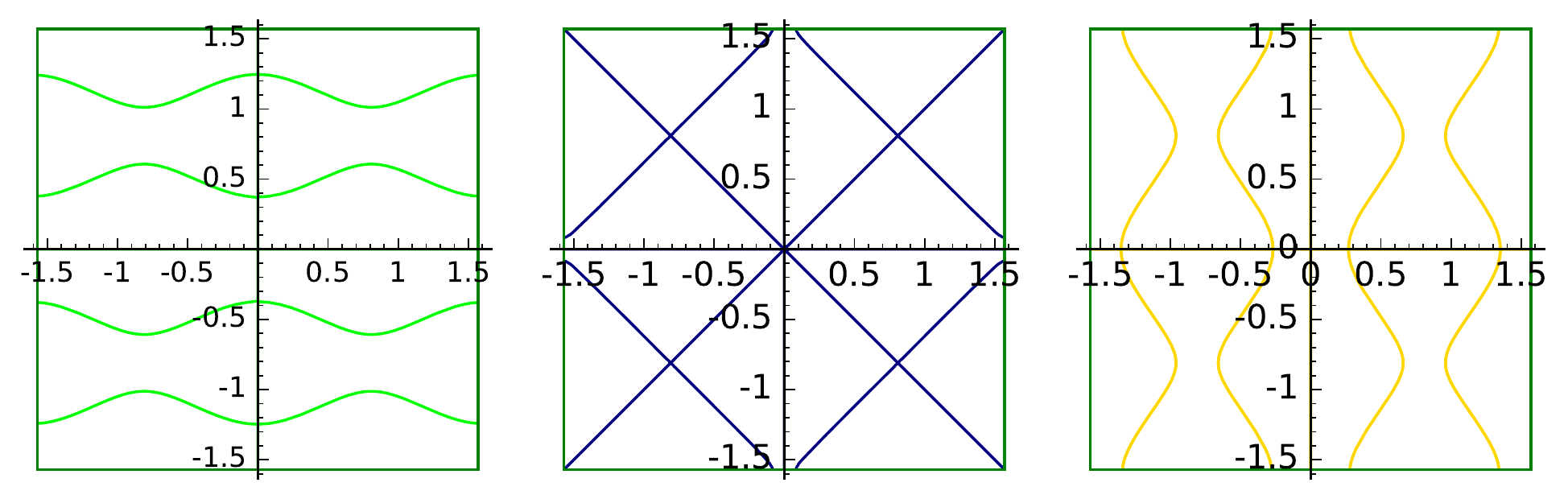}%
  }
\caption{The Robin eigenfunction $ \Phi_{h,\theta,5,1}$ for $h=20$ and various values of $\theta$.}
\label{fig:25crith20}
\end{figure}

For $0\leq\theta< \frac{\pi}{2} - \theta_t$, there are 12 boundary critical points, 5 interior critical points
and 12 nodal domains.

For $\theta=\frac{\pi}{2} - \theta_t$, there are 12 boundary critical points, 3 interior critical points
and 12 nodal domains.

For $\frac{\pi}{2} - \theta_t < \theta < \theta_m$, there are 12 boundary critical points, 1 interior critical point
and 8 nodal domains  (see, for example, Part (a) of Figure~\ref{fig:25crith20} (maroon)).

For $\theta=\theta_m$, there are 8 boundary critical points, 1 interior critical point and 8 nodal domains.

For $\theta_m < \theta < \frac{\pi}{2} - \theta_m$, there are 4 boundary critical points, 1 interior critical point
and 8 nodal domains.

For $\theta = \frac{\pi}{2} - \theta_m$, there are 8 boundary critical points, 1 interior critical point
and 8 nodal domains.

For $\frac{\pi}{2} - \theta_m < \theta < \theta_t$, there are 12 boundary critical points, 1 interior critical points
and 8 nodal domains (see, for example, Part (c) of Figure~\ref{fig:25crith20} (deep sky blue)).

For $\theta = \theta_t$, there are 12 boundary critical points, 3 interior critical points and 12 nodal domains.

For $\theta_t < \theta < \frac{3\pi}{4}$, there are 12 boundary critical points, 5 interior critical points
and 12 nodal domains.

For $\theta = \frac{3\pi}{4}$, there are 16 boundary critical points, 5 interior critical points
and 16 nodal domains.

For $\frac{3\pi}{4}< \theta < \pi$, there are 12 boundary critical points, 5 interior critical points
and 12 nodal domains.

 By comparing Figure~\ref{fig:25crith20} with Figure~\ref{fig:h20com} below, we see that for $h=20$ and $\theta=\frac{3\pi}4$, the nodal structure for the twenty-fifth Robin eigenfunction is not obtained from the nodal structure of the fifth Robin eigenfunction (by the aforementioned folding procedure on the square $ (-\frac{\pi}{2},\frac{\pi}{2})^2$
 which holds for the Dirichlet case).  For example, this can be seen by comparing the navy curve in
 Part (d) of Figure~\ref{fig:25crith20} with the  navy curve in Figure~\ref{fig:h20com}.

 \begin{figure}[!ht]
 \begin{center}
 \includegraphics[width=8cm]{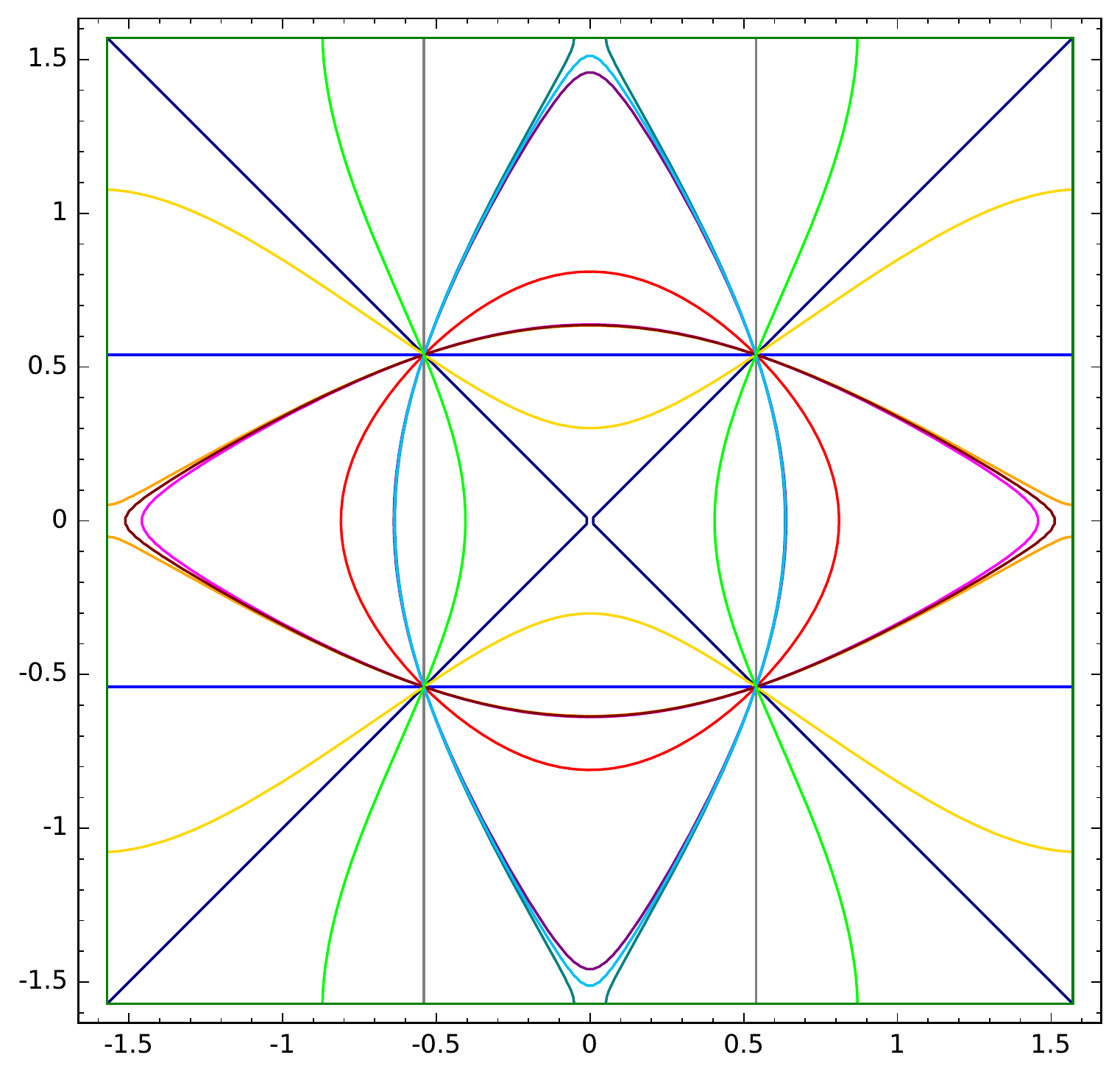}
 \caption{The fifth Robin eigenfunction $ \Phi_{h,\theta,0,2}$ for $h=20$ and various values of $\theta$.
 The values of $\theta$ are $\theta=0$ (blue), $\theta=\frac{\pi}{2} - \theta_t$ (orange),
 $\theta=\frac12 (\frac{\pi}{2} - \theta_t + \theta_m)$ (maroon), $\theta=\theta_m$ (magenta),
 $\theta=\frac{\pi}{4}$ (red), $\theta=\frac{\pi}2 - \theta_m$ (purple), $\theta=\frac12 (\frac{\pi}{2} - \theta_m + \theta_t)$ (deep sky blue), $\theta=\theta_t$ (teal),
 $\theta=\frac{\pi}2$ (grey), $\frac{5\pi}{8}$ (lime), $\theta=\frac{3\pi}{4}$ (navy), $\theta=\frac{13\pi}{16}$ (gold).}
 \label{fig:h20com}
  \end{center}
 \end{figure}

 The numerical experiment discussed above suggests that there are no new transitions and no new critical points appear as $h$ increases from $h=20$ to $h=+\infty$ (see, for example, Figure~\ref{fig:25DirCrit}).\\

We remark that in general, for $h=h_{25}^*$ any eigenfunction corresponding to $\lambda_{25,h}$ is a linear combination of $u_{4,3}(x,y)$, $u_{3,4}(x,y)$, $u_{5,1}(x,y)$ and $u_{1,5}(x,y)$. Such an eigenfunction  might not have any common symmetries.  We note that we have not shown that $\lambda_{25,h_{25}^*(S)}$ is not Courant-sharp.

\subsection{Multiple crossings: analysis of examples.}\label{ss:lambda84}

Although Proposition~\ref{p:crossing} asserts that the curves corresponding to two distinct pairs can cross
at most once, it is possible that an eigenvalue $\lambda_{n,h}(S)$ is given by more than two distinct curves
as $h$ varies.

 The situation for the eigenvalues $\lambda_{84,h}, \dots, \lambda_{92,h}$ seems to be quite complicated.
We claim that these eigenvalues are given by the curves corresponding to the pairs $(9,4)$, $(7,7)$, $(10,0)$, $(8,6)$, $(10,1)$.
We first show that none of the curves corresponding to other pairs intersect these ones.

By considering Appendix~\ref{sA} and monotonicity of the Robin eigenvalues with respect to $h$, we have that all curves corresponding to pairs $(p,q)$ with $\pi^{-2}(\alpha_{p}(h)^2 + \alpha_{q}(h)^2) \geq 130$ for all $h \geq 0$ do not  intersect the curves corresponding to the pairs $(9,4), (7,7), (10,0), (8,6), (10,1)$. That is $(11,3), (9,7)$ and so on.
From above, we must consider the curves corresponding to $(10,2)$, $(9,5)$, $(10,3)$, $(8,7)$, $(10,4)$, $(9,6)$,
$(11,0)$, $(11,1)$, $(11,2)$, $(10,5)$, $(8,8)$.

Note that if $q \leq r$, then $\alpha_{p}(h)^2 + \alpha_{q}(h)^2 \leq \alpha_{p}(h)^2 + \alpha_{r}(h)^2$ for all $h \geq 0$.
So it suffices to show that the curves corresponding to $(9,5), (8,7), (11,0)$ do not intersect those
corresponding to $(9,4)$, $(7,7)$, $(10,0)$, $(8,6)$, $(10,1)$. Numerically, we compute that
\begin{align*}
 \lambda_{9,5,h}(S)  &= \lambda_{5,9,h}(S) \geq 130 \text{ for } h > 26.9531, \\
 \lambda_{8,7,h}(S)  &= \lambda_{7,8,h}(S) \geq 130 \text{ for } h > 9.3456, \\
 \lambda_{11,0,h}(S) &= \lambda_{0,11,h}(S) \geq 130 \text{ for } h > 7.3264.
 \end{align*}
So we need to consider $h \leq 27$.
We note that the curves corresponding to $(8,6)$ and $(10,2)$ give rise to the same value at $h=+\infty$.

Again by Appendix~\ref{sA} and monotonicity of the Robin eigenvalues with respect to $h$, we have that all pairs $(p,q)$ with $\pi^{-2}(\alpha_{p}(h)^2 + \alpha_{q}(h)^2) \leq 97$ for all $h \geq 0$ do not intersect the curves corresponding to the pairs $(9,4)$, $(7,7)$, $(10,0)$, $(8,6)$, $(10,1)$. That is $(8,3), (8,2)$ and below.
Hence, from below, we must consider the curves corresponding to $(6,6)$, $(7,5)$, $(8,4)$, $(9,0)$, $(9,1)$, $(9,2)$,
$(7,6)$, $(8,5)$, $(9,3)$. Similarly to the above, it suffices to consider $(7,6)$, $(8,5)$, $(9,3)$.
We compute numerically that
\begin{align*}
 \lambda_{9,4,h}(S)  & =\lambda_{4,9,h}(S) \geq 117 \text{ for } h > 17.5353, \\
 \lambda_{7,7,h}(S)  & \geq 117 \text{ for } h > 12.4168, \\
 \lambda_{10,0,h}(S) &= \lambda_{0,10,h}(S) \geq 117 \text{ for } h > 28.8245, \\
 \lambda_{8,6,h}(S) &= \lambda_{6,8,h}(S) \geq 117 \text{ for } h > 9.9784, \\
 \lambda_{10,1,h}(S) &= \lambda_{1,10,h}(S) \geq 117 \text{ for } h > 16.9735.
 \end{align*}
So we must consider $h \leq 29$.

In Figure~\ref{fig:lambda84curves}, we plot the curves corresponding to the pairs $(7,6)$, $(8,5)$, $(9,3)$, $(9,4)$, $(7,7)$, $(10,0)$, $(8,6)$, $(10,1)$, $(10,2)$, $(9,5)$, $(8,7)$, $(11,0)$ for $h \leq 29$ and we see that the eigenvalues $\lambda_{84,h}, \dots, \lambda_{92,h}$ are indeed given by the pairs $(9,4)$, $(7,7)$, $(10,0)$, $(8,6)$, $(10,1)$.

\begin{figure}[!h]
 \begin{center}
\includegraphics[width=10cm]{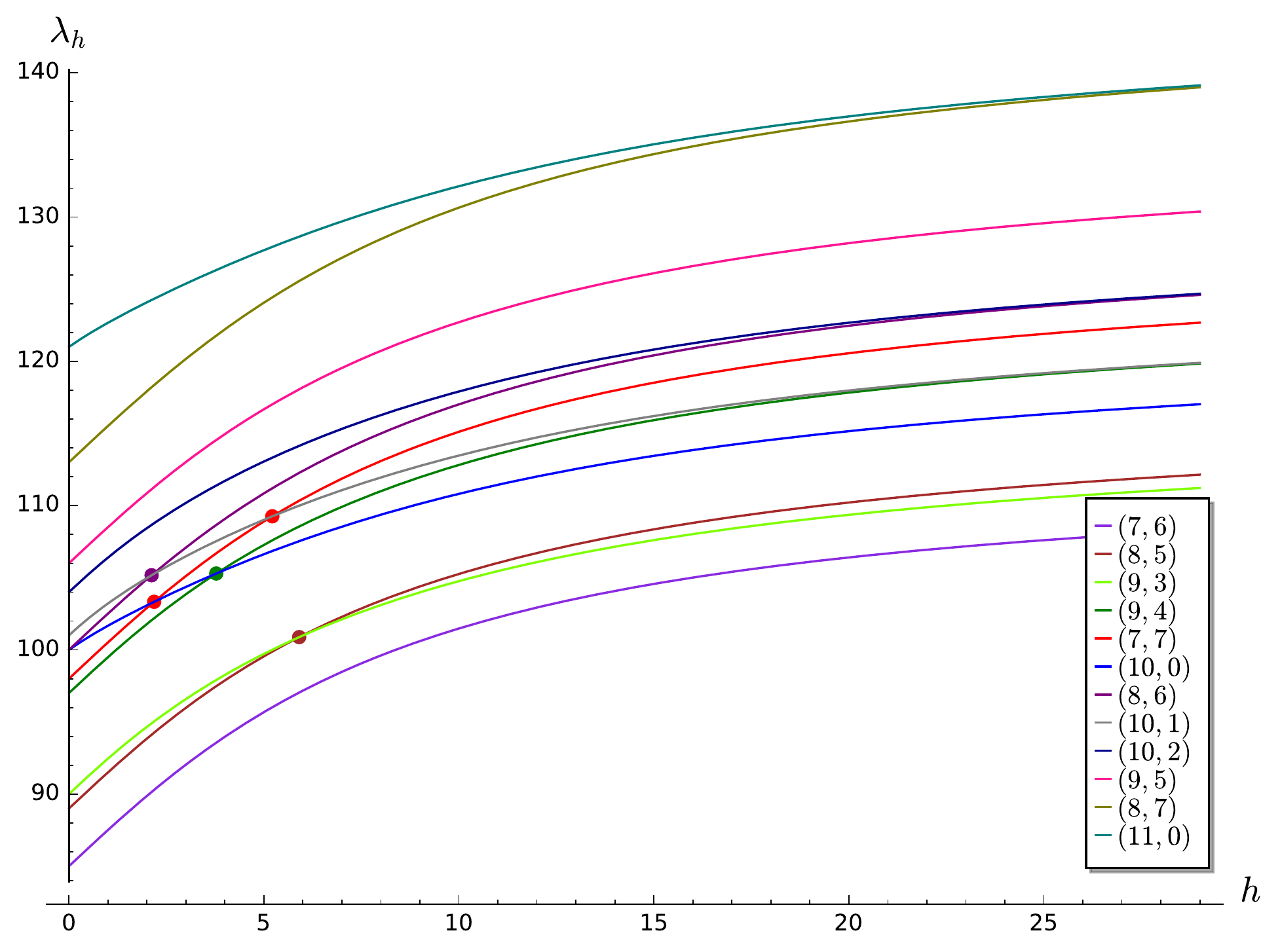}
 \caption{The curves $(\alpha_m(h)^2 + \alpha_n(h)^2)/\pi^2$ for $h\leq 29$ corresponding to the pairs $(7,6)$, $(8,5)$, $(9,3)$, $(9,4)$, $(7,7)$, $(10,0)$, $(8,6)$, $(10,1)$, $(10,2)$, $(9,5)$, $(8,7)$, $(11,0)$.}
 \label{fig:lambda84curves}
  \end{center}
 \end{figure}

By the above, we have that there exist $0< h_a < h_b < h_c < h_d < +\infty$ such that the following hold.

For $0 \leq h \leq h_a$, the curve corresponding to $(9,4)$ lies below that corresponding to $(7,7)$ which
lies below $(10,0)$, which lies below $(8,6)$ which in turn lies below $(10,1)$.
So for $0 \leq h \leq h_a$, $\lambda_{84,h} = \lambda_{85,h}$ is given by $(9,4)$, $\lambda_{86,h}$ by $(7,7)$,
$\lambda_{87,h}=\lambda_{88,h}$ by $(10,0)$, $\lambda_{89,h}=\lambda_{90,h}$ by $(8,6)$ and $\lambda_{91,h}=\lambda_{92,h}$ by $(10,1)$.

Similarly for $h_a \leq h \leq h_b$, $\lambda_{84,h} = \lambda_{85,h}$ is given by $(9,4)$, $\lambda_{86,h}$ by $(7,7)$, $\lambda_{87,h}=\lambda_{88,h}$ by $(10,0)$, $\lambda_{89,h}=\lambda_{90,h}$ by $(10,1)$ and $\lambda_{91,h}=\lambda_{92,h}$ by $(8,6)$.

For $h_b \leq h \leq h_c$, $\lambda_{84,h} = \lambda_{85,h}$ is given by $(9,4)$, $\lambda_{86,h}=\lambda_{87,h}$ by $(10,0)$, $\lambda_{88,h}$ by $(7,7)$, $\lambda_{89,h}=\lambda_{90,h}$ by $(10,1)$ and $\lambda_{91,h}=\lambda_{92,h}$ by $(8,6)$.

For $h_c \leq h \leq h_d$, $\lambda_{84,h} = \lambda_{85,h}$ is given by $(10,0)$, $\lambda_{86,h}=\lambda_{87,h}$ by $(9,4)$, $\lambda_{88,h}$ by $(7,7)$, $\lambda_{89,h}=\lambda_{90,h}$ by $(10,1)$ and $\lambda_{91,h}=\lambda_{92,h}$ by $(8,6)$.

For $h \geq h_d$, $\lambda_{84,h} = \lambda_{85,h}$ is given by $(10,0)$, $\lambda_{86,h}=\lambda_{87,h}$ by $(9,4)$,
$\lambda_{88,h}=\lambda_{89,h}$ by $(10,1)$, $\lambda_{90,h}$ by $(7,7)$ and $\lambda_{91,h}=\lambda_{92,h}$ by $(8,6)$.

We compute numerically that $h_a \sim 2.1209$, $h_b \sim 2.1864$, $h_c \sim 3.7786$, and $h_d \sim 5.2167$.
 In Figure~\ref{fig:lambda84}, we plot the curves corresponding to the pairs $(7,6)$, $(8,5)$, $(9,3)$, $(9,4)$, $(7,7)$, $(10,0)$, $(8,6)$, $(10,1)$, $(10,2)$, $(9,5)$, $(8,7)$, $(11,0)$ for $h \leq 10$.

\begin{figure}[!h]
 \begin{center}
\includegraphics[width=10cm]{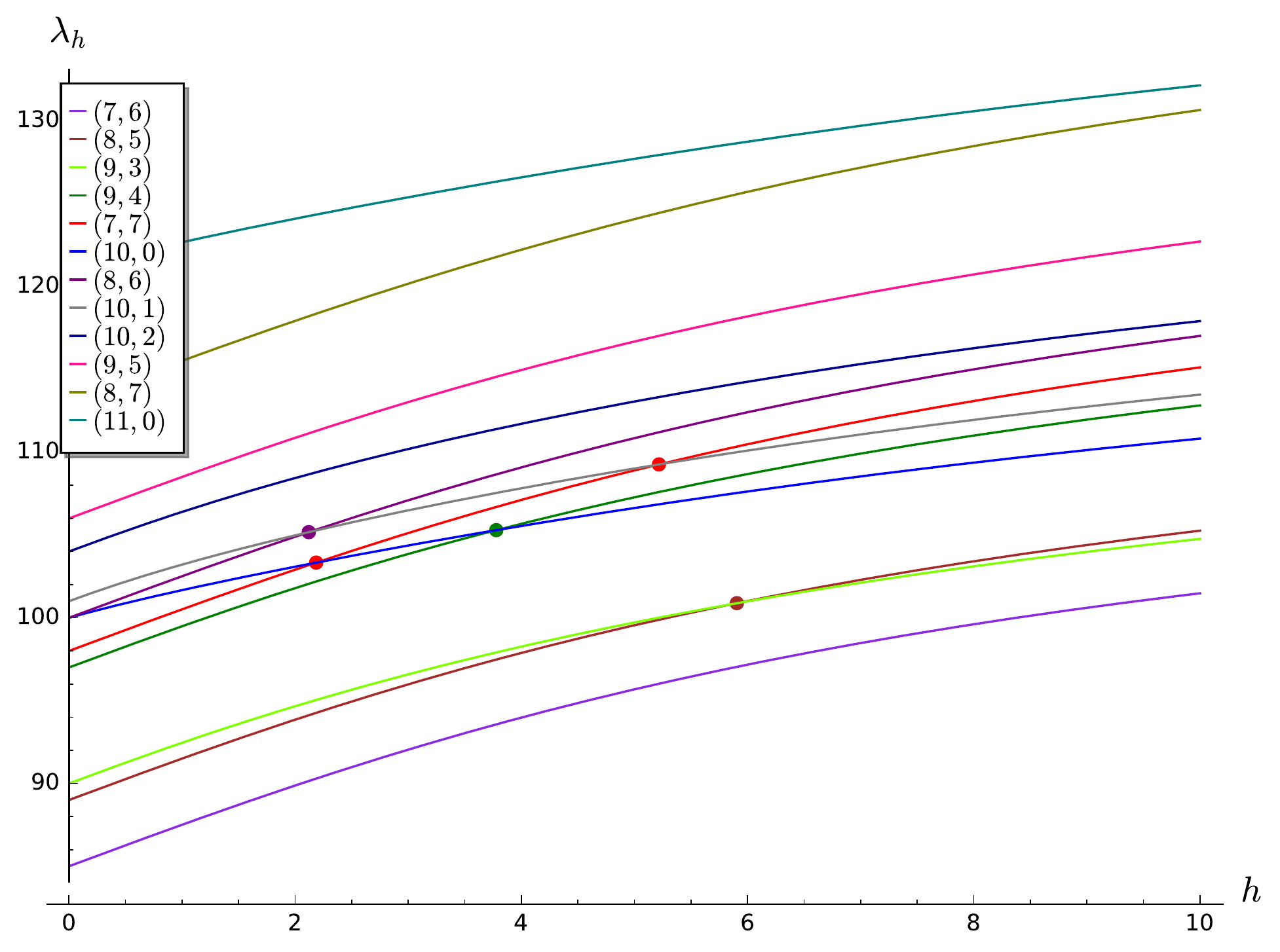}
 \caption{The Robin eigenvalues of the square $(\alpha_m(h)^2 + \alpha_n(h)^2)/\pi^2$ for $h\leq 10$ corresponding to the pairs $(7,6)$, $(8,5)$, $(9,3)$, $(9,4)$, $(7,7)$, $(10,0)$, $(8,6)$, $(10,1)$, $(10,2)$, $(9,5)$, $(8,7)$, $(11,0)$.}
 \label{fig:lambda84}
  \end{center}
 \end{figure}

We note that the curves corresponding to $(8,6)$ and $(10,0)$ give rise to the same Neumann eigenvalue when $h=0$.
In addition, the curves corresponding to $(9,4)$ and $(10,1)$ give rise to the same Dirichlet eigenvalue
at $h = +\infty$. (see Appendix~\ref{sA}).\\

We see that it is possible that the labelling could switch more than once for a given eigenvalue
(that is, the eigenvalue could be given by more than two pairs).

\appendix

\section{Comparison Dirichlet-Neumann.}\label{sA}
 In this appendix, we recall from \cite{BH1,HPS1} the Neumann eigenvalues $\lambda_k^{N}$ of $S$
(in the left-hand side below) and the Dirichlet eigenvalues $\lambda_k^{D}$ of $S$ (in the right-hand side below)
for $k \leq 129$. We also recall the pairs $(m,n)$ corresponding to these eigenvalues and the values $m^2+n^2$ of the eigenvalues. The purpose is to illustrate the values of $k \leq 129$ for which there are crossings between the curves corresponding to the Robin eigenvalues of $S$. We use colours to emphasise this. For example, the Robin eigenvalue $\lambda_{2,2,h}(S)$ starts as the ninth eigenvalue when $h=0$ but as $h \to +\infty$, it corresponds to the eleventh eigenvalue.
\\

\begin{tabular}{|c  c  c  c |}
\hline
Neumann & & & \\
\hline
$m$ & $n$ & $m^2+n^2$ & $k$\\
\hline
0 & 0 & 0 & 1\\
1 &	0 & 1 &	2,3\\
0 & 1 & 1 & 2,3\\
1 &	1 &	2 &	4\\
2 &	0 &	4 &	5,6\\
0 &	2	& 4 &	5,6\\
2	& 1 &	5 &	7,8\\
1 &	2 &	5 &	7,8\\
\rowcolor{yellow}
2 &	2 &	8 &	9\\
\rowcolor{cyan}
3 &	0 &	9 &	10,11\\
\rowcolor{cyan}
0 &	3 &	9 &	10,11\\
3 &	1 &	10 &	12,13\\
1 &	3 &	10 &	12,13\\
3 &	2 &	13 &	14,15\\
2 &	3 &	13 &	14,15\\
4 &	0 &	16 &	16,17\\
0 &	4 &	16 &	16,17\\
4 &	1 &	17 &	18,19\\
1 &	4 &	17 &	18,19\\
3 &	3 &	18 &	20\\
4 &	2 &	20 &	21,22\\
2 &	4 &	20 &	21,22\\
5 &	0 &	25 &	23,24,25,26\\
0 &	5 &	25 &	23,24,25,26\\
\rowcolor{yellow}
4 &	3 &	25 &	23,24,25,26\\
\rowcolor{yellow}
3 &	4 &	25 &	23,24,25,26\\
\rowcolor{cyan}
5 &	1 &	26 &	27,28\\
\rowcolor{cyan}
1 &	5 &	26 &	27,28\\
5 &	2 &	29 &	29,30\\
2 &	5 &	29 &	29,30\\
4 &	4 &	32 &	31\\
\rowcolor{yellow}
5 &	3 &	34 &	32,33\\
\rowcolor{yellow}
3 &	5 &	34 &	32,33\\
\rowcolor{cyan}
6 &	0 &	36 &	34,35\\
\rowcolor{cyan}
0 &	6 &	36 &	34,35\\
6 &	1 &	37 &	36,37\\
1 &	6 &	37 &	36,37\\
6 &	2 &	40 &	38,39\\
2 &	6 &	40 &	38,39\\
5 &	4 &	41 &	40,41\\
4 &	5 &	41 &	40,41\\
6 &	3 &	45 &	42,43\\
3 &	6 &	45 &	42,43\\
7 &	0 &	49 &	44,45\\
0 &	7 &	49 &	44,45\\
\hline
\end{tabular}
\quad
\begin{tabular}{|c  c  c  c |}
\hline
Dirichlet & & & \\
\hline
$m$ & $n$ & $m^2+n^2$ & $k$\\
\hline
1 & 1  & 2 & 1\\
2 &	1 &	5 &	2,3\\
1 &	2 &	5 &	2,3\\
2 &	2 &	8 &	4\\
3 &	1 &	10 &	5,6\\
1 &	3 &	10 &	5,6\\
3 &	2 &	13 &	7,8\\
2 &	3 &	13 &	7,8\\
\rowcolor{cyan}
4 &	1 &	17 &	9,10\\
\rowcolor{cyan}
1 &	4 &	17 &	9,10\\
\rowcolor{yellow}
3 &	3 &	18 &	11\\
4 &	2 &	20 &	12,13\\
2 &	4 &	20 &	12,13\\
4 &	3 &	25 &	14,15\\
3 &	4 &	25 &	14,15\\
5 &	1 &	26 &	16,17\\
1 &	5 &	26 &	16,17\\
5 &	2 &	29 &	18,19\\
2 &	5 &	29 &	18,19\\
4 &	4 &	32 &	20\\
5 &	3 &	34 &	21,22\\
3 &	5 &	34 &	21,22\\
6 &	1 &	37 &	23,24\\
1 &	6 &	37 &	23,24\\
\rowcolor{cyan}
6 &	2 &	40 &	25,26\\
\rowcolor{cyan}
2 &	6 &	40 &	25,26\\
\rowcolor{yellow}
5 &	4 &	41 &	27,28\\
\rowcolor{yellow}
4 &	5 &	41 &	27,28\\
6 &	3 &	45 &	29,30\\
3 &	6 &	45 &	29,30\\
5 &	5 &	50 &	31,32,33\\
\rowcolor{cyan}
7 &	1 &	50 &	31,32,33\\
\rowcolor{cyan}
1 &	7 &	50 &	31,32,33\\
\rowcolor{yellow}
6 &	4 &	52 &	34,35\\
\rowcolor{yellow}
4 &	6 &	52 &	34,35\\
7 &	2 &	53 &	36,37\\
2 &	7 &	53 &	36,37\\
7 &	3 &	58 &	38,39\\
3 &	7 &	58 &	38,39\\
6 &	5 &	61 &	40,41\\
5 &	6 &	61 &	40,41\\
8 &	1 &	65 &	42,43,44,45\\
7 &	4 &	65 &	42,43,44,45\\
4 &	7 &	65 &	42,43,44,45\\
1 &	8 &	65 &	42,43,44,45\\
\hline
\end{tabular}

\begin{tabular}{|c  c  c  c |} 
\hline
Neumann & & & \\
\hline
$m$ & $n$ & $m^2+n^2$ & $k$\\
\hline
7 &	1 &	50 &	46,47,48\\
5 &	5 &	50 &	46,47,48\\
1 &	7 &	50 &	46,47,48\\
\rowcolor{yellow}
6 &	4 &	52 &	49,50\\
\rowcolor{yellow}
4 &	6 &	52 &	49,50\\
\rowcolor{cyan}
7 &	2 &	53 &	51,52\\
\rowcolor{cyan}
2 &	7 &	53 &	51,52\\
7 &	3 &	58 &	53,54\\
3 &	7 &	58 &	53,54\\
\rowcolor{yellow}
6 &	5 &	61 &	55,56\\
\rowcolor{yellow}
5 &	6 &	61 &	55,56\\
\rowcolor{cyan}
8 &	0 &	64 &	57,58\\
\rowcolor{cyan}
0 &	8 &	64 &	57,58\\
8 &	1 &	65 &	59,60,61,62\\
1 &	8 &	65 &	59,60,61,62\\
7 &	4 &	65 &	59,60,61,62\\
4 &	7 &	65 &	59,60,61,62\\
8 &	2 &	68 &	63,64\\
2 &	8 &	68 &	63,64\\
\rowcolor{yellow}
6 &	6 &	72 &	65\\
\rowcolor{cyan}
8 &	3 &	73 &	66,67\\
\rowcolor{cyan}
3 &	8 &	73 &	66,67\\
7 &	5 &	74 &	68,69\\
5 &	7 &	74 &	68,69\\
\rowcolor{yellow}
8 &	4 &	80 &	70,71\\
\rowcolor{yellow}
4 &	8 &	80 &	70,71\\
\rowcolor{cyan}
9 &	0 &	81 &	72,73\\
\rowcolor{cyan}
0 &	9 &	81 &	72,73\\
\rowcolor{magenta}
9 &	1 &	82 &	74,75\\
\rowcolor{magenta}
1 &	9 &	82 &	74,75\\
9 &	2 &	85 &	76,77,78,79\\
2 &	9 &	85 &	76,77,78,79\\
7 &	6 &	85 &	76,77,78,79\\
6 &	7 &	85 &	76,77,78,79\\
\rowcolor{cyan}
8 &	5 &	89 &	80,81\\
\rowcolor{cyan}
5 &	8 &	89 &	80,81\\
\rowcolor{yellow}
9 &	3 &	90 &	82,83\\
\rowcolor{yellow}
3 &	9 &	90 &	82,83\\
\rowcolor{magenta}
9 &	4 &	97 &	84,85\\
\rowcolor{magenta}
4 &	9 &	97 &	84,85\\
\rowcolor{orange}
7 &	7 &	98 &	86\\
\rowcolor{gray}
10 & 0 & 100 &	87,88,89,90\\
\rowcolor{gray}
0 &	10 & 100 &	87,88,89,90\\
\rowcolor{red}
8 &	6 &	100	& 87,88,89,90\\
\rowcolor{red}
6 &	8 &	100	& 87,88,89,90\\
\rowcolor{blue}
10 & 1 & 101 &	91,92\\
\rowcolor{blue}
1 &	10 & 101 &	91,92\\
\hline
\end{tabular}
\quad
\begin{tabular}{|c  c  c  c |}
\hline
Dirichlet & & & \\
\hline
$m$ & $n$ & $m^2+n^2$ & $k$\\
\hline
8 &	2 &	68 &	46,47\\
2 &	8 &	68 &	46,47\\
6 &	6 &	72 &	48\\
\rowcolor{cyan}
8 &	3 &	73 &	49,50\\
\rowcolor{cyan}
3 &	8 &	73 &	49,50\\
\rowcolor{yellow}
7 &	5 &	74 &	51,52\\
\rowcolor{yellow}
5 &	7 &	74 &	51,52\\
8 &	4 &	80 &	53,54\\
4 &	8 &	80 &	53,54\\
\rowcolor{cyan}
9 &	1 &	82 &	55,56\\
\rowcolor{cyan}
1 &	9 &	82 &	55,56\\
\rowcolor{yellow}
7 &	6 &	85 &	57,58\\
\rowcolor{yellow}
6 &	7 &	85 &	57,58\\
9 &	2 &	85 &	59,60\\
2 &	9 &	85 &	59,60\\
8 &	5 &	89 &	61,62\\
5 &	8 &	89 &	61,62\\
9 &	3 &	90 &	63,64\\
3 &	9 &	90 &	63,64\\
\rowcolor{cyan}
9 &	4 &	97 &	65,66\\
\rowcolor{cyan}
4 &	9 &	97 &	65,66\\
\rowcolor{yellow}
7 &	7 &	98 &	67\\
8 &	6 &	100 &	68,69\\
6 &	8 &	100 &	68,69\\
\rowcolor{cyan}
10 & 1 &	101 &	70,71\\
\rowcolor{cyan}
1 &	10 &	101	& 70,71\\
\rowcolor{magenta}
10 & 2 &	104	& 72,73\\
\rowcolor{magenta}
2 &	10 &	104	& 72,73\\
\rowcolor{yellow}
9 &	5 &	106	& 74,75\\
\rowcolor{yellow}
5 &	9 &	106	& 74,75\\
10 & 3 &	109 &	76,77\\
3 &	10 &	109	& 76,77\\
8 &	7 &	113 &	78,79\\
7 &	8 &	113	& 78,79\\
\rowcolor{yellow}
10 & 4 &	116 &	80,81\\
\rowcolor{yellow}
4 &	10 &	116	& 80,81\\
\rowcolor{cyan}
9 &	6 &	117 &	82,83\\
\rowcolor{cyan}
6 &	9 &	117 &	82,83\\
\rowcolor{gray}
11 & 1 & 122 &	84,85\\
\rowcolor{gray}
1 &	11 & 122 &	84,85\\
\rowcolor{magenta}
10 & 5 & 125 &	86,87,88,89\\
\rowcolor{magenta}
5 &	10 & 125 &	86,87,88,89\\
\rowcolor{blue}
11 & 2 & 125 &	86,87,88,89\\
\rowcolor{blue}
2 &	11 & 125 &	86,87,88,89\\
\rowcolor{orange}
8 &	8 &	128 &	90\\
\rowcolor{red}
9 &	7 &	130 &	91,92,93,94\\
\rowcolor{red}
7 &	9 &	130 &	91,92,93,94\\
\hline
\end{tabular}

\begin{tabular}{|c  c  c  c |} 
\hline
Neumann & & & \\
\hline
$m$ & $n$ & $m^2+n^2$ & $k$\\
\hline
10 & 2 & 104 & 93,94\\
2 &	10 & 104 & 93,94\\
9 & 5 & 106 & 95,96\\
5 &	9 &	106 & 95,96\\
10 & 3 & 109 & 97,98\\
3 &	10 & 109 & 97,98\\	
8 &	7 &	113 & 99,100\\
7 &	8 &	113 & 99,100\\
\rowcolor{yellow}	
10 & 4 & 116 & 101,102\\
\rowcolor{yellow}
4 &	10 & 116 & 101,102\\	
\rowcolor{cyan}
9 &	6 &	117 & 103,104\\
\rowcolor{cyan}
6 &	9 &	117 & 103,104\\	
\rowcolor{magenta}
11 & 0 & 121 & 105,106\\
\rowcolor{magenta}
0 &	11 & 121 & 105,106\\	
\rowcolor{orange}
11 & 1 & 122 & 107,108\\
\rowcolor{orange}
1 &	11 & 122 & 107,108\\	
11 & 2 & 125 & 109 - 112\\
2 &	11 & 125 & 109 - 112\\	
10 & 5 & 125 & 109 - 112\\	
5 &	10 & 125 & 109 - 112\\	
\rowcolor{yellow}
8 &	8 &	128 & 113\\
\rowcolor{cyan}
11 & 3 & 130 & 114 - 117\\
\rowcolor{cyan}
3 &	11 & 130 & 114 - 117\\	
9 &	7 &	130 & 114 - 117\\	
7 &	9 &	130 & 114 - 117\\	
\rowcolor{yellow}
10 & 6 & 136 & 118,119\\
\rowcolor{yellow}
6 &	10 & 136 & 118,119\\
\rowcolor{cyan}
11 & 4 & 137 & 120,121\\
\rowcolor{cyan}
4 &	11 & 137 & 120,121\\	
12 & 0 & 144 & 122,123\\
0 &	12 & 144 & 122,123\\
12 & 1 & 145 & 124 - 127\\
9 &	8 &	145 & 124 - 127\\	
8 &	9 &	145 & 124 - 127\\	
1 &	12 & 145 & 124 - 127\\	
11 & 5 & 146 & 128,129\\
5 &	11 & 146 & 128,129\\	
\hline
\end{tabular}
\quad
\begin{tabular}{|c  c  c  c |} 
\hline
Dirichlet & & & \\
\hline
$m$ & $n$ & $m^2+n^2$ & $k$\\
\hline
11 & 3 & 130 & 93 - 96\\
3 &	11 & 130 & 93 - 96\\
10 & 6 & 136 & 95,96\\
6 &	10 & 136 & 95,96\\	
11 & 4 & 137 & 97,98\\
4 &	11 & 137 & 97,98\\	
9 &	8 &	145 & 99 - 102\\
8 &	9 &	145 & 99 - 102\\	
\rowcolor{magenta}
12 & 1 & 145 & 99 - 102\\	
\rowcolor{magenta}
1 &	12 & 145 & 99 - 102\\	
\rowcolor{yellow}
11 & 5 & 146 & 103,104\\
\rowcolor{yellow}
5 &	11 & 146 & 103,104\\	
\rowcolor{orange}
12 & 2 & 148 & 105,106\\
\rowcolor{orange}
2 &	12 & 148 & 105,106\\
\rowcolor{cyan}
10 & 7 & 149 & 107,108\\
\rowcolor{cyan}
7 &	10 & 149 & 107,108\\	
12 & 3 & 153 & 109,110\\
3 &	12 & 153 & 109,110\\	
11 & 6 & 157 & 111,112\\
6 &	11 & 157 & 111,112\\	
\rowcolor{cyan}
12 & 4 & 160 & 113,114\\
\rowcolor{cyan}
4 &	12 & 160 & 113,114\\	
\rowcolor{yellow}
9 &	9 &	162 & 115\\
10 & 8 & 164 & 116,117\\
8 & 10 & 164 & 116,117\\	
\rowcolor{cyan}
12 & 5 & 169 & 118,119\\
\rowcolor{cyan}
5 &	12 & 169 & 118,119\\	
\rowcolor{yellow}
11 & 7 & 170 & 120 - 123\\
\rowcolor{yellow}
7 &	11 & 170 & 120 - 123\\
13 & 1 & 170 & 120 - 123\\	
1 &	13 & 170 & 120 - 123\\	
13 & 2 & 173 & 124,125\\
2 &	13 & 173 & 124,125\\	
13 & 3 & 178 & 126,127\\	
3 & 13 & 178 & 126,127\\	
12 & 6 & 180 & 128,129\\	
6 &	12 & 180 & 128,129\\	
\hline
\end{tabular}
\newpage

\section{On the local structure of the nodal set.}\label{appB}
 In this appendix, we prove some well-known results for the nodal set of an eigenfunction of the Neumann problem and extend
them to the Robin problem.
Although used in various contributions, for example \cite{HHOHOO}, no detailed proofs seem to be published for
the Neumann problem. For the Dirichlet problem, see \cite{HOMN} and \cite{HHOT} where the case with corners or cracks is also considered.
In addition, we require these results under weaker regularity assumptions on the boundary.\\

\subsection{Main statement.}
\begin{thm}\label{thm:nodinfo} Let $\Omega$ be an open set in $\mathbb R^2$ with $C^{2,+}$ boundary.
Let $h \in [0,+\infty)$ and let $u$  be a real-valued eigenfunction of the Laplacian with $h$-Robin  boundary conditions.
  Then $u\in C^2(\overline\Omega)$. Furthermore, $u$ has the following
  properties:
  \begin{enumerate}
    \item\label{item:taylor} If $u$ and $\nabla u$ vanish at a point $x_0\in \Omega$ then there exists $\ell >1$, $\epsilon >0$ and  a real-valued, non-zero, harmonic, homogeneous
          polynomial of degree $\ell$ such that:
          \begin{equation}\label{eqn:harmpoly}
           u(x)=p_\ell (x-x_0)+\mathcal O(|x-x_0|^{\ell+\epsilon}).
          \end{equation}
  \item If $u$ vanishes at $x_0\in\partial\Omega$, then  \eqref{eqn:harmpoly} holds for some $\ell >0$ and
          \begin{equation}\label{eqn:cosexp}
            u(x)=ar^\ell \cos \ell \omega+\mathcal O(r^{\ell +\epsilon})
          \end{equation}
          for some non-zero $a\in\mathbb R$, where $(r,\omega)$ are polar coordinates
          of $x$ around $x_0$. The angle $\omega$ is chosen so that the tangent
          to the boundary at $x_0$ is given by the equation $\sin\omega=0$.
    \item \label{item:nod}The nodal set $N(u)$ is the union of finitely many, $C^2$-immersed circles in $\Omega$, and $C^1$-immersed lines which connect points of
          $\partial\Omega$. Each of these immersions is called a
          \textit{nodal line}. Note that self-intersections are allowed.
          The connected components of $\Omega\setminus N(u)$ are called
          \textit{nodal domains}.
    \item \label{item:order} If $u$ has a zero of order $\ell$ at a point $x_0\in
          \Omega$ then exactly $\ell $ segments of nodal lines pass through $ x_0$.
          The tangents to the nodal lines at $x_0$ dissect the full circle of radius $B(x_0,\alpha)$ (for $\alpha >0$ small enough) into $2\ell $ equal angles.
          \item  If $u$ has a zero of order $\ell $ at a point $x_0 \in
          \partial\Omega$ then exactly $\ell$ segments of nodal lines meet the
          boundary at $x_0$. The tangents to the nodal lines at $x_0$ are given
          by the equation $\cos \ell \omega=0$, where $\omega$ is chosen as
          in~\eqref{eqn:cosexp}.
  \end{enumerate}
\end{thm}
\subsection{Proof of the theorem.}
The $C^2$-regularity of $u$ up to the boundary is a consequence of standard Schauder estimates (see \cite{GT}).
The proof now  is in four steps.\\
\subsubsection{Reduction to the Neumann case.}
 The first step is to reduce the problem from the Robin case to the Neumann case.
 This is done through a change of functions.
 Setting $u= \exp \phi_h \, v$, we can choose $\phi_h$ such that $v$ satisfies the Neumann condition.
 Indeed, this $\phi_h$ should be in $C^2(\overline{\Omega})$ and satisfy $\partial_\nu \phi_h= -h $ on the boundary of $\Omega$ (take $ h\, {\rm dist} (x,\partial \Omega)$ near $\partial \Omega$ and then use a cut-off function). We obtain a Neumann problem where the Laplacian is replaced by $\exp - \phi_h \circ(-\Delta) \circ \exp \phi_h$, that is the Laplacian with an additional one-dimensional term with $C^1(\overline{\Omega})$ coefficients.\\
   From this point onwards, we consider the Neumann case.

 \subsubsection{Double manifold.}
 The second step is to use the double manifold as suggested in Donnelly-Feffermann,  \cite{DF3, DF2, DF1}.
 As we only wish to prove a local result, by a diffeomorphism we can
 reduce to the case when the boundary is given by $x_1=0$. In these new coordinates, the operator reads
 $$
 H:= \sum_{ij} g_{ij}(x_1,x_2) \partial_{x_i}\partial_{x_j} + \sum a_{i}(x_1,x_2) \partial_{x_i} + c(x)\,.
 $$
 In addition, this diffeomorphism can be chosen as a conformal map (see \cite{DF3}), so more precisely,  we have
 $$
 H:= - \rho(x) \Delta + \sum a_{i}(x) \partial_{x_i} + c(x)\,.
 $$
 Note that in the Neumann case, there are no linear terms. This would make the proof easier and would permit weaker assumptions.\\
If $u$ denotes the eigenfunction defined locally in $x_1 >0$, we define $\tilde u$ by
$$
\tilde u (x_1,x_2) = \left\{  \begin{array}{c}u(x_1,x_2) \mbox{ for }  x_1 >0\,,\\
   u(-x_1,x_2) \mbox{ for } x_1 < 0\,.
\end{array}\right.
$$
We can then define the extension of the operator as $\widetilde H$
$$
 \widetilde H:= - \tilde \rho(x) \Delta + \sum \tilde a_{i} \partial_{x_i} + \tilde c(x)\,.
$$
where $\tilde \rho$, $\tilde a_2$ and $\tilde c$ are  the extensions of $\rho$, $a_2$ and $c$ by  reflection and $\tilde a_1$ is defined by  odd reflection.\\
So $\tilde \rho$, $\tilde a_2$ and $\tilde c$ are Lipschitz and $\tilde a_1$ is only bounded.\\
With this definition, we verify that $\tilde u$ is an even function (with respect to $x_1$)  that satisfies the Neumann condition, and  a solution of
$$
\widetilde H \tilde u = \lambda \tilde u\,.
$$
We know that $\tilde u  \in C^2(\overline{\mathbb R_-}\times \mathbb R) \cap C^2(\overline{\mathbb R_+}\times \mathbb R)$. Also, $\tilde u$ is clearly in $C^{1,1} (\mathbb R^2)$.\\
We note that from $\tilde u(x_1,x_2)= \tilde u (-x_1,x_2)$, we get $\partial^2_{x_1,x_2}\tilde u (0,x_2) =0\,$.
The other second derivatives match on $x_1=0$. Hence $\tilde u$ is actually in $C^2(\R^2)$.

\subsubsection{Nodal structure for solutions of  a  second-order elliptic operator with coefficients with less regularity.}

The third step is to determine whether the local nodal structure that holds for the Laplacian still holds for this second-order elliptic operator  which has coefficients with less regularity.
This problem is analysed by Hardt-Simon in \cite{HaSi} (at least in a weaker sense) and more precisely in \cite{HQ} (see Theorem 1.5 and Theorem 3.1).
 The following theorem is Theorem 3.1 of \cite{HQ} applied to $\tilde u$ and $$L:= \widetilde H -\lambda$$
in the neighbourhood of a point in the zero set on the boundary, which is assumed to be $(0,0)$. From this point onwards, we omit the  tildes.
\begin{thm}\label{thm:QH}
Suppose that $Lu=0$ and that $u$ is not flat at $(0,0)$,  that is, $u$ has finite  vanishing order at $(0,0)$.
Then there exists a homogeneous harmonic polynomial $P$ of degree $d$ and, for any $p>1$, an $\epsilon >0$ such that
$\psi:= u-P $ satisfies:
$$
\psi (x) =\mathcal O (|x|^{d+\epsilon})\,,
$$
and
$$
r^2 \left( \int_{B(0,r)} |D^2\psi (x)|^p\, dx \right)^\frac 1p + r \left (\int_{B(0,r)} |D\psi (x)|^p\, dx \right)^\frac 1p \leq C r^{d+\epsilon + \frac 2p}\,,\, r>0\,.
$$
\end{thm}
 \begin{rem}
 Note that to apply the theorem, we need to know that $u$ is not flat. According to \cite{HQ} (p. 985, lines  7--9), this is the case under our assumptions and the reference is \cite{GL}.
 \end{rem}
This theorem gives a good indication of the nodal structure:  it
should be close to the zero set of the harmonic polynomial $P$ whose structure is well known.

\subsubsection{Cheng-Kuo's argument.}
Hence the last step is to verify if Cheng's argument  \cite{Chg} applies (a former reference is  \cite{Bers}).
We can apply the following lemma attributed by Cheng \cite{Chg} to Kuo \cite{Ku}.
 \begin{lem}
 Suppose that $u$ and $p$ are smooth functions in $\mathbb R^2$ such that, with $\psi=u-p$, we have for some $d\geq 1$ and $\epsilon>0$,
 \begin{itemize}
 \item[(i)]
 $
\psi (x) =\mathcal O (|x|^{d+\epsilon})\,,
$
\item[(ii)]
$
\nabla \psi(x) = \mathcal O (|x|^{d-1+\epsilon})\,,
$
\item[(iii)]
$p$ vanishes at order $d$ at $0$,
\item[(iv)]
$|\nabla p(x) | \geq \frac{1}{C} |x|^{d-1}$.
\end{itemize}
Then there exists a local $C^1$ diffeomorphism $\Theta$ fixing the origin such that
$$
u (x)= p (\Theta(x))\,.
$$
\end{lem}
In \cite{Chg}, Cheng applies the lemma to $C^{\infty}$ functions, but the regularity of $u$ and $p$
is not discussed there. The proof clearly holds for $C^2$ functions and this assumption is satisfied in our case.\\

To apply this lemma to the present situation, we observe that  a homogeneous harmonic polynomial of degree $d$ in dimension $2$ satisfies (iii) and (iv) above. We note that (i) holds by Theorem~\ref{thm:QH}.

  It remains to verify that (ii) holds.
  We compare this condition with the property established in the previous theorem.
  By Theorem~\ref{thm:QH}, we get a control of $\nabla \psi$ in $W^{1,p}$ in any ball $B(0,r)$ hence by Sobolev's embedding
  theorem we have, as soon as $p>2$,  the control of $\nabla \psi$ in $L^\infty (B(0,r))$ (see, for example, Part II Case C' of Theorem 5.4 in \cite{Ad}).
  It remains to control the constants appearing in the continuity of this injection.
  To do this, for $r>0$, we introduce a cut-off $\chi (x/r)$ where $\chi=1$ on $B(0,1)$ and ${\rm supp} \chi \subset B (0,2)$, and apply the standard Sobolev embedding
  theorem to $\chi (x/r) \partial_{x_i} \psi$ and use the two estimates from Theorem~\ref{thm:QH}.
  We get
  $$
  \sup_{x\in B(0,r)} |\nabla \psi (x)| \leq C_p\, r^{-2 + d +\epsilon + \frac 2 p}\,,\, \mbox{ for } p>2\,.
  $$
  For $p>2$ sufficiently close to $2$ (for example $-1 +\frac 2 p =\frac \epsilon 2$), we get
  $$
  \sup_{x\in B(0,r)} |\nabla \psi (x)| \leq C_p\, r^{-1 + d + \frac \epsilon 2} \,,\, \mbox{ for } p>2\,.
  $$
  This is  sufficient to apply
  the lemma.
  \begin{rem}
  There is some controversy  regarding Cheng's paper \cite{Chg} when applied to a dimension larger than $2$.
  The reason is that a harmonic homogeneous polynomial does not always satisfy Item (iv) when the dimension is larger than $2$ (see Appendix E in \cite{BM}).
\end{rem}

  \subsection{Remarks.}
  We note that all the proofs are local and the results can be obtained locally if we have the corresponding local regularity property.\\
  The proofs  also work in the Dirichlet case (with a  different condition on $\omega$).
  Instead of the  reflection argument,  in order to construct $\tilde{u}$, we can introduce an extension via odd reflection:
  $$
\tilde u (x_1,x_2) = \left\{  \begin{array}{c}u(x_1,x_2) \mbox{ for }  x_1 >0\,,\\
   - u(-x_1,x_2) \mbox{ for } x_1 < 0\,.
\end{array}\right.
$$
Analogously to the above, if $u$ is an eigenfunction in $C^2(\overline{\Omega})$ satisfying the Dirichlet condition, one can verify that $\tilde u$ is in $C^2(\R^2)$.

\begin{thm}\label{thm:nodinfoD}
Let $\Omega$ be an open set in $\mathbb R^2$ with $C^{2,+}$ boundary and let $u$  be a real-valued eigenfunction of
 the Laplacian with Dirichlet boundary conditions.
Then $u\in C^2(\overline\Omega)$. Furthermore, $u$ has the following properties:
  \begin{enumerate}
    \item\label{item:taylor2} If $u$ and $\nabla u$ vanish at a point $x_0\in \overline{\Omega}$ then there exists $\ell >1$, $\epsilon >0$ and a real-valued, non-zero, harmonic, homogeneous
          polynomial of degree $\ell$ such that:
          \begin{equation}\label{eqn:harmpoly2}
           u(x)=p_\ell (x-x_0)+\mathcal O(|x-x_0|^{\ell+\epsilon}).
          \end{equation}
  \item If moreover $x_0\in\partial\Omega$, then
          \begin{equation}\label{eqn:cosexp2}
            u(x)=ar^\ell \sin \ell \omega+\mathcal O(r^{\ell +\epsilon})
          \end{equation}
          for some non-zero $a\in\mathbb R$, where $(r,\omega)$ are polar coordinates
          of $x$ around $x_0$. The angle $\omega$ is chosen so that the tangent
          to the boundary at $x_0$ is given by the equation $\omega=0$.
    \item \label{item:nod2}The nodal set $N(u)$ is the union of finitely many, $C^2$-immersed circles in $\Omega$,
    and $C^1$-immersed lines which connect points of $\partial\Omega$.
    \item \label{item:order2} If $u$ has a zero of order $\ell$ at a point $x_0\in
          \Omega$, then exactly $\ell $ segments of nodal lines pass through $ x_0$.
          The tangents to the nodal lines at $ x_0$ dissect the full circle of radius $B(x_0,\alpha)$ (for $\alpha >0$ small enough) into $2\ell $ equal angles.
          \item  If $u$ has a zero of order $\ell $ at a point $x_0 \in
          \partial\Omega$ then exactly $\ell -1$ segments of nodal lines meet the
          boundary at $x_0$. The tangents to the nodal lines at $x_0$ are given
          by the equation $\sin  \ell \omega=0$, $\omega \neq 0,\pi$.
  \end{enumerate}
\end{thm}

We can, for example, refer to \cite{HHOT} for the Dirichlet case which  gives the results (except $C^2$ regularity) under the  weaker assumption  that the boundary is piecewise $C^{1,+}$.

\end{document}